\documentclass[10pt]{amsart}

\usepackage{tikz}
\usetikzlibrary{decorations.pathmorphing}
\usepackage{subcaption}
\usepackage{float}
\usetikzlibrary{arrows.meta,positioning,calc}
\usepackage{amssymb}
\usepackage{enumitem}
\usepackage{comment}

\newtheorem*{theorem*}{Theorem}
\newtheorem{theorem}{Theorem}[section]
\newtheorem{lemma}[theorem]{Lemma}
\newtheorem{proposition}[theorem]{Proposition}
\newtheorem{corollary}[theorem]{Corollary}

\newtheorem{problem}[theorem]{Problem}

\tikzset{
	side marker/.style={
		double,
		double distance=0.7pt,
		line width=0.35pt
	}
}
\tikzset{
  wavy/.style={
    decorate,
    decoration={snake,amplitude=.75mm,segment length=4mm}
  }
}
\newcommand{\mline}{%
=\!=\!=
}
\newcommand{\mto}{%
=\!=\!\vartriangleright
}
\newcommand{\mfrom}{%
\vartriangleleft\!=\!=
}
\newcommand{\mtom}{%
\vartriangleleft\!=\!=\!\vartriangleright
}
\newcommand{\clique}{{\sc Clique}}
\newcommand{\independence}{{\sc Independence}}
\newcommand{\path}{{\sc Path}}
\newcommand{\betweenness}{{\sc Betweenness}}
\newcommand{\claw}{{\sc Claw}}

\tikzstyle{Y} = [circle,draw=black,minimum width=.05em, inner sep=.04em]
\tikzstyle{W} = [rectangle,draw=black,fill=white,minimum width=.15em, inner sep=.3em]
\tikzstyle{A} = [circle,fill=black,minimum width=.1em, inner sep=.1em]
\tikzset{AA/.style={circle,fill=black,inner sep=.5pt,minimum size=2.5pt}}
\tikzstyle{Ar} = [rectangle,fill,minimum width=.15em, inner sep=.15em]  
\tikzstyle{Wr} = [rectangle,fill=red,minimum width=.15em, inner sep=.15em]      

\title{Hilbert fragments}

\author{Csaba Bir\'o}
\address{Department of Mathematics, University of Louisville, Louisville, KY 40220}

\author{Andr\'e E. K\'ezdy}
\address{Department of Mathematics, University of Louisville, Louisville, KY 40220}

\author{Jen\H{o} Lehel}
\address{Department of Mathematics, University of Louisville, Louisville, KY 40220}
\address{Alfr\'ed R\'enyi Institute of Mathematics, P.O.~Box 127, H-1364 Budapest, Hungary}

\begin{document}
\date{\today}
\begin{abstract}
In earlier work, we proved that the length polyhedron of an interval order has a unique Hilbert basis and that every element of this basis is a binary vector. The supports of these binary vectors form a set system whose intersection graph is a cocomparability graph. This leads naturally to the question of which graphs can arise as intersection graphs of such Hilbert-basis supports. 
We introduce Hilbert fragments, defined as induced subgraphs of these intersection graphs. Thus Hilbert fragments form a hereditary class of cocomparability graphs recording the intersection patterns among the supports of Hilbert-basis elements of interval-order length polyhedra. We characterize Hilbert fragments by an axiomatically defined vertex-labeling, called a track-labeling, and give applications illustrating how this characterization can reveal an underlying one-dimensional structure even when the Hilbert fragments themselves are not interval graphs.
\end{abstract}
\maketitle

\section{Introduction}
\label{sec:intro}
\subsection*{Motivation, history, and outline}

Interval orders and interval graphs arose from different directions, but both express the same fundamental idea: encoding one-dimensional structure by a family of intervals. The origins of interval orders can be traced to Wiener \cite{wiener1914}, while 
semiorders were introduced by Luce~\cite{Luce1956} in his study of preferences with a threshold of discrimination and subsequently characterized by Scott and Suppes \cite{ScottSuppes} through unit-interval representations. On the graph-theoretic side, Hajós \cite{Hajos1957} posed the recognition problem for intersection graphs of intervals in 1957, and Benzer \cite{Benzer1959} independently encountered the same structure in 1959 in studying whether observed overlaps among genetic fragments were compatible with a linear organization. Fishburn's work subsequently established interval orders as a basic class of partially ordered sets, characterized by the exclusion of  $\mathbf{2+2}$ as an induced subposet. Thus interval orders and interval graphs became two closely related combinatorial models of one-dimensional structure. We refer to Fishburn's monograph~\cite{FishburnBook} and Trotter's survey~\cite{TrotterSurvey} for background on interval orders, and to Golumbic~\cite{Golumbic2004} for the graph-theoretic setting, including interval graphs and related classes of perfect graphs.

A classical refinement of interval representation is the \emph{interval-count problem}, originally suggested by Graham (see \cite{FMOS} for historical background), which asks for the minimum number of distinct interval lengths needed to represent a given interval order. The problem was studied systematically by Fishburn~\cite{FishburnBook}; related restrictions on the number and range of interval lengths were investigated by Fishburn and Graham~\cite{FishburnGraham1985}. The one-length case gives the classical semiorders, while already the two-length case exhibits substantially richer behavior. Progress in restricted settings includes our characterization of two-count interval orders of height three, arising from our work on interval representations of permutations~\cite{BiKeLe}.

A complementary viewpoint is to study all feasible length assignments at once. Fishburn developed systems of length inequalities, Isaak~\cite{Isaak2009} applied directed-graph methods to constrained interval representations, and Doignon and Pauwels~\cite{DoignonPauwels} investigated the polyhedral structure of the space of interval representations. This line of thought leads naturally from the interval-count problem to the geometry and arithmetic of the associated length polyhedron.

In our previous paper~\cite{LengthPolyhedron}, we studied this length polyhedron, a translated pointed rational polyhedral cone, and determined the unique minimal Hilbert basis of its integral points. Remarkably, all elements of this basis are binary vectors. Their supports can therefore be viewed as distinguished subsets of the intervals in the canonical representation. Moreover, these generators have a direct geometric meaning: relative to the canonical length vector, a generator represents an elementary admissible increment, increasing by one the lengths of precisely the intervals in its support.

This brings a graph into the picture. The \emph{Hilbert graph} of an interval order is the intersection graph---equivalently, the line graph---of the hypergraph formed by these supports. Thus two vertices of the Hilbert graph are adjacent precisely when the corresponding elementary increments affect a common interval. This graph already plays an algorithmic role in computing the Hilbert basis: deciding whether a candidate set decomposes into smaller Hilbert sets can be reduced to a weighted independent-set problem in an induced subgraph of the Hilbert graph~\cite{LengthPolyhedron}.

The present paper reverses this viewpoint. Rather than starting with an interval order and constructing its Hilbert graph, we ask:
\begin{center}
\emph{Which graphs can occur as intersection patterns of Hilbert sets of an interval order?}
\end{center}
This question deliberately forgets most of the original interval representation and retains only the pairwise intersections among selected Hilbert sets. In applications one may likewise encounter only part of the underlying one-dimensional structure. This motivates our central object: a \emph{Hilbert fragment}, an induced subgraph of a Hilbert graph. The problem is then to characterize these graphs intrinsically, without reference to the interval order from which they arise.

Our characterization reveals that surprisingly strong traces of the underlying one-dimensional structure survive this loss of information. We prove that Hilbert fragments form a new hereditary subclass of cocomparability graphs and hence a class of perfect graphs~\cite{Golumbic2004}. Although arising from interval orders through a different construction, Hilbert fragments have a natural place alongside the classical model of one-dimensional structure provided by interval graphs, which are precisely the chordal cocomparability graphs~\cite{GilmoreHoffman1964}.

For an intrinsic description of Hilbert fragments, we introduce and axiomatize a class of vertex-labelings, called \emph{track-labelings}. Our main characterization is the following.

\begin{theorem*}
A graph is a Hilbert fragment if and only if every connected component admits a track-labeling.
\end{theorem*}

The paper is organized as follows.
In the rest of Section~\ref{sec:intro} we introduce the most important definitions, 
some well-known facts about interval orders, and a few natural lemmas about Hilbert fragments. 
These lemmas are included primarily for reference and are used throughout the paper.

In Section~\ref{sec:properties} we introduce the track-axioms that define track-labelings and prove that connected Hilbert fragments admit a track-labeling. In Section~\ref{sec:digraph} we introduce an auxiliary digraph and identify a problematic configuration in it, called a \emph{bi-fork}. We prove that the auxiliary digraph arising from a gap-labeling of a Hilbert fragment is bi-fork-free. Section~\ref{sec:bi-fork-free} presents a construction of an interval order representing a connected bi-fork-free track-labeled graph, thus proving that such graphs are Hilbert fragments. However, the auxiliary digraph associated with an arbitrary track-labeling need not be bi-fork-free. In Section~\ref{sec:eliminate} we show that the track-labeling can be adjusted to eliminate all bi-forks, thereby completing the proof of the main theorem.

In Section~\ref{sec:family}, 
we prove that Hilbert fragments are cocomparability graphs. 
In Section~\ref{sec:applications}, we discuss two exploratory applications of Hilbert fragments.
In the concluding Section~\ref{sec:conclusion}, we pose two problems for further research: characterize Hilbert fragments in terms of forbidden induced subgraphs, and determine the complexity of recognizing Hilbert fragments. 
In Appendix~\ref{app:obstructions}, we display all minimal forbidden induced subgraphs of order at most 6. In Appendices~\ref{app:proofs} and~\ref{app:proofC}, we provide detailed technical proofs of auxiliary results.

In this paper all graphs, undirected or directed, are finite and simple. We use standard graph-theoretic terminology and notation as in West \cite{WestBook},
partial-order terminology as defined in Trotter \cite{TrotterBook}, and polyhedral terminology as in Schrijver \cite{SchrijverBook}.

\subsection*{Interval representations, gap sets, and fundamental extenders} \label{intro:interval}

Let $P=(X,\prec)$ be a nonempty finite poset and let $n = |X|$.  The poset $P$ is an
\emph{interval order} if there is a multiset
$$
\mathcal F=\{I_x=[\ell_x,r_x]\subset \mathbb R:x\in X\}
$$
of compact (closed and bounded) intervals such that, for all $x,y\in X$,
$$
x\prec y \quad\Longleftrightarrow\quad r_x<\ell_y.
$$
Such a family $\mathcal F$ is an \emph{interval representation} of $P$.
An interval representation is rational (integral) if all of its intervals have rational (integral) endpoints.
Fix an ordering $X=\{x_1,\ldots,x_n\}$.  All vectors indexed by $X$ are written
with respect to this ordering.
The length of $I_x$ is $\rho_x=r_x-\ell_x$, and
$\rho=(\rho_{x_1},\ldots,\rho_{x_n})\in\mathbb R^n$ is the \emph{length vector} of
$\mathcal F$.  

Let
$$
c_1<c_2<\cdots<c_t
$$
be the distinct endpoints appearing in $\mathcal F$; thus
$$
\{c_1,\ldots,c_t\}
=
\{\ell_x:x\in X\}\cup\{r_x:x\in X\}.
$$
A \emph{bounded gap} of $\mathcal F$ is an interval $[c_{i-1},c_i]$ with
$2\le i\le t$.  We also use two \emph{extreme gaps}, denoted
$[-\infty,c_1]$ and $[c_t,+\infty]$.

For a bounded gap $I=[c_{i-1},c_i]$, define the \emph{gap set} 
$$
G[I]=\{I_x\in\mathcal F:\ell_x\le c_{i-1}\text{ and }c_i\le r_x\}.
$$
Equivalently, $G[I]$ is the set of intervals in $\mathcal F$ containing the
gap $I$.  We write $G[c_{i-1},c_i]$ for this gap set.  The gap sets of the extreme gaps are
$$
G[-\infty,c_1]=\varnothing=G[c_t,+\infty].
$$
A representation $\mathcal F$ is \emph{connected} if these two are its only empty gap sets.

For a bounded gap $I=[c_{i-1},c_i]$, define its \emph{left-borderline set} and
\emph{right-borderline set} as
$$
L[I]=\{I_x\in\mathcal F:r_x=c_{i-1}\},
\qquad
R[I]=\{I_x\in\mathcal F:\ell_x=c_i\}.
$$
Thus $L[I]$ consists of intervals ending at the left-hand endpoint of the gap,
and $R[I]$ consists of intervals beginning at the right-hand endpoint of the
gap.  The borderline sets for the extreme gaps are
$$
L[-\infty,c_1]=\varnothing,
\qquad
R[-\infty,c_1]=\{I_x\in\mathcal F:\ell_x=c_1\},
$$
and
$$
L[c_t,+\infty]=\{I_x\in\mathcal F:r_x=c_t\},
\qquad
R[c_t,+\infty]=\varnothing.
$$
	
A nonempty subfamily $S\subseteq\mathcal F$ is a \emph{fundamental extender}
with respect to $\mathcal F$ if there is a gap $I$ of $\mathcal F$ and a
subfamily $Z\subseteq\mathcal F$ such that
$$
S=G[I]\cup Z
$$
and either
$$
Z\subseteq L[I]
\qquad\text{or}\qquad
Z\subseteq R[I].
$$
In the first case $S$ is a \emph{left fundamental extender} at $I$, and in the second
case $S$ is a \emph{right fundamental extender} at $I$.  Such an expression is called
a \emph{gap--borderline decomposition} (or simply decomposition) of $S$.\footnote{In~\cite{LengthPolyhedron}, gap sets, borderline sets, and fundamental extenders were defined only for the canonical representation, as subsets of the ground set \(X\). In contrast, the present paper does not assume a fixed representation \emph{a priori}, and these objects are defined as subfamilies of the intervals in the given representation. For a fixed representation, the two conventions are equivalent via the correspondence \(x\leftrightarrow I_x\). However, the present convention is more convenient when interval representations are augmented or concatenated.}
Geometrically, a fundamental extender records a set of intervals whose
lengths increase when a single gap is enlarged, together with an arbitrary
choice of intervals attached to one side of that gap.

A fundamental extender may have more than one gap--borderline
decomposition.  For each fundamental extender
used in the paper, we choose and fix one decomposition.
All later references to the gap, side, or borderline set are made with respect to
this fixed choice.  

For a fixed decomposition $S=G[I]\cup Z$, every set
$G[I]\cup Z'$ with $Z'\subseteq Z$ is again a fundamental extender.  Thus, shrinking
the chosen borderline part of a fundamental extender produces smaller fundamental
extenders contained in it.  More generally, the phrase smaller fundamental extender
contained in $S$ means a proper subset $T\subsetneq S$ that is itself a fundamental
extender.

A fundamental extender $S$ is \emph{reducible with respect to $\mathcal F$}
if it admits a nontrivial partition into fundamental extenders with respect to
$\mathcal F$; that is, if there exist fundamental extenders
$S_1,\ldots,S_t$, with $t\ge 2$, such that each $S_i$ is properly contained in
$S$ and
$$
S=S_1\dot{\cup}S_2\dot{\cup}\cdots\dot{\cup}S_t.
$$
If no such partition exists, then $S$ is \emph{irreducible with respect to
$\mathcal F$}. When $\mathcal F$ is understood, we simply say that $S$ is
reducible or irreducible.

Given a collection $\mathcal S$ of irreducible fundamental extenders with
respect to a rational representation $\mathcal F$, the fixed decompositions of
these extenders determine a natural map, called the \emph{gap labeling} of
$\mathcal S$,
$$
\eta_{\mathcal F}:\mathcal S\to \mathbb Q\cup\{-\infty,+\infty\}.
$$
If $S=G[I]\cup Z$ is the fixed decomposition of $S$, define
$$
\eta_{\mathcal F}(S)=
\begin{cases}
	c_i, & \text{if } I=[c_{i-1},c_i] \text{ for some } i\in\{2,\ldots,t\},\\
	-\infty, & \text{if } I=[-\infty,c_1],\\
	+\infty, & \text{if } I=[c_t,+\infty].
\end{cases}
$$

\subsection*{Canonical representations and Hilbert sets} \label{intro:canonical}
Every interval order $P=(X,\prec)$  admits an essentially unique \emph{minimal-endpoint
representation}, where uniqueness is understood after disregarding the physical
locations of the endpoints \cite[Theorem 2.4]{Greenough}. 
The representation we now describe is 
called the minimal-endpoint representation, or the \emph{canonical representation} \cite{LengthPolyhedron}.

For an interval order $P=(X,\prec)$, let
$$
D(x)=\{y:y\prec x\},\qquad U(x)=\{y:x\prec y\}.
$$
The distinct down-sets and up-sets form inclusion chains of the same cardinality:
$$
\varnothing=D_0\subset\cdots\subset D_{m-1},\qquad
U_0\supset\cdots\supset U_{m-1}=\varnothing.
$$
The common cardinality $m=m(P)$, called the \emph{magnitude} of $P$, is the minimum number of distinct endpoints in any interval representation of $P$. The \emph{canonical representation} $\mathcal C(P)$ assigns to each $x\in X$ the interval
$$
I_x^*=[i,j]\quad\text{where}\quad D(x)=D_i,\quad U(x)=U_j,
$$
and these intervals are called \emph{canonical intervals}. 
Greenough's result  \cite[Theorem 2.4]{Greenough} on the minimal-endpoint representation  has the following corollary.

\begin{proposition}
\label{prop:canonical}
A family of intervals  with a set $S$ of endpoints is the canonical representation of some interval order if and only if every $p\in S$ occurs both as
a left endpoint and as a right endpoint of some interval.
\end{proposition}

Let $P=(X,\prec)$ be an interval order and let $\mathcal C(P)$ be the canonical representation on distinct endpoints $\{0,1,\ldots,m-1\}$, where $m=m(P) $ is the magnitude of $P$; equivalently, $m$ is the minimum number of distinct endpoints in an interval representation of $P$.

The \emph{canonical length vector} is the length vector of $\mathcal C(P)$, whose components are the lengths of the canonical intervals.
	
	
	Let $Q_P\subseteq \mathbb R^n$ be the \emph{length polyhedron of $P$}, the
	convex hull of all length vectors of integral interval representations of $P$.
	Recall several basic facts about $Q_P$ (for details, see \cite{LengthPolyhedron}). 
	The canonical length vector
	$\rho^*$ is the vertex of the affine cone $Q_P$, and we write
	$$
	Q_P=\rho^*+Q_P^*
	$$
	for the corresponding pointed rational cone.  
	
	Recall that a \emph{Hilbert basis} for a polyhedral cone is a finite set of vectors 
	such that each integral vector in the cone
	is a nonnegative integral linear combination of vectors from this set.  A Hilbert basis
	is {\em integral} if it consists entirely of integral vectors.  
	For a subfamily $S\subseteq\mathcal C(P)$, let
	$\chi_S\in\{0,1\}^n$ denote the characteristic vector whose $j$th coordinate is $1$ if
	$I_{x_j}^*\in S$, and is $0$ otherwise. 
	The characteristic vectors $\chi_S$, as $S$ ranges over all fundamental extenders with respect to $\mathcal C(P)$, form an integral
	Hilbert basis for $Q_P^*$ (Theorem 6.1 of \cite{LengthPolyhedron}).
	
	Geometrically, a fundamental extender with respect to $\mathcal C(P)$ records a basic
	way to extend a collection of canonical intervals simultaneously; its characteristic
	vector is the corresponding binary direction in the cone $Q_P^*$.  Since $Q_P^*$ is
	pointed and rational, it has a unique minimal integral Hilbert basis (see page 233 of \cite{SchrijverBook}).  A fundamental
	extender $S$ with respect to $\mathcal C(P)$ is called a \emph{Hilbert set} of $P$ if $\chi_S$ belongs to this unique minimal Hilbert
	basis for $Q_P^*$.
	Equivalently (see Theorem~6.2 of \cite{LengthPolyhedron}), Hilbert sets are the irreducible fundamental extenders with respect to $\mathcal C(P)$.  
	
The next lemma captures a useful fact about Hilbert sets.
\begin{lemma}\label{lem:Hilbertset}
	If $S$ is a Hilbert set with a gap--borderline decomposition that has an empty gap set, then the borderline set contains exactly one interval, that is $|S|=1$.
\end{lemma}
\begin{proof}
	Assume, to the contrary, that the borderline set contains more than one interval. 
	Since the gap set is empty, each singleton subfamily consisting of one interval
	from the borderline set is an irreducible fundamental extender, hence a Hilbert
	set.  
	This contradicts the irreducibility of S.
\end{proof}
	
\subsection*{Hilbert fragments and their gap labelings} \label{intro:fragments}

Let $P=(X,\prec)$ be an interval order.  The \emph{Hilbert extender graph} of $P$,
denoted $\mathcal H_P$, is the intersection graph of the Hilbert sets of $P$.
Thus the vertices of $\mathcal H_P$ are the Hilbert sets, and two vertices
corresponding to Hilbert sets $S$ and $T$ are adjacent if and only if
$$
S\cap T\ne\varnothing,
$$
where $S$ and $T$ are regarded as subfamilies of $\mathcal C(P)$; equivalently, they contain a common canonical interval $I_x^*$.

A graph $H=(V,E)$ is a \emph{Hilbert fragment} (or simply \emph{fragment}) if there is an interval order $P$ such
that $H$ is isomorphic to an induced subgraph of the Hilbert extender graph $\mathcal H_P$.  In this case,
we call $P$ an \emph{underlying interval order} for $H$.  More precisely, a
\emph{realization} of $H$ consists of an interval order $P$ together with an
injective map
$$
\varphi:V\to V(\mathcal H_P)
$$
such that, for all distinct $u,v\in V$,
$$
uv\in E
\quad\Longleftrightarrow\quad
\varphi(u)\cap\varphi(v)\ne\varnothing,
$$
where each $\varphi(v)$ is regarded as the corresponding Hilbert set.
The canonical representation of $P$,
together with $\varphi$, is called an \emph{underlying representation} of the fragment.
A Hilbert fragment may have many underlying interval orders.  

A \emph{gap-labeled Hilbert fragment} is a pair $(H,\eta)$, where $H$ is a
Hilbert fragment and $\eta$ is a vertex labeling obtained from some
realization of $H$ as follows.  Let $(P,\varphi)$ be a realization of $H$, and
let $m$ be the magnitude of $P$.  For each $v\in V$, choose a gap--borderline decomposition of
the Hilbert set $\varphi(v)$ in the canonical representation $\mathcal C(P)$:
$$
\varphi(v)=G[I_v]\cup Z_v.
$$
The associated gap-labeling is the map
$$
\eta:V\to \{1,\ldots,m-1\}\cup\{-\infty,+\infty\}
$$
defined by
$$
\eta(v)=
\begin{cases}
	i, & \text{if } I_v=[i-1,i] \text{ for some } i\in\{1,\ldots,m-1\},\\
	-\infty, & \text{if } I_v=[-\infty,0],\\
	+\infty, & \text{if } I_v=[m-1,+\infty].
\end{cases}
$$
We say that $(P,\varphi)$, together with the chosen decompositions of the
Hilbert sets $\varphi(v)$, \emph{realizes} the gap-labeled fragment
$(H,\eta)$.  Since underlying interval orders and gap--borderline
decompositions need not be unique, the same Hilbert fragment may admit many
gap-labelings.

For a realized gap-labeled Hilbert fragment, we write
$$
S_v=\varphi(v),\qquad G_v=G[I_v],
$$
so that
$$
S_v=G_v\cup Z_v
$$
is the chosen gap--borderline decomposition of the Hilbert set corresponding
to $v$.

\subsection*{Augmentation and concatenation} \label{intro:augmentationAndConcatenation}

To construct realizations of Hilbert fragments, we need a way to pass from a rational interval representation $\mathcal F$ to a canonical representation without losing the irreducible fundamental extenders determined by $\mathcal F$.  This is achieved via \emph{augmentation}: from $\mathcal F$ we build an augmented canonical representation $\mathcal F^+$ of a larger interval order $P^+$ containing $P$ as a subposet.  Lemma~\ref{lem:augmentation} shows that every irreducible fundamental extender with respect to $\mathcal F$ becomes a Hilbert set of $P^+$.  We now describe this construction.

Let
$\mathcal F=\{[\ell_x,r_x]:x\in X\}$ be a rational interval representation of
an interval order $P=(X,\prec)$, and let
$$
c_1<c_2<\cdots<c_t
$$
be the distinct endpoints of $\mathcal F$.  Choose a positive integer $N$ such
that
$$
d_i=N(c_i-c_1)
$$
is an integer for every $i=1,\ldots,t$.  Thus
$$
0=d_1<d_2<\cdots<d_t=M
$$
for some integer $M$.  Let
$$
\alpha(s)=N(s-c_1).
$$
For each $x\in X$, define
$$
I_x^+=[\alpha(\ell_x),\alpha(r_x)].
$$
For each integer $q=0,1,\ldots,M$, add a new element $p_q$ represented by the
singleton interval $[q,q]$.  Let
$$
X^+=X\cup\{p_0,p_1,\ldots,p_M\},
$$
and let
$$
\mathcal F^+=\{I_x^+:x\in X\}\cup\{[q,q]:0\le q\le M\}.
$$
The positive affine map $\alpha$ preserves the order relations among the
elements of $X$, so the interval order $P^+$ represented by $\mathcal F^+$
contains $P$ as the subposet on $X$.  Since every integer in
$\{0,1,\ldots,M\}$ occurs in $\mathcal F^+$ as both a left endpoint and a right
endpoint, $\mathcal F^+$ is the canonical representation of $P^+$ by Proposition~\ref{prop:canonical}.  The poset
$P^+$ and representation $\mathcal F^+$ are called the \emph{augmentation} and
\emph{augmented representation}, respectively, determined by $P$ and
$\mathcal F$.

\begin{lemma} \label{lem:augmentation}
	Suppose that $\mathcal F=\{I_x=[\ell_x,r_x]:x\in X\}$ is a rational interval
	representation of an interval order $P=(X,\prec)$.  Let $P^+$ and
	$\mathcal F^+$ be the augmentation determined by $P$ and $\mathcal F$.  If
	$S\subseteq \mathcal F$ is an irreducible fundamental extender with respect to
	$\mathcal F$, then
	$$
	S^+=\{I_x^+:I_x\in S\}
	$$
	is a Hilbert set of $P^+$.
\end{lemma}
\begin{proof}
See Appendix~\ref{app:proofs}.
\end{proof}

The characterization of Hilbert fragments reduces to the connected case.  The reduction is obtained by concatenating underlying representations of the connected components, an operation called \emph{concatenation}.  We now describe this concatenation operation and prove that it preserves the selected Hilbert sets (Lemma~\ref{lem:concatenation}); this yields the closure of Hilbert fragments under disjoint union (Lemma~\ref{lem:disjointunion}).

Let $P_1,\ldots,P_k$ be interval orders on pairwise disjoint ground sets
$X_1,\ldots,X_k$, and let $\mathcal C(P_i)$ be the canonical representation of
$P_i$.  Suppose that the magnitude of $P_i$ is $m_i$, and write the endpoint
set of $\mathcal C(P_i)$ as $\{0,1,\ldots,m_i-1\}$.  Set
$$
a_1=0,
\qquad
a_{i+1}=a_i+m_i
\quad\text{for }i=1,\ldots,k.
$$
Translate $\mathcal C(P_i)$ by $a_i$, so that its endpoint set becomes
$$
\{a_i,a_i+1,\ldots,a_i+m_i-1\}.
$$
Let $\mathcal C_i'$ denote this translated copy of $\mathcal C(P_i)$, let
$$
\mathcal C=\mathcal C_1'\cup\cdots\cup\mathcal C_k',
$$
and let $P$ be the interval order represented by $\mathcal C$.  The resulting
interval order $P$ is the \emph{ordered concatenation} of the $P_i$ (also called \emph{lexicographic sum over a chain}).  Observe
that $\mathcal C$ is the canonical representation of $P$.

If $A\subseteq\mathcal C(P_i)$ is a subfamily, write $A^{(i)}\subseteq
\mathcal C_i'$ for its translated image.

\begin{lemma} \label{lem:concatenation}
	Let $P$ be the ordered concatenation of interval orders $P_1,\ldots,P_k$.  If
	$S$ is a Hilbert set of $P_i$, then  $S^{(i)}$ is a Hilbert set of $P$.
\end{lemma}
\begin{proof}
See Appendix~\ref{app:proofs}.
\end{proof}

Consequently, we have this useful conclusion.

\begin{lemma} \label{lem:disjointunion}
	The disjoint union of Hilbert fragments is a Hilbert fragment.
\end{lemma}

\begin{proof}
	Suppose $G_1,\ldots,G_k$ are Hilbert fragments, and choose realizations
	$(P_i,\varphi_i)$ of $G_i$.  Concatenating the canonical representations
	$\mathcal C(P_i)$ gives an interval order $P$ in which the translated image of
	every selected Hilbert set $\varphi_i(v)$ is a Hilbert set, by
	Lemma~\ref{lem:concatenation}.  Hilbert sets selected from different blocks
	share no represented canonical interval, while intersections among selected
	Hilbert sets inside a single block are unchanged.  Thus the map
	$$
	\varphi(v)=\varphi_i(v)^{(i)}
	\quad\text{for }v\in V(G_i)
	$$
	realizes the disjoint union of the $G_i$ as an induced subgraph of
	$\mathcal H_P$.
\end{proof}

Our goal is to characterize Hilbert fragments.  Since induced subgraphs of
Hilbert extender graphs are Hilbert fragments, every connected component of a
Hilbert fragment is again a Hilbert fragment.  This observation, together with
Lemma~\ref{lem:disjointunion}, shows that it suffices to characterize connected
Hilbert fragments.


\section{Track-labeling axioms}
\label{sec:properties}

This section introduces the track-axioms for track-labelings and proves that every
connected Hilbert fragment admits one
(Proposition~\ref{prop:axioms_satisfied}). 
This sets the stage for
the paper’s main characterization theorem, Theorem~\ref{thm:characterization},
which states that a connected graph is a Hilbert fragment if and only if it
admits a track-labeling.

Let $H=(V,E)$ be a connected, simple, undirected graph. We consider vertex labels from $\mathbb{Q}\cup\{-\infty,+\infty\}$. 
Rational labels are called \emph{finite labels}; $\pm\infty$ labels are called \emph{extreme labels}.

A function
$$
\lambda:V\to\mathbb{Q}\cup\{-\infty,+\infty\}
$$
is called a \emph{track-labeling} of $H$ if it satisfies the following track-axioms:

\begin{description}
[style=standard,
leftmargin=\leftmargini,
itemindent=-\parindent]
\item[\clique] Vertices with the same finite label induce a clique.
\item[\independence] Vertices with extreme labels form an independent set.
\item[\path] If $(x,u,v)$ is an induced path and $\lambda(u)=\lambda(v)$, then $\lambda(x)$ is finite.
\item[\betweenness] If $\lambda(a)<\lambda(b)<\lambda(c)$ and
$ac\in E$,\ then $ab\in E$ and $bc\in E$.
\item[\claw] In every induced claw, at least two leaves have the same extreme label.
\end{description}

A graph together with a track-labeling is called a \emph{track-labeled graph}. An example of a track-labeled graph is given in Figure~\ref{fig:running-example-graph}.
When the labeling is relevant, we write $(H,\lambda)$ for a track-labeled graph with underlying graph $H$ and track-labeling $\lambda$.
Intuitively, $\lambda$ places the vertices on a single ordered track with two extremes, and \betweenness\ imposes the 
monotonicity constraint that reflects the interval-like nature of the model.

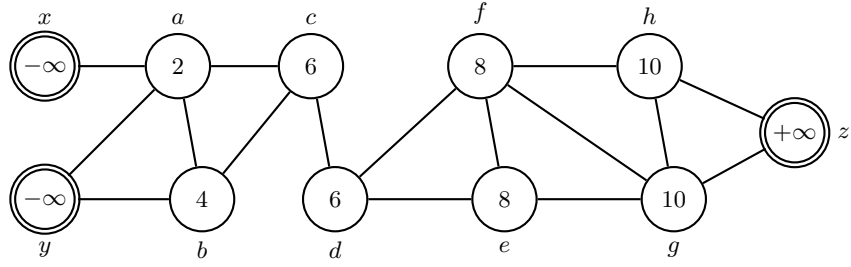
\begin{figure}
	\centering
	\begin{tikzpicture}[scale=0.8]
		
		\tikzset{
			vertex/.style={
				circle,
				draw,
				thick,
				minimum size=8.5mm,
				inner sep=0pt,
				font=\small
			},
			extreme/.style={
				vertex,
				double,
				double distance=1pt
			},
			every label/.style={font=\small}
		}
		
		\node[extreme,label=above:$x$] (x) at (-2.2,2.2) {$-\infty$};
		\node[extreme,label=below:$y$] (y) at (-2.2,0.0) {$-\infty$};
		
		\node[vertex,label=above:$a$] (a) at (0.0,2.2) {$2$};
		\node[vertex,label=below:$b$] (b) at (0.4,0.0) {$4$};
		
		\node[vertex,label=above:$c$] (c) at (2.2,2.2) {$6$};
		\node[vertex,label=below:$d$] (d) at (2.6,0.0) {$6$};
		
		\node[vertex,label=above:$f$] (f) at (5.0,2.2) {$8$};
		\node[vertex,label=below:$e$] (e) at (5.4,0.0) {$8$};
		
		\node[vertex,label=above:$h$] (h) at (7.8,2.2) {$10$};
		\node[vertex,label=below:$g$] (g) at (8.2,0.0) {$10$};
		
		\node[extreme,label=right:$z$] (z) at (10.2,1.1) {$+\infty$};
		
		\draw[thick]
		(x)--(a)
		(y)--(a)
		(y)--(b)
		(a)--(b)
		(a)--(c)
		(b)--(c)
		(c)--(d)
		(d)--(e)
		(d)--(f)
		(e)--(f)
		(e)--(g)
		(f)--(g)
		(f)--(h)
		(g)--(h)
		(g)--(z)
		(h)--(z);
		
	\end{tikzpicture}
	\caption{An example of a track-labeled graph.  The value displayed inside each vertex is its track-label.}
	\label{fig:running-example-graph}
\end{figure}

Note that the track-labeling axioms are independent; see Figure~\ref{fig:independent}.

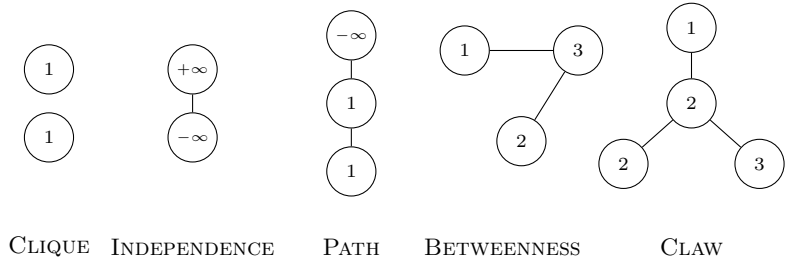
\begin{figure}
\begin{tikzpicture}[
    v/.style={circle, draw, minimum size=6.5mm, inner sep=0.8pt, font=\scriptsize},
    every node/.style={font=\small}
]

\def\labely{-1.9}

\begin{scope}[xshift=0cm]
    \node[v] (g1a) at (0,0.45) {$1$};
    \node[v] (g1b) at (0,-0.45) {$1$};
    \node at (0,\labely) {\clique};
\end{scope}

\begin{scope}[xshift=1.9cm]
    \node[v] (g2a) at (0,0.45) {$\scriptstyle +\infty$};
    \node[v] (g2b) at (0,-0.45) {$\scriptstyle -\infty$};
    \draw (g2a) -- (g2b);
    \node at (0,\labely) {\independence};
\end{scope}

\begin{scope}[xshift=4.0cm]
    \node[v] (g3a) at (0,0.9) {$\scriptstyle -\infty$};
    \node[v] (g3b) at (0,0) {$1$};
    \node[v] (g3c) at (0,-0.9) {$1$};
    \draw (g3a) -- (g3b) -- (g3c);
    \node at (0,\labely) {\path};
\end{scope}

\begin{scope}[xshift=6.0cm]
    \node[v] (g4a) at (-.5,0.7) {$1$};
    \node[v] (g4b) at (1,0.7) {$3$};
    \node[v] (g4c) at (0.25,-.5) {$2$};
    \draw (g4a) -- (g4b) -- (g4c);
    \node at (0,\labely) {\betweenness};
\end{scope}

\begin{scope}[xshift=8.5cm]
    \node[v] (g5c) at (0,0) {$2$};
    \node[v] (g5a) at (0,1.0) {$1$};
    \node[v] (g5b) at (-0.9,-0.8) {$2$};
    \node[v] (g5d) at (0.9,-0.8) {$3$};

    \draw (g5c) -- (g5a);
    \draw (g5c) -- (g5b);
    \draw (g5c) -- (g5d);

    \node at (0,\labely) {\claw};
\end{scope}

\end{tikzpicture}
\caption{Independence of the track-axioms. For each axiom, the displayed labeled graph satisfies all of the other track-axioms but violates the axiom named beneath it.}
\label{fig:independent}
\end{figure}

Observe that the result of applying a rational translation or a nonzero rational scaling to
a track-labeling is again a track-labeling.
Scaling by a negative rational reverses the order; in
particular, scaling by $-1$ is a \emph{reflection}.

Our immediate goal is to prove that any gap-labeling of a connected Hilbert fragment induced 
by a connected canonical representation satisfies the track-axioms.  
Proposition~\ref{prop:axioms_satisfied} then implies that every connected Hilbert fragment admits a track-labeling.  
The next lemmas establish the structural properties of gap-labeled Hilbert fragments needed for this proof.

\begin{lemma}
\label{lem:simplicial}
Every extreme-labeled vertex of a gap-labeled Hilbert fragment is a simplicial\footnote{A vertex is simplicial if its neighborhood induces a clique.} vertex. 
\end{lemma}

\begin{proof} Let $H=(V,E)$ be a gap-labeled Hilbert fragment.
Let $u,v\in V$ be distinct neighbors of an extreme-labeled vertex $z\in V$. Because $z$ is extreme-labeled, its gap set is empty, and by Lemma~\ref{lem:Hilbertset}, its borderline set contains a single interval $J$. Since $u$ and $v$ are neighbors of $z$, their Hilbert sets contain $J$, and therefore they are adjacent.
\end{proof}

\begin{lemma}\label{lem:Twintriple}
Suppose that $(H,\eta)$ is a gap-labeled Hilbert fragment.
Let $(x,u,v)$ be an induced 3-path in $H$, and assume $\eta(u)=\eta(v)=c$ is finite.\footnote{\ These labeled 3-paths, called \emph{twin triples} in Section~\ref{sec:digraph}, play a fundamental role in the proof of the representation theorem (Theorem~\ref{thm:representation}).}
If $J$ is an interval in $S_x\cap S_u$, then $J\in Z_u$. Furthermore,
\begin{enumerate}
\item[--] if $\eta(x)<c$ then  $J$ is a left-borderline interval  of $S_u$,
\item[--] if $\eta(x)>c$ then $J$ is a right-borderline interval  of $S_u$.
\end{enumerate} 
\end{lemma}
\begin{proof}
Note that $G_u=G_v$, because $\eta(u)=\eta(v)$.
If $J\in G_u$, then $J\in G_v$, so $J\in S_v$, and hence $J\in S_x\cap S_v$, contradicting that $x$ and $v$ are not adjacent.
Thus $J\notin G_u$, and therefore $J\in Z_u$.

Since $\eta(u)=c$ and $J\in Z_u$, the interval $J$ is either left-borderline
or right-borderline at the gap $[c-1,c]$.  Suppose $\eta(x)<c$.  Since
$J\in S_x$, the interval $J$ must either contain the gap indexed by $\eta(x)$
or lie in the chosen borderline set of $S_x$.  If $J$ began at $c$, then it
could do neither, because all endpoints relevant to the decomposition of
$S_x$ lie strictly to the left of $c$.  Hence $J$ must end at $c-1$, so $J$ is
left-borderline.  The case $\eta(x)>c$ is symmetric.
\end{proof}

We need the following observation. 

\begin{lemma}\label{lem:connected-realization}
	Every connected Hilbert fragment admits a realization by a connected canonical
	interval representation.
\end{lemma}
\begin{proof}
See Appendix~\ref{app:proofC}.
\end{proof}

\begin{proposition}
\label{prop:axioms_satisfied}
If $H$ is a connected Hilbert fragment, then $H$ admits a track-labeling.
\end{proposition}
\begin{proof}
Let $H=(V,E)$ be a connected Hilbert fragment.
By Lemma~\ref{lem:connected-realization}, $H$ has a connected canonical interval representation.
Let $\lambda = \eta$ be a gap-labeling of $H$ determined by this representation.
We verify that $\lambda$ satisfies the track-axioms.

\begin{description}
\item[\clique]
Suppose $\lambda(x)$, $\lambda(y)$ are finite, and $\lambda(x)=\lambda(y)$. 
The Hilbert sets $S_x$ and $S_y$ have the same nonempty gap set; therefore $x$ and $y$ are adjacent. 

\item[\independence]
If $\lambda(x)$ and $\lambda(y)$ are extreme labels, then their Hilbert sets have empty gap sets.
By Lemma~\ref{lem:Hilbertset}, each of the Hilbert sets consists of a single (distinct) interval. Hence the Hilbert sets are disjoint, so $x$ and $y$ are not adjacent.

\item[\path]
Let $(x,u,v)$ be an induced path with $\lambda(u)=\lambda(v)$ finite.
Assume for contradiction that $\lambda(x)$ is extreme. Then $S_x=Z_x=\{I\}$ by Lemma~\ref{lem:Hilbertset}.
Since $xu\in E$, we have $I\in S_x\cap S_u$, and Lemma~\ref{lem:Twintriple} yields $I\in Z_u$.
As $\lambda(u)$ is finite, $G_u\neq\varnothing$, so $S_u\setminus\{I\}\neq\varnothing$.
Hence $S_u$ can be decomposed as the union of $\{I\}$ and $S_u\setminus\{I\}$,
contradicting that $S_u$ is a Hilbert set.

\item[\betweenness]
Let $\lambda(a)<\lambda(b)<\lambda(c)$ and assume $ac\in E$.
Pick $I\in S_a\cap S_c$.
Since $\lambda(b)$ lies strictly between $\lambda(a)$ and $\lambda(c)$, the interval $I$ meets every intermediate gap, in particular
$I\in G_b\subseteq S_b$.
Therefore $S_a\cap S_b\neq\varnothing$ and $S_b\cap S_c\neq\varnothing$, so $ab,bc\in E$.

\item[\claw]
Let  $(v_0;v_1, v_2, v_3)$ be a claw with center $v_0$. By Lemma~\ref{lem:simplicial}, $v_0$  is not a simplicial vertex, so $\lambda(v_0)=\ $ $\lambda_0$ is finite. 
We claim that 
two leaf labels are equal. 

Assume, for a contradiction, that  $\lambda(v_1)<\lambda(v_2)<\lambda(v_3)$.   If 
$\lambda(v_2)>\ $ {$\lambda_0$}, then $v_0$, $v_2$, $v_3$ form a triple violating \betweenness. If $\lambda(v_2)<\ $ $\lambda_0$, then $v_1$, $v_2$, $v_0$ violate \betweenness. Hence $\lambda(v_0)=\lambda(v_2)=\ $ {$\lambda_0$}.

Now apply Lemma~\ref{lem:Twintriple} to the induced paths 
$(v_1,v_0,v_2)$ and  $(v_3,v_0,v_2)$. Since $\lambda(v_1)<\ $ $\lambda_0$, any $J_1\in S_{v_1}\cap S_{v_0}$ is a left-borderline interval in $Z_{v_0}$;
since $\lambda(v_3)>\ $ $\lambda_0$, any $J_3\in S_{v_3}\cap S_{v_0}$ is a right-borderline interval in $Z_{v_0}$. 
By the definition of a Hilbert set, $Z_{v_0}$ cannot contain borderline intervals from distinct sides of the gap  $[\lambda_0-1,\lambda_0]$, a contradiction. Therefore, among the leaf  labels   some two are equal, say
$\lambda(v_1)=\lambda(v_2)$. 

If their common label were finite, then $G_{v_1}=G_{v_2}$ would be a nonempty
gap set, forcing $S_{v_1}\cap S_{v_2}\ne\varnothing$, contrary to
$v_1v_2\notin E$.  Hence their common label is extreme.
\end{description}
\end{proof}

The results of Sections~\ref{sec:bi-fork-free} and~\ref{sec:eliminate} prove the converse of Proposition~\ref{prop:axioms_satisfied}, 
completing the verification that the track-axioms characterize connected Hilbert fragments.

Note that the connectedness hypothesis in Proposition~\ref{prop:axioms_satisfied} is essential. Consider the disjoint union of two H-graphs 
(two $P_3$'s  whose middle vertices are adjacent), and an edge. 
It is an easy exercise to find a track-labeling for each component. 
The results of this paper therefore show that both components are Hilbert fragments, 
and then Lemma~\ref{lem:disjointunion} implies that their union $G$ is a Hilbert fragment. Yet one can show that there is no track-labeling for $G$.

\section{Twin-arcs digraph and bi-fork-free track-labelings}
\label{sec:digraph}

This section introduces an auxiliary digraph and several related notions that are
used in the proof of Theorem~\ref{thm:characterization} in Section~\ref{sec:eliminate}.

Let $H=(V,E)$ be a track-labeled graph with a track-labeling $\lambda$.
An induced path $(u,w,v)$ is a \emph{twin triple} if
$$
\lambda(u)\neq \lambda(w)=\lambda(v).
$$
To simplify notation, we use $\beta$ for $\lambda(w)$ (equal to $\lambda(v)$). Similarly, to emphasize cases, we use $\alpha$ for $\lambda(u)$ when $\lambda(u)<\beta$, and $\gamma$ when $\lambda(u)>\beta$. Accordingly, a twin triple has one of the two label patterns
$$
(\alpha,\beta,\beta)\quad\text{with }\alpha<\beta,
\qquad\text{or}\qquad
(\gamma,\beta,\beta)\quad\text{with }\gamma>\beta.
$$

Note that $\alpha$, $\beta$, and $\gamma$ are all finite. Indeed, $\beta$ is finite by \independence, whereas $\alpha$ and $\gamma$ are finite due to \path.

Using twin triples, we now define an arc-colored digraph on $V$ associated with $H$, called the \emph{twin-arcs digraph}. A colored arc from $w$ to $v$ is defined for each twin triple $(u,w,v)$ with $\beta=\lambda(w)=\lambda(v)$; it is colored red if $\beta<\lambda(u)$, and blue if $\beta>\lambda(u)$. This coloring agrees with the spectral metaphor: high $\beta$ indicates blue light, low $\beta$ indicates red light. If two arcs have the same endpoints, same direction, and same color, then only one copy is included. 

The twin triples of the track-labeled graph in Figure~\ref{fig:running-example-graph} are
\begin{align*}
	\beta=6:&\quad
	(a,c,d),\ (b,c,d),\ (e,d,c),\ (f,d,c),\\
	\beta=8:&\quad
	(h,f,e),\\
	\beta=10:&\quad
	(e,g,h).
\end{align*}
Equivalently, the nontrivial twin-arcs are the blue arc $c\to d$, the red arc
$d\to c$, the red arc $f\to e$, and the blue arc $g\to h$.  Figure~\ref{fig:running-example-twin-arcs} shows the
twin-arcs digraph of the track-labeled graph in Figure~\ref{fig:running-example-graph}.

\begin{figure} [ht]
	\centering
	\begin{tikzpicture}[scale=0.8,>=Stealth]
		
		\tikzset{
			vertex/.style={
				circle,
				draw,
				thick,
				minimum size=8.5mm,
				inner sep=0pt,
				font=\small
			},
			extreme/.style={
				vertex,
				double,
				double distance=1pt
			},
			every label/.style={font=\small}
		}
		
		\node[extreme,label=above:$x$] (x) at (-2.2,2.2) {$-\infty$};
		\node[extreme,label=below:$y$] (y) at (-2.2,0.0) {$-\infty$};
		
		\node[vertex,label=above:$a$] (a) at (0.0,2.2) {$2$};
		\node[vertex,label=below:$b$] (b) at (0.4,0.0) {$4$};
		
		\node[vertex,label=above:$c$] (c) at (2.2,2.2) {$6$};
		\node[vertex,label=below:$d$] (d) at (2.6,0.0) {$6$};
		
		\node[vertex,label=above:$f$] (f) at (5.0,2.2) {$8$};
		\node[vertex,label=below:$e$] (e) at (5.4,0.0) {$8$};
		
		\node[vertex,label=above:$h$] (h) at (7.8,2.2) {$10$};
		\node[vertex,label=below:$g$] (g) at (8.2,0.0) {$10$};
		
		\node[extreme,label=right:$z$] (z) at (10.2,1.1) {$+\infty$};
		
		\draw[cyan,thick,->] (c) to[bend left=18] (d);
		\draw[red,thick,->]  (d) to[bend left=18] (c);
		
		\draw[red,thick,->]  (f) -- (e);
		\draw[cyan,thick,->] (g) -- (h);
		
	\end{tikzpicture}
	\caption{The twin-arcs digraph of the track-labeled graph in Figure~\ref{fig:running-example-graph}. }
	\label{fig:running-example-twin-arcs}
\end{figure}
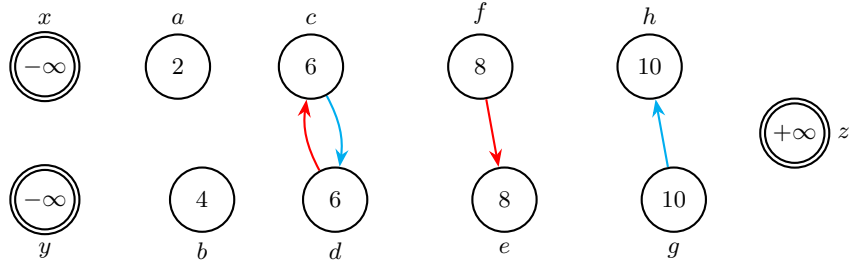

Figure~\ref{fig:configurations} displays the basic configurations in twin-arcs digraphs that are used later: bi-arrows, bi-cycles, bi-forks, and eye-to-eye bi-cycles. 
The notation $(x;y,z)$ for a bi-fork, and for the corresponding eye-to-eye bi-cycle, records the ordered roles of the vertices: 
$x$ is the center, $y$ is the blue eye, and $z$ is the red eye; equivalently, $xy$ is a blue arc and $xz$ is a red arc.

For a finite label $\beta$, the \emph{$\beta$-kernel digraph}, or simply the
\emph{$\beta$-kernel}, is the subdigraph $D(\beta)=(B,\mathcal A)$ of $D$
induced by the vertices labeled $\beta$ that occur in some twin triple;
equivalently, $B$ is the set of vertices labeled $\beta$ that are endpoints of
twin-arcs.  By \clique, the set $B$ induces a clique in $H$.

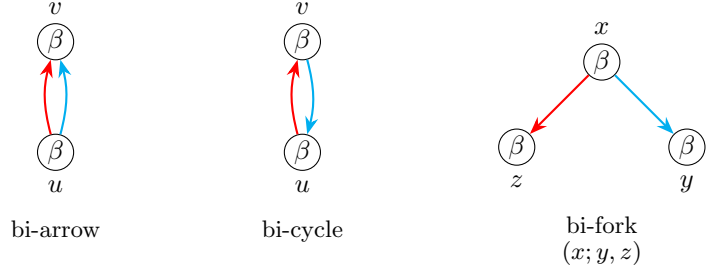
\begin{figure}[ht]
	\centering
	\captionsetup[subfigure]{labelformat=empty,justification=centering,singlelinecheck=false}
	
	\begin{subfigure}[t]{0.1\textwidth}
		\centering
		\makebox[\linewidth][c]{%
			\begin{tikzpicture}[scale=.65,>=Stealth, node distance=15mm,
				v/.style={circle,draw,inner sep=1pt,minimum size=10pt}]
				\node[v,label=below:$u$] (u) at (0,0){$\beta$};
				\node[v,label=above:$v$] (v)at (0,2.25) {$\beta$};
				\draw[->,red,thick] (u)  to[bend left=15] (v);
				\draw[->,cyan,thick] (u)  to[bend right=15]  (v);
		\end{tikzpicture}}
		\caption{bi-arrow}
	\end{subfigure}\hskip2cm
	\begin{subfigure}[t]{0.1\textwidth}
		\centering
		\makebox[\linewidth][c]{%
			\begin{tikzpicture}[scale=.65,>=Stealth, node distance=15mm,
				v/.style={circle,draw,inner sep=1pt,minimum size=10pt}]
				\node[v,label=below:$u$] (u) at (0,0){$\beta$};
				\node[v,label=above:$v$] (v)at (0,2.25) {$\beta$};
				\draw[->,red,thick] (u)  to[bend left=15] (v);
				\draw[->,cyan,thick] (v)  to[bend left=15]  (u);
		\end{tikzpicture}}
		\caption{bi-cycle}
	\end{subfigure}\hskip2cm
	\begin{subfigure}[t]{0.21\textwidth}
		\centering
		\makebox[\linewidth][c]{%
			\begin{tikzpicture}[scale=.65,>=Stealth, node distance=11mm,
				v/.style={circle,draw,inner sep=1pt,minimum size=10pt}]
				\node[v,label=above:$x$] (x) {$\beta$};
				\node[v,below right=of x,label=below:$y$] (y) {$\beta$};
				\node[v,below left=of x,label=below:$z$] (z) {$\beta$};
				\draw[->,cyan,thick] (x) -- (y);
				\draw[->,red, thick] (x) -- (z);
		\end{tikzpicture}}
		\caption{\shortstack[c]{bi-fork\\$(x;y,z)$}}
	\end{subfigure}

	\caption{The bi-arrow is  a forbidden configuration; the others are  used in the analysis of
twin-arcs digraphs. }
	\label{fig:configurations}
\end{figure}

\begin{lemma}\label{lem:kernel}
	For all $\beta$, $D(\beta)$ satisfies:
	\begin{enumerate}[label=(\alph*),ref=\alph*]
		\item\label{lem:kernel:no-biarrow} there is no bi-arrow;
		\item\label{lem:kernel:bi-fork-eye} every $(x;y,z)$ bi-fork extends to an $(x;y,z)$ eye-to-eye bi-cycle.
	\end{enumerate}
\end{lemma}

\begin{proof}
(a)
Assume, for a contradiction, that $D(\beta)$ contains a bi-arrow from $u$ to $v$.
By definition, $\lambda(u)=\lambda(v)=\beta$, and there exist twin triples
$(a,u,v)$ and $(c,u,v)$ such that $\lambda(a)<\beta$ and $\lambda(c)>\beta$.

Because all labels are finite in a twin triple, the vertices $u$, $v$, $a$, $c$ do not induce a claw. So 
$ac\in E$, and because $v$ is not adjacent to either $a$ or $c$, this contradicts \betweenness.

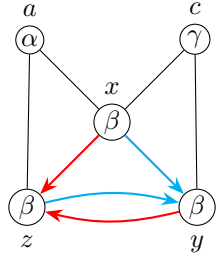
\begin{figure}[H]
\centering
	\begin{subfigure}[t]{0.21\textwidth}
		\centering
		\makebox[\linewidth][c]{%
			\begin{tikzpicture}[scale=.65,>=Stealth, node distance=11mm,
				v/.style={circle,draw,inner sep=1pt,minimum size=10pt}]
				\node[v,label=above:$x$] (x) {$\beta$};
				\node[v,below right=of x,label=below:$y$] (y) {$\beta$};
				\node[v,below left=of x,label=below:$z$] (z) {$\beta$};
			\node[v,above right=of x,label=above:$c$] (c) {$\gamma$};
			\node[v,above left=of x,label=above:$a$] (a) {$\alpha$};		
				\draw[->,cyan,thick] (x) -- (y);
				\draw[->,red, thick] (x) -- (z);
				\draw (z)--(a)--(x) (y)--(c)--(x);
\draw[red,thick,->]  (y) to[bend left=14] (z);
\draw[cyan,thick,->]  (z) to[bend left=14] (y);
		\end{tikzpicture}}
		
	\end{subfigure}

\caption{A bi-fork extends to an eye-to-eye bi-cycle.}
\label{fig:Catout}
\end{figure}
(b) 
Let $(x;y,z)$ be a bi-fork defined by twin triples $(a,x,y)$ and $(c,x,z)$ with
$\lambda(x)=\lambda(y)=\lambda(z)=\beta$, $\lambda(a)<\beta$, and $\lambda(c)>\beta$
(see  Figure~\ref{fig:Catout}).

By \clique, $y$ and $z$ are adjacent.  We next show that $az$ and $cy$ are
edges.  If $az\notin E$, then $(a,x,z)$ is a twin triple, producing a blue arc
from $x$ to $z$.  Together with the red arc from $x$ to $z$ coming from
$(c,x,z)$, this gives a bi-arrow, contradicting part (a).  Hence $az\in E$.
Similarly, if $cy\notin E$, then $(c,x,y)$ is a twin triple, producing a red
arc from $x$ to $y$; together with the blue arc from $x$ to $y$ coming from
$(a,x,y)$, this gives a bi-arrow.  Hence $cy\in E$.

Therefore $(a,z,y)$ and $(c,y,z)$ are twin triples, so the bi-fork extends to
the claimed eye-to-eye bi-cycle.
\end{proof}

A $\beta$-kernel is \emph{nontrivial} if it contains a bi-fork; otherwise, it is
\emph{trivial}.  By Lemma~\ref{lem:kernel}(\ref{lem:kernel:bi-fork-eye}), every
bi-fork in a $\beta$-kernel extends to the corresponding eye-to-eye bi-cycle.  A
track-labeled graph is \emph{bi-fork-free} if its twin-arcs digraph contains no
bi-fork; equivalently, $D(\beta)$ is trivial for every finite label $\beta$.  When
the underlying graph $H$ is clear, we also say that a track-labeling $\lambda$ of
$H$ is bi-fork-free if the track-labeled graph $(H,\lambda)$ is bi-fork-free.


Note that for a twin triple 
 $(u,w,v)$, a witness for the edge $uw$ is a borderline interval in $Z_w$, whose left or right position is determined by Lemma~\ref{lem:Twintriple}.

\begin{lemma}
\label{lem:bi-forkfree}
Every gap-labeling of a Hilbert fragment is bi-fork-free.
\end{lemma}

\begin{proof}
Assume that $(H,\eta)$ is a gap-labeled Hilbert fragment. 
Suppose, to the contrary, that there exists a vertex $w$ from which two twin-arcs of different
colors emanate in the twin-arcs digraph for $H$.  
This means there exist vertices $u,v,u',v'$ such that both
$$
(u,w,v)\quad\text{and}\quad (u',w,v')
$$
are twin triples, and $-\infty<\eta(u)<\eta(w)<\eta(u')<+\infty$.

According to Lemma~\ref{lem:Twintriple}, a witness for the edge $uw$ and a witness 
for the edge $u'w$ are borderline intervals in $Z_w$ from distinct sides, a contradiction.
\end{proof}

\section{Connected bi-fork-free track-labeled graphs}
\label{sec:bi-fork-free}
Ultimately, our goal is to prove that every connected graph that admits a track-labeling is a Hilbert fragment. Although Lemma~\ref{lem:bi-forkfree} shows that every gap-labeling is bi-fork-free, an arbitrary
track-labeled graph $H$ need not be bi-fork-free. Because bi-forks cause difficulties, we first consider bi-fork-free track-labeled graphs. We resolve the general case in Section~\ref{sec:eliminate} by eliminating bi-forks successively, adjusting the track-labeling appropriately.
 
\subsection*{Side-markers and witness intervals}
\label{sub:witness}
Let $H=(V,E)$ be a connected  track-labeled graph with a bi-fork-free labeling $\lambda$. 
The one-vertex graph is plainly a Hilbert fragment, so in the construction below
we assume that $H$ has at least two vertices.
By translation and scaling, we may assume that all finite labels are positive integers  
and \emph{loose}; that is, for all finitely labeled
vertices $u,v\in V$,
\begin{equation*}
\lambda(u)\neq\lambda(v)\quad\Rightarrow\quad |\lambda(u)-\lambda(v)|\ge 2.
\end{equation*}
We may also assume that the smallest finite label is $2$ and the largest is $m-2$ for
some positive integer $m$.
The unused labels $1$ and $m-1$ are intentional.   They provide room at the two ends of the construction so that
the intervals $M_c$, defined below, function as bounded gap intervals rather than as extreme-gap artifacts. 
These conventions ensure the required gap-borderline form and are used later 
to prove that the constructed sets are irreducible fundamental extenders.

The values of $\lambda$ are taken as interval endpoints and determine gaps along the real line. The two \emph{extreme gaps} are $[-\infty,0]$ and $[m-1,+\infty]$, and for every $1\le c\le m-1$, the unit interval $[c-1,c]$ is called a \emph{bounded gap}.

We now construct a multiset of intervals determining an interval order whose
Hilbert extender graph ultimately contains an induced subgraph isomorphic to H.
For  every $1\le c\le m-1$,
we set the gap marker interval $M_c=[c-1,c]$.  
These intervals serve as witnesses for edges between vertices sharing the same finite label, and guarantee irreducibility of some fundamental extenders.

If $\lambda(z)$ is an extreme label, then the edges  of $H$ incident with $z$  must be witnessed by a single interval, according to Lemma~\ref{lem:Hilbertset}.
First set
$$
\nu_z=\begin{cases}
\max\{\lambda(b): bz\in E\}&\text{if $\lambda(z)=-\infty$},\\
\min\{\lambda(b): bz\in E\}&\text{if $\lambda(z)=+\infty$},
\end{cases}
$$
which is well-defined because $H$ is connected with at least two vertices.  Now define the \emph{peripheral interval} $J(z)$ as follows:
$$
J(z)=\begin{cases}
	[0,\nu_z]&\text{if $\lambda(z)=-\infty$},\\
	[\nu_z-1,m-1]&\text{if $\lambda(z)=+\infty$}.
\end{cases}
$$

\begin{lemma}\label{lem:periph-iff}
If $\lambda(z)=\pm\infty$ and $\lambda(x)$ is finite, then
$
zx\in E$ if and only if 
$ [\lambda(x)-1,\lambda(x)]\subseteq J(z).
$
\end{lemma}

\begin{proof}
We prove the case $\lambda(z)=-\infty$; the case $\lambda(z)=+\infty$ is symmetric.

($\Rightarrow$)
Assume $zx\in E$. Note that
$J(z)=[0,\nu_z]$, where $\nu_z=\max\{\lambda(b):zb\in E\}$.
In particular, $2\leq \lambda(x)\le \nu_z$, therefore
$[\lambda(x)-1,\lambda(x)]\subseteq [0,\nu_z]=J(z)$.

($\Leftarrow$)
Assume $[\lambda(x)-1,\lambda(x)]\subseteq J(z)$. Choose a neighbor $c$ of $z$ with $\lambda(c)=\nu_z$.

If $\lambda(x)<\nu_z$, apply \betweenness\ to the edge $zc$ and the vertex $x$.
Since $\lambda(z)=-\infty<\lambda(x)<\lambda(c)$, \betweenness\ yields $zx\in E$.

If $\lambda(x)=\nu_z$, then $cx\in E$ by \clique. If $zx\notin E$, then
$(z,c,x)$ would be an induced $P_3$ with $\lambda(z)$ extreme and
$\lambda(c)=\lambda(x)$, contradicting \path. Hence $zx\in E$.
\end{proof}

Next we introduce intervals that witness edges  of $H$ whose endpoints have distinct finite labels.  
For an edge $uw\in E$ with $\lambda(u)<\lambda(w)$, the witness interval is denoted $J(uw)$.  
Its endpoints are chosen so that any resulting borderline incidence is consistent with Lemma~\ref{lem:Twintriple}. 
The witnesses must also respect an additional 
restriction arising from chorded 4-cycles, expressed
in terms of the following configuration, called a diamond ring.

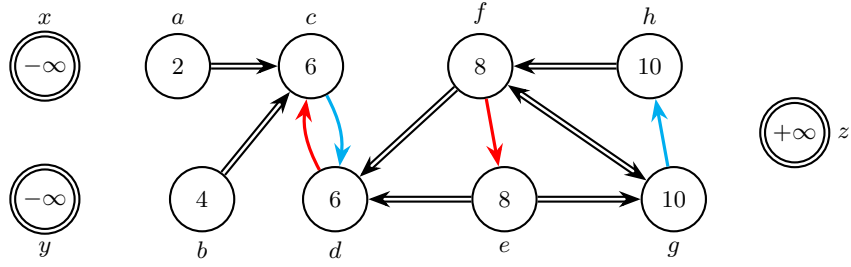
\begin{figure}
	\centering
	\begin{tikzpicture}[scale=0.8,>=Stealth]
		
		\tikzset{
			vertex/.style={
				circle,
				draw,
				thick,
				minimum size=8.5mm,
				inner sep=0pt,
				font=\small
			},
			extreme/.style={
				vertex,
				double,
				double distance=1pt
			},
			every label/.style={font=\small}
		}
		
		\node[extreme,label=above:$x$] (x) at (-2.2,2.2) {$-\infty$};
		\node[extreme,label=below:$y$] (y) at (-2.2,0.0) {$-\infty$};
		
		\node[vertex,label=above:$a$] (a) at (0.0,2.2) {$2$};
		\node[vertex,label=below:$b$] (b) at (0.4,0.0) {$4$};
		
		\node[vertex,label=above:$c$] (c) at (2.2,2.2) {$6$};
		\node[vertex,label=below:$d$] (d) at (2.6,0.0) {$6$};
		
		\node[vertex,label=above:$f$] (f) at (5.0,2.2) {$8$};
		\node[vertex,label=below:$e$] (e) at (5.4,0.0) {$8$};
		
		\node[vertex,label=above:$h$] (h) at (7.8,2.2) {$10$};
		\node[vertex,label=below:$g$] (g) at (8.2,0.0) {$10$};
		
		\node[extreme,label=right:$z$] (z) at (10.2,1.1) {$+\infty$};
		
		\draw[->,double,double distance=1pt,thick] (a)--(c);
		\draw[->,double,double distance=1pt,thick] (b)--(c);
		
		\draw[->,double,double distance=1pt,thick] (e)--(d);
		\draw[->,double,double distance=1pt,thick] (f)--(d);
		
		\draw[->,double,double distance=1pt,thick] (e)--(g);
		\draw[<->,double,double distance=1pt,thick] (f)--(g);
		\draw[->,double,double distance=1pt,thick] (h)--(f);
		
		\draw[cyan,very thick,->] (c) to[bend left=18] (d);
		\draw[red,very thick,->]  (d) to[bend left=18] (c);
		
		\draw[red,very thick,->]  (f) -- (e);
		\draw[cyan,very thick,->] (g) -- (h);
		
	\end{tikzpicture}
	\caption{The side markers and twin arcs for the track-labeled graph in Figure~\ref{fig:running-example-graph}.  Side markers are shown in black; blue twin arcs are drawn in cyan.  Note that $fg$ is a diamond chord.}
	\label{fig:running-example-side-twin}
\end{figure}

A \emph{diamond graph} is a graph isomorphic to \(K_4-e\). The edge joining
its two vertices of degree \(3\) is called the \emph{diamond chord}. 
Observe that a diamond in a track-labeled graph may have both finite and extremal vertex labels. We now define a special track-labeled diamond configuration in which all track-labels are finite. Subject to this requirement, the track-labeling is essentially forced and is shown in Figure~\ref{fig:diamond_ring}.

A \emph{diamond ring} 
\((x,a,y,b)\) is a track-labeled diamond graph whose chord is
\(xy\),  and whose vertices have two distinct finite labels satisfying
\[
-\infty<\lambda(x)=\lambda(a)=\alpha<\beta=\lambda(y)=\lambda(b)<+\infty.
\]

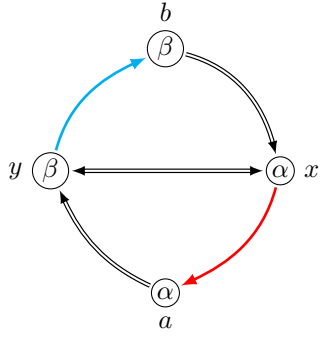
\begin{figure}[ht]
\centering

\begin{tikzpicture}[scale=.95,>=latex, node distance=15mm,
	v/.style={circle,draw,inner sep=1pt,minimum size=10pt}]
	\tikzset{
		side marker/.style={
			double,
			double distance=0.8pt,
			line width=0.45pt
		},
		twin arc/.style={
			line width=.1em,
			->,
			shorten >=1pt,
			shorten <=1pt
		}
	}
	
	\node[v,label=left:$y$]  (y) at (-1.6,0) {$\beta$};
	\node[v,label=right:$x$] (x) at (1.6,0) {$\alpha$};
	\node[v,label=above:$b$] (v) at (0,1.7) {$\beta$};
	\node[v,label=below:$a$] (u) at (0,-1.7) {$\alpha$};
	
	\draw[side marker,<->,shorten >=1pt,shorten <=1pt] (x) -- (y);
	\draw[side marker,->,shorten >=1pt,shorten <=1pt]
	(v) to[out=-15,in=105] (x);
	\draw[side marker,->,shorten >=1pt,shorten <=1pt]
	(u) to[out=155,in=-65] (y);
	
	\draw[cyan,twin arc]
	(y) to[out=75,in=205] (v);
	\draw[red,twin arc]
	(x) to[out=-105,in=25] (u);
\end{tikzpicture}

\caption{The diamond ring $(x,a,y,b)$ in the twin-arcs digraph associated with
$H$ ($-\infty<\alpha < \beta<+\infty$).}
\label{fig:diamond_ring}
\end{figure}

The term \emph{ring} is motivated by the fact that, in the \emph{hybrid digraph}
defined by the twin-arcs and side markers, the twin-arcs $x\to a$ and
$y\to b$, together with the side markers $a\mto y$ and $b\mto x$, form a
directed $4$-cycle.  
A  diamond
ring is shown in Figure~\ref{fig:diamond_ring}, while the graph
in Figure~\ref{fig:running-example-graph} contains the  diamond ring
\((f,e,g,h)\).


We create a log of constraints called \emph{side-marker arrows}, determined by twin triples. We interpret these arrows as orientation of the edges of $H$ added to the twin-arcs digraph. For each edge $uw\in E$ whose
endvertices have distinct finite labels, define the side-markers:
\begin{equation*}
\begin{array}{rl}
u\mto w,     & \text{if }(u,w,v)\text{ is a twin triple for some }v\in V, \\
u\mfrom w,   & \text{if }(w,u,z)\text{ is a twin triple for some }z\in V, \\
u\mtom w,    & \text{if } uw \text{ is a diamond chord, or both twin triples are present,}\\
u\mline w,   & \text{if no twin triple includes } uw. \\
\end{array}
\end{equation*}

It is understood that an edge $uw\in E$ carries a two-way side-marker $u\mtom w$ if
either both $(u,w,v)$ and $(w,u,z)$ are twin triples for some $v,z\in V$, or $uw$ is a
diamond chord. Multiple twin triples may produce the same side-marker, but only one is logged. Thus exactly one of these four markers appears between each pair of vertices with distinct finite labels.

The markers $u\mto w,$  and $u\mfrom w$ are called \emph{one-way side-markers}.  Figure~\ref{fig:diamond_ring} shows the side-markers of a  diamond ring. 


A \emph{central interval} is a witness for an edge of $H$ whose endvertices receive distinct finite labels.
For an edge $xy\in E$ with $2\leq\lambda(x)=\alpha<\lambda(y)=\beta\leq m-2$, we introduce a
central interval $J(xy)$, with endpoints specified below (see also Figure~\ref{fig:Jxy}):
$$
J(xy)=
\begin{cases}
[\alpha-1,\beta],     & \text{if } x\mline y,\\
[\alpha-1,\beta-1],   & \text{if } x\mto y,\\
[\alpha,\beta],       & \text{if } x\mfrom y,\\
[\alpha,\beta-1],     & \text{if } x\mtom y.
\end{cases}
$$
Note that $J(xy)\neq J(z)$ for all $x,y,z$, because $2\leq\lambda(x)=\alpha<\lambda(y)=\beta\leq m-2$ imply $J(xy)\subseteq [1,m-2]$, whereas, either $0$ or $m-1$ is an endpoint of the interval $J(z)$.

Observe that if a marker points toward $x$ then the gap marker interval $[\lambda(x)-1,\lambda(x)]$ is not covered by $J(xy)$ for any $y$.

Let $\mathcal{J}_0$ be the multiset of witness intervals constructed so far:
\begin{align*}
\mathcal{J}_0=\quad&\{M_c:1\leq c\leq m-1\}\quad\cup\\
&\{J(z):\lambda(z)=\pm\infty\}\quad\cup\\
&\{J(uw):uw\in E\text{ and }-\infty<\lambda(u)<\lambda(w)<+\infty\}.
\end{align*}
The final multiset $\mathcal J$ used to construct the representation is obtained from $\mathcal J_0$ by adding the private tags introduced in Section~\ref{sub:fundamental-extender-reps}.

For the loose track-labeled graph shown in Figure~\ref{fig:running-example-graph},  the twin-arcs and the side-markers are in Figure~\ref{fig:running-example-side-twin}; the initial interval multiset $\mathcal J_0$ consists of the gap intervals, the peripheral intervals, and the central intervals.
The gap intervals are
$$
M_i=[i-1,i]\qquad\text{for }1\le i\le 11.
$$
The peripheral intervals are
\begin{align*}
	J(x)&=[0,2],&
	J(y)&=[0,4],&
	J(z)&=[9,11].
\end{align*}
The central intervals are
\begin{align*}
	J(ab)&=[1,4],&
	J(ac)&=[1,5],&
	J(bc)&=[3,5],\\
	J(de)&=[6,8],&
	J(df)&=[6,8],&
	J(eg)&=[7,9],\\
	J(fg)&=[8,9],&
	J(fh)&=[8,10].
\end{align*}
The central intervals $J(de)$ and $J(df)$ have the same endpoints, but they
are distinct elements of the multiset $\mathcal J_0$.

\begin{figure}
  \centering
  \begin{tikzpicture}[scale=.65]
      \foreach \x in {1,2,5,6}
        \draw[line width=.05em] (\x,0) -- (\x,1);
   \node at (1,1.25){$\alpha-1$};    \node at (2,1.25){$\alpha$};
        \node at (3.5,-.25){$x\mline y$ };
    \node at (5,1.35){$\beta-1$};    \node at (6,1.25){$\beta$};
      \node[A](L)at(1,.5){};  \node[A](R)at(6,.5){};
      \draw[line width=.1em](L)--(R);
      \node[label=right: {$J(xy)$}]at(2.5,.75){};
\end{tikzpicture}
\hskip1.5cm
  \begin{tikzpicture}[scale=.65]
      \foreach \x in {1,2,5,6}
        \draw[line width=.05em] (\x,0) -- (\x,1);
   \node at (1,1.25){$\alpha-1$};    \node at (2,1.25){$\alpha$};
       \node at (3.5,-.25){$x \mfrom y$};
          \node at (5,1.35){$\beta-1$};    \node at (6,1.25){$\beta$};
      \node[A](L)at(2,.5){};  \node[A](R)at(6,.5){};
      \draw[line width=.1em](L)--(R);
      \node[label=right: {$J(xy)$}]at(2.5,.75){};
\end{tikzpicture}
\vskip1cm
  \begin{tikzpicture}[scale=.65]
      \foreach \x in {1,2,5,6}
        \draw[line width=.05em] (\x,0) -- (\x,1);
   \node at (1,1.25){$\alpha-1$};    \node at (2,1.25){$\alpha$};
        \node at (3.5,-.25){$x \mto y$};
          \node at (5,1.35){$\beta-1$};    \node at (6,1.25){$\beta$};
      \node[A](L)at(1,.5){};  \node[A](R)at(5,.5){};
      \draw[line width=.1em](L)--(R);
      \node[label=right: {$J(xy)$}]at(2.5,.75){};
\end{tikzpicture}
\hskip1.5cm
  \begin{tikzpicture}[scale=.65]
      \foreach \x in {1,2,5,6}
        \draw[line width=.05em] (\x,0) -- (\x,1);
   \node at (1,1.25){$\alpha-1$};    \node at (2,1.25){$\alpha$};
       \node at (3.5,-.25){$x \mtom y$};
          \node at (5,1.35){$\beta-1$};    \node at (6,1.25){$\beta$};
      \node[A](L)at(2,.5){};  \node[A](R)at(5,.5){};
      \draw[line width=.1em](L)--(R);
      \node[label=right: {$J(xy)$}]at(2.5,.75){};
\end{tikzpicture}
\caption{Side-markers determine the central intervals.}
\label{fig:Jxy}
\end{figure}
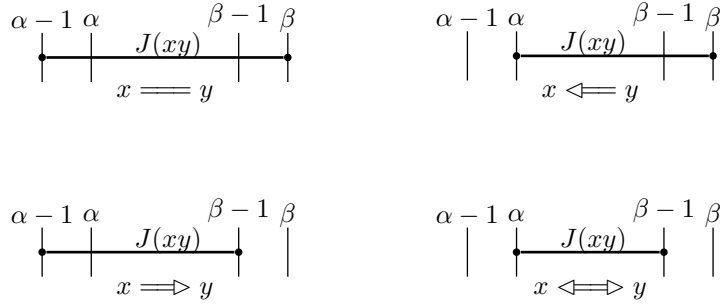
\subsection*{Construction and irreducibility of sets representing vertices}
\label{sub:fundamental-extender-reps}

For each vertex $x$ of $H$, we are nearly ready to define a set of intervals $S_x$ using the $\mathcal{J}_0$-intervals.
In this subsection our goal is to show that the sets $S_x$ are irreducible fundamental extenders. 
Each set $S_x$ is constructed as a union of a gap set and a borderline set. 
The gap sets are defined as
$$
G_x=\begin{cases}
\varnothing & \text{if $\lambda(x)=\pm\infty$}\\
\{I\in\mathcal{J}_0:M_\beta\subset I\}&\text{if $\lambda(x)=\beta$ and $-\infty<\beta<+\infty$}.\\
\end{cases}
$$
To accommodate the later addition of private tags, which will make the sets $S_x$ pairwise distinct, we first define pre-borderline sets.   Private tags are one-point intervals;
before adding  them, we must verify that each pre-borderline set lies on a single side of the relevant gap.   

For $x \in V$, 
define the \emph{pre-borderline set} $Z^0_x$ this way:
$$
Z^0_x=\begin{cases}
\{J(x)\} & \text{if $\lambda(x)=\pm\infty$}\\
\{J(ux):ux\in E, M_\beta\not\subset J(ux)\}&\text{if $\lambda(x)=\beta$ and $-\infty<\beta<+\infty$}.\\
\end{cases}
$$
Also define
$$
S^0_x=G_x\cup Z^0_x.
$$

A priori, the pre-borderline set $Z_x^0$ could contain intervals meeting different endpoints of the gap set $M_\beta$. 
 We call $Z_x^0$ \emph{legitimate} 
if all of its members meet $M_\beta$ at the same endpoint.  
Lemma~\ref{lem:legitimate} shows that $Z_x^0$ is legitimate for every finite-labeled vertex $x$.  
This will allow us to add private tags on the appropriate side of $M_\beta$ without destroying the gap--borderline form.  
Notice also that peripheral intervals are not used in the pre-borderline sets of finite-labeled vertices; by Lemma~\ref{lem:periph-iff}, they are accounted for through the corresponding gap sets.

Now we take a technical detour to address difficulties arising from diamond rings.
\begin{lemma}
\label{lem:let-it-shrink}
Let $H=(V,E)$ be a bi-fork-free track-labeled graph. Let $(x,a,y,b)$ be 
a diamond ring of the hybrid digraph associated with $H$, and let
$w\in V$, $w\notin\{x,a,y,b\}$.
If  $wx\in E$ carries a side-marker pointing to $x$ (that is, $w\mto x$ or
$w\mtom x$), then $\lambda(w)>\lambda(x)$.
If   $wy\in E$ carries a side-marker pointing to $y$, then
$\lambda(w)<\lambda(y)$.
\end{lemma}

\begin{proof}  
We prove the first implication; the second follows by reflection. Note that the presence of the $wx$ side-marker implies $\lambda(w)\neq\lambda(x)$. Suppose, for a contradiction, that $\lambda(w)<\lambda(x)=\alpha$.

Recall that the arc $x\to a$ of the diamond ring is  red (see Figure~\ref{fig:diamond_ring}).

\smallskip
\noindent\textsc{Case 1:} $w\mto x$ or ($w\mtom x$ and $wx$ is not a diamond chord). By definition of side-markers, there exists $t$ with
$\lambda(t)=\lambda(x)=\alpha$ such that $(w,x,t)$ is a twin triple. Because $\lambda(w)<\lambda(x)=\alpha$, it follows that $x\to t$ is a blue arc.

\smallskip
\noindent\textsc{Case 2:} $w\mtom x$ and $wx$ is a diamond chord. Then there exists a diamond ring
$(w,u,x,t)$, where $\lambda(u)=\lambda(w)<\lambda(x)=\lambda(t)=\alpha$, thus  $x\to t$ is a blue arc.

In either case, the red arc $x\to a$ and the newly produced blue arc $x\to t$ both leave
$x$.  If their heads coincide, then $D(\alpha)$ contains a bi-arrow,
contradicting Lemma~\ref{lem:kernel}(\ref{lem:kernel:no-biarrow}).  If their
heads are distinct, then these two arcs form a bi-fork centered at $x$,
contradicting that the track-labeling is bi-fork-free.  This contradiction proves
that $\lambda(w)>\alpha$.  By symmetry, the second implication follows as well.

\end{proof}

\begin{lemma}\label{lem:legitimate}
If $H=(V,E)$ is a bi-fork-free track-labeled graph and $\lambda(x)$ is finite,
then $Z^0_x$ is legitimate.
\end{lemma}
\begin{proof}
Fix $x\in V$ with finite label $\lambda(x)=\beta$.
Recall that for finite $\beta$,
$$
Z^0_x=\{J(ux):ux\in E, M_\beta\not\subset J(ux)\}.
$$
Thus $J(ux)\in Z^0_x$ holds exactly for those edges $ux$ whose marker
points \emph{toward} $x$ (by the definition of the central intervals). As noted earlier, in this case, the interval $[\beta-1,\beta]$ is not covered by $J(ux)$, and so 
$J(ux)$ meets the gap $[\beta-1,\beta]$ in exactly one endpoint.

Assume for a contradiction that $Z^0_x$ is not legitimate. Then there exist neighbors
$u,v$ of $x$ with $\lambda(u)<\beta<\lambda(v)$ such that the markers on both edges
$ux$ and $xv$ point toward $x$.

\smallskip
\noindent\textsc{Case 1:} neither $ux$ nor $xv$ is a diamond chord.
The marker on $ux$ implies a twin triple $(u,x,u')$ with
$\lambda(u')=\lambda(x)=\beta$ so there is a blue arc $xu'$. Similarly the marker $vx$ implies a twin triple $(v,x,v')$ with $\lambda(v')=\beta$ so there is a red arc $xv'$. By Lemma~\ref{lem:kernel}(\ref{lem:kernel:no-biarrow}), $u'\neq v'$, so there is a bi-fork centered at
$x$, contradicting that $H$ is bi-fork-free.

\smallskip
\noindent\textsc{Case 2:} at least one of $ux$ and $xv$ is a diamond chord.
If $ux$ is a diamond chord, then by Lemma~\ref{lem:let-it-shrink}, $\lambda(v)<\beta$, a contradiction. Similarly, if $vx$ is a diamond chord, then $\lambda(u)>\beta$, a contradiction.
\end{proof}

To ensure that the representing sets are distinct, we add \emph{private tags} to finite-labeled vertices.  
For each vertex $x$ with finite label $\lambda(x)=\beta$, 
add a new represented interval $T_x$, distinct from all previously constructed represented intervals.  By  Lemma~\ref{lem:legitimate}, $Z_x^0$ is legitimate.
If the intervals in $Z_x^0$ meet  the gap marker $M_\beta$ at its left endpoint, set $T_x=[\beta-1,\beta-1]$; if they meet $M_\beta$ at its right endpoint, set $T_x=[\beta,\beta]$.  
If $Z_x^0=\varnothing$, choose either side, say $T_x=[\beta-1,\beta-1]$.  

After all private tags have been added, set
$$
\mathcal J=\mathcal J_0\cup\{T_x:x\in V\text{ and }-\infty<\lambda(x)<+\infty\}.
$$

The final borderline sets are $$ Z_x= \begin{cases} Z_x^0, & \text{if } \lambda(x)=\pm\infty,\\ Z_x^0\cup\{T_x\}, & \text{if } -\infty<\lambda(x)<+\infty. \end{cases} $$ Finally, set $ S_x=G_x\cup Z_x. $ 
Since private tags contain no bounded gap, the gap sets $G_x$ are unchanged when passing from $\mathcal J_0$ to $\mathcal J$.

The essential purpose of the private tags is to ensure that the sets $S_x$ are pairwise distinct.
Each private tag contains no bounded gap and belongs only to its associated set $S_x$. 
Thus private tags create no new intersections among distinct sets $S_x$, and they do not destroy any existing intersections.  
They also preserve the fundamental-extender property: each private tag is placed on the same side 
of $M_\beta$ as the pre-borderline set to which it is added.  

\begin{corollary} \label{cor:fundamental}
The set $S_x$ is a fundamental extender with respect to $\mathcal{J}$.
\end{corollary}

\begin{proof}
For $\lambda(x)=\pm\infty$, one has $S_x=\{J(x)\}$, which is a fundamental extender at an extreme gap.  Otherwise, Lemma~\ref{lem:legitimate} implies that $Z^0_x$ is a borderline set of the gap set $G_x$.  
Because the private tag $T_x$ lies on the same side of the gap as $Z_x^0$, it follows that $S_x = G_x \cup Z_x$ is a fundamental extender.
\end{proof}

\begin{lemma}\label{lem:Hilbert}
The set $S_x$ is irreducible with respect to $\mathcal{J}$.
\end{lemma}
\begin{proof}
If $\lambda(x)=\pm\infty$, then $S_x=\{J(x)\}$, so it is irreducible.
So we may assume $-\infty < \lambda(x) <+\infty$.  Set $\beta = \lambda(x)$.

Suppose, for a contradiction, that $S_x$ is reducible.  Refining the parts if
necessary, choose a nontrivial partition
$$
S_x=A_1\dot\cup\cdots\dot\cup A_t
$$
into irreducible fundamental extenders.

Because $M_\beta\in S_x$, without loss of generality, $M_\beta\in A_1$. Fix a decomposition of $A_1$ and assume it uses the gap set $G[\beta'-1,\beta']$. Since $M_\beta\in A_1$, it follows that $\beta'\in\{\beta-1,\beta,\beta+1\}$. However, because $2\leq\beta\leq m-2$, and only one of $M_{\beta-1}$, $M_\beta$, $M_{\beta+1}$ is in $S_x$, it follows that $\beta'=\beta$.

Now we examine the decompositions of $A_2$.

\smallskip
\noindent\textsc{Case 1:} $A_2$ has a decomposition as $G\cup Z$ with $G\neq\varnothing$. Because every nonempty gap contains a marker interval, there exists $\alpha$ such that $M_\alpha\in G$. So $M_\alpha\in S_x$. Each $S_x$ contains only one marker interval, so $\alpha=\beta$, contradicting $A_1\cap A_2=\varnothing$.

\smallskip
\noindent\textsc{Case 2:} Every decomposition of $A_2$ as $G\cup Z$ satisfies $G=\varnothing$.  It follows that $G=G[-\infty,0]$ or $G=G[m-1,+\infty]$. In either case, because $A_2\neq\varnothing$, there exists a $z$ such that $J(z)\in Z$. By the definition of $S_x$ and $J(z)$, $[\beta-1,\beta]\subset J(z)$. Because the gap set of $A_1$ is $G[\beta-1,\beta]$, $J(z)\in A_1$, contradicting $A_1\cap A_2=\varnothing$.
\end{proof}

\subsection*{The representation theorem}
\label{sub:proof}

Let $H=(V,E)$ be a connected track-labeled graph with a bi-fork-free labeling $\lambda$, and consider the family 
$\{S_x:x\in V\}$ constructed in Section~\ref{sub:fundamental-extender-reps}. The goal of this subsection is to prove that the intersection graph of $\{S_x:x\in V\}$ is isomorphic to $H$ via $S_x\mapsto x$.

\begin{lemma}
\label{lem:periph-membership}
Let $\lambda(a)=\pm\infty$, and $x\in V$. Then
$
J(a)\in S_x $ 
if and only if 
$\bigl(x=a\bigr)\  \text{or}\  (ax\in E).
$
\end{lemma}

\begin{proof}
If $x$ is extreme, then $S_x=\{J(x)\}$, so $J(a)\in S_x$ forces $x=a$.

Assume $\lambda(x)$ is finite. Then $J(a)\notin Z_x$ (since $Z_x$ contains only
central intervals), hence $J(a)\in S_x$ is equivalent to $J(a)\in G_x$. By the
definition of $G_x$, this holds if and only if
$[\lambda(x)-1,\lambda(x)]\subseteq J(a)$.
By Lemma~\ref{lem:periph-iff}, this containment is equivalent to $ax\in E$. 
\end{proof}

\begin{proposition}\label{prop:intersection}
If $xy\in E$, then $S_x\cap S_y\neq\varnothing$.
\end{proposition}

\begin{proof}
If $\lambda(x)=\lambda(y)=c$ is finite, then $M_c\in S_x\cap S_y$.
If one of $x,y$ is extreme, say $x$, then $S_x=\{J(x)\}$, and by Lemma~\ref{lem:periph-membership}, $J(x)\in S_y$.

Finally, assume $0<\lambda(x)=\alpha<\beta=\lambda(y)<m$. By definition, the central interval $J(xy)\in\mathcal{J}$. If $[\alpha-1,\alpha]\subset J(xy)$ then $J(xy)\in G_x$; otherwise $J(xy)\in Z_x$. In either case $J(xy)\in S_x$. Similarly, $J(xy)\in S_y$.
\end{proof}

\begin{lemma}
	\label{lem:between-cliquex}
	Let $pq\in E$ be an edge and assume $-\infty < \lambda(p)<\lambda(q) < +\infty$.  The set
	$$
	B(p,q)=\{p,q\}\cup \{z\in V:\lambda(p)<\lambda(z)<\lambda(q)\}
	$$
	induces a clique in $H$.
\end{lemma}

\begin{proof}    
Fix $z\in B(p,q)\setminus\{p,q\}$. Since $\lambda(p)<\lambda(z)<\lambda(q)$, \betweenness\ implies that $pz\in E$ and
$zq\in E$.

Now take $z_1,z_2\in B(p,q)\setminus\{p,q\}$ with $\lambda(z_1)<\lambda(z_2)$. Because
$pz_2\in E$, applying \betweenness\ to the edge $pz_2$ with the vertex $z_1$,
and using $\lambda(p)<\lambda(z_1)<\lambda(z_2)$, implies $z_1z_2\in E$.

Since $pq\in E$,
$B(p,q)$ is a clique. 
\end{proof}

\begin{lemma} \label{lem:arrow-universalx}
If $xy\in E$   carries a one-way side-marker $x\mto y$, then $y$ is adjacent to every vertex of label $\lambda(x)$.
\end{lemma}

\begin{proof}
Let $u$ be any vertex with $\lambda(u)=\lambda(x)$. By
\clique, $ux\in E$. If $uy\notin E$, then $(y,x,u)$ is  a twin triple, which would force
the opposite marker $y\mto x$ on the edge $xy$, a contradiction. 
\end{proof}

\begin{lemma}
\label{lem:unmarked-diamondx}
Suppose $xy\in E$, $-\infty< \lambda(x)=\alpha<\beta=\lambda(y)<+\infty$.
If there exist $u,v\in V$ such that $uv\notin E$ and $\lambda(u)=\alpha$, $\lambda(v)=\beta$, 
then $x\mto y$ or $x\mfrom y$ or $x\mtom y$.
\end{lemma}

\begin{proof}
By \clique, $xu\in E$ and $yv\in E$.
If $xv\notin E$, then $(x,y,v)$ is a twin triple, which would force $x\mto y$. So we may assume $xv\in E$.
If $uy\notin E$, then $(y,x,u)$ is a twin triple, which would force $y\mto x$. Therefore we may assume $uy\in E$.
Thus $(x,u,y,v)$ is a diamond ring, so $x\mtom y$ follows by the definition of the marker arrows.
\end{proof}

\begin{proposition}\label{prop:correctness}
For all distinct $u,v\in V$, if $S_u\cap S_v\neq\varnothing$, then $uv\in E$.
\end{proposition}

\begin{proof}
Pick $I\in S_u\cap S_v$.
If $I$ is a private tag, then $I$ belongs to exactly one set $S_x$, so it cannot
belong to $S_u\cap S_v$ for distinct $u,v$.  Thus $I$ is one of the intervals
considered in the cases below.

\smallskip
\noindent\textsc{Case 1:} $I=M_c$ for some $2\le c\le m-2$.
Then $M_c=[c-1,c]\in S_u\cap S_v$ implies $\lambda(u)=\lambda(v)=c$, hence $uv\in E$ by \clique.

\smallskip
\noindent\textsc{Case 2:} $I=J(a)$ is a peripheral interval, where $\lambda(a)=\pm\infty$ and
$J(a)\in S_u\cap S_v$.
If $u=a$ or $v=a$, then $S_a=\{J(a)\}$ and Lemma~\ref{lem:periph-membership} yields
$uv\in E$.
Otherwise, Lemma~\ref{lem:periph-membership} gives $au,av\in E$ and $\lambda(u),\lambda(v)$
finite. By symmetry, we may assume $\lambda(u)\le\lambda(v)$ and $\lambda(a)=-\infty$. 
If $\lambda(u)=\lambda(v)$ then $uv\in E$ by \clique; otherwise, $\lambda(u)<\lambda(v)$, so
$uv\in E$ by \betweenness\ applied to the edge $av$ with the
vertex $u$.

\smallskip
\noindent\textsc{Case 3:} $I=J(xy)$ is a central interval, where $xy\in E$ and
$\lambda(x)=\alpha<\beta=\lambda(y)$ are finite. Set
$$
W(I)=\{w\in V:\ I\in S_w\}.
$$
It suffices to show that $W(I)$ induces a clique, because $u,v\in W(I)$.

The set of vertices $z\in W(I)$ for which $\alpha<\lambda(z)<\beta$ induces a clique, and each of them is adjacent to both $x$ and $y$ by
Lemma~\ref{lem:between-cliquex}.

Subcase 3a: $x\mtom y$.
Then $I=[\alpha,\beta-1]$, so $I$ contains neither endpoint gap $[\alpha-1,\alpha]$ nor
$[\beta-1,\beta]$. Hence no vertex of label $\alpha$ or $\beta$ can contain $I$ via its
gap set.

Consequently, if $w\in W(I)$ has label $\alpha$ or $\beta$, then necessarily $I\in Z_w$.
Since $I=J(xy)$ is a central interval, it follows that 
$w\in\{x,y\}$ by the definition of $Z_w$. Therefore the only vertices of
labels $\alpha$ and $\beta$ in $W(I)$ are $x$ and $y$ themselves. All other vertices
in $W(I)$ have labels strictly between $\alpha$ and $\beta$, hence $ W(I)\subset B(x,y)$. By Lemma~\ref{lem:between-cliquex}, $W(I)$ is a
clique.

Subcase 3b: $x\mto y$.
Then $I=[\alpha-1,\beta-1]$. Thus vertices of label $\alpha$ may contain $I$ via their gap
sets, but no vertex of label $\beta$ can contain $I$ via its gap set, since
$[\beta-1,\beta]\nsubseteq I$.

Hence, if $w\in W(I)$ has label $\beta$, then $I\in Z_w$. As above, since $I=J(xy)$, the
definition of $Z_w$ implies $w\in\{x,y\}$, and therefore $w=y$. So $y$ is the only
$\beta$-labeled vertex in $W(I)$.

Let $A=\{a\in W(I)\setminus\{x\}:\lambda(a)=\alpha\}$, $B=B(x,y)\cap W(I)$. Because the only vertex with label $\beta$ is $y$, it follows that $W(I)=A\dot\cup B$. By Lemma~\ref{lem:between-cliquex}, $B$ induces a clique, and by \clique, $A\cup\{x\}$ induces a clique, so it suffices to show that $A\subset N(y)$, and $az\in E$ for all $a\in A$ and $z\in B\setminus\{x,y\}$.

Consider an arbitrary $a\in A$. Lemma~\ref{lem:arrow-universalx} implies $a\in N(y)$.
Consider an arbitrary $z\in B\setminus\{x,y\}$. Apply \betweenness\ to the edge $ay$ with $z$ to conclude $az\in E$.

Subcase 3c: $x\mfrom y$.
This is symmetric to Subcase~3b.

Subcase 3d: $x\mline y$.
Then $I=[\alpha-1,\beta]$ contains both endpoint gaps $[\alpha-1,\alpha]$ and $[\beta-1,\beta]$.
Hence every vertex of label $\alpha$ or $\beta$ contains $I$ via its gap set and therefore
belongs to $W(I)$.

Suppose there exist vertices $u'$ and $v'$ with $\lambda(u')=\alpha$, $\lambda(v')=\beta$,
and $u'v'\notin E$. Since the edge $xy$ carries no one-way side-marker (it carries $x\mline y$),
the proof of Lemma~\ref{lem:unmarked-diamondx} forces $(x,u',y,v')$ to be a diamond ring.
In particular, $xy$ is a diamond chord, so by the definition of the marker arrows the edge $xy$
must carry the marker $x\mtom y$, contradicting $x\mline y$. Therefore every $\alpha$-vertex is
adjacent to every $\beta$-vertex.

Now let $u_\alpha\in W(I)$ have label $\alpha$, let $v_\beta\in W(I)$ have label $\beta$, and
let $z\in W(I)$ with $\alpha<\lambda(z)<\beta$. Since $u_\alpha v_\beta\in E$ and
$\lambda(u_\alpha)<\lambda(z)<\lambda(v_\beta)$, Lemma~\ref{lem:between-cliquex} (applied to the
edge $u_\alpha v_\beta$) yields $u_\alpha z\in E$ and $zv_\beta\in E$. Together with
the clique on interior vertices and the clique on each label class,  we conclude
that $W(I)$ induces a clique.

\smallskip\noindent
In all subcases, $W(I)$ is a clique. Hence $u,v\in W(I)$ implies $uv\in E$, completing the proof.
\end{proof}
\begin{theorem}\label{thm:representation}
	A connected  track-labeled bi-fork-free graph is a Hilbert fragment.
\end{theorem}

\begin{proof}
	Suppose that $(H,\lambda)$ is a connected bi-fork-free track-labeled graph.
	Let $\mathcal J$ be the multiset of intervals constructed from $(H,\lambda)$
	in Sections~\ref{sub:witness} and~\ref{sub:fundamental-extender-reps}, including 
	the peripheral and central intervals, the gap markers, and the private tags;
	and let $P$ be the interval order represented by $\mathcal J$.  The intervals
	in $\mathcal J$ have integer, hence rational, endpoints.  In particular,
	$\mathcal J$ is a rational representation of $P$.
	
	For each $x\in V$, let $S_x$ be the subfamily of $\mathcal J$ defined in
	Section~\ref{sub:fundamental-extender-reps}.  By Proposition~\ref{prop:intersection}, every edge
	$uv\in E$ satisfies
	$$
	S_u\cap S_v\ne\varnothing,
	$$
	and by Proposition~\ref{prop:correctness}, every intersection
	$$
	S_u\cap S_v\ne\varnothing
	$$
	for distinct vertices $u,v$ implies $uv\in E$.  Thus the indexed
	intersection graph of the family $\{S_x:x\in V\}$ is isomorphic to $H$.
	
	By Corollary~\ref{cor:fundamental}, each $S_x$ is a fundamental extender with
	respect to $\mathcal J$.  By Lemma~\ref{lem:Hilbert}, each $S_x$ is
	irreducible with respect to $\mathcal J$.  Let $P^+$ and $\mathcal J^+$ be
	the augmentation determined by $P$ and $\mathcal J$.  Lemma~\ref{lem:augmentation}
	implies that
	$$
	S_x^+=\{I^+:I\in S_x\}
	$$
	is a Hilbert set of $P^+$ for each $x\in V$.
	
	The affine map used in the augmentation gives a bijection from $\mathcal J$ to
	the original-interval subfamily of $\mathcal J^+$.  Hence, for all distinct
	$u,v\in V$,
	$$
	S_u\cap S_v\ne\varnothing
	\quad\text{if and only if}\quad
	S_u^+\cap S_v^+\ne\varnothing.
	$$
	By the private tags, the Hilbert sets $S_x^+$ are pairwise distinct.  Hence
	the map
	$$
	x\mapsto S_x^+
	$$
	is injective from $V$ into $V(\mathcal H_{P^+})$ and satisfies
	$$
	uv\in E
	\quad\text{if and only if}\quad
	S_u^+\cap S_v^+\ne\varnothing.
	$$
	Therefore $H$ is isomorphic to an induced subgraph of $\mathcal H_{P^+}$, and
	so $H$ is a Hilbert fragment.
\end{proof}

Figure~\ref{fig:running-example-side-twin} displays the twin-arcs digraph with the side markers for the track-labeled graph in Figure~\ref{fig:running-example-graph}. 
The table below lists the representing Hilbert sets, their gap--borderline decompositions, and their explicit interval contents.
\begingroup
\renewcommand{\arraystretch}{1.2}
$$
\begin{array}{r@{\;}c@{\;}l@{\quad}c@{\;}l}
	S_x &=& \varnothing\cup\{J(x)\}
	&=& \{J(x)\}\\
	S_y &=& \varnothing\cup\{J(y)\}
	&=& \{J(y)\}\\
	S_a &=& G[1,2]\cup\{T_a\}
	&=& \{M_2,J(x),J(y),J(ab),J(ac),T_a\}\\
	S_b &=& G[3,4]\cup\{T_b\}
	&=& \{M_4,J(y),J(ab),J(ac),J(bc),T_b\}\\
	S_c &=& G[5,6]\cup\{J(ac),J(bc),T_c\}
	&=& \{M_6,J(ac),J(bc),T_c\} \\
	S_d &=& G[5,6]\cup\{J(de),J(df),T_d\}
	&=& \{M_6,J(de),J(df),T_d\}\\
	S_e &=& G[7,8]\cup\{T_e\}
	&=& \{M_8,J(de),J(df),J(eg),T_e\}\\
	S_f &=& G[7,8]\cup\{J(fg),J(fh),T_f\}
	&=& \{M_8,J(de),J(df),J(eg), J(fg),J(fh),T_f\}\\
	S_g &=& G[9,10]\cup\{J(eg),J(fg),T_g\}
	&=& \{M_{10},J(fh),J(z),J(eg),J(fg),T_g\}\\
	S_h &=& G[9,10]\cup\{T_h\}
	&=& \{M_{10},J(fh),J(z),T_h\}\\
	S_z &=& \varnothing\cup\{J(z)\}
	&=& \{J(z)\}.
\end{array}
$$
\endgroup

\begin{figure}[htp]
	\centering
	\begin{tikzpicture}[x=.78cm,y=.45cm]
		
		\tikzset{
			endpoint/.style={font=\scriptsize},
			interval/.style={line width=.9pt},
			gapinterval/.style={line width=.7pt},
			tag/.style={circle,fill=black,inner sep=1.3pt},
			intervalname/.style={font=\scriptsize,anchor=west},
			tagname/.style={font=\scriptsize,anchor=south}
		}
		
		\foreach \x in {0,...,11}{
			\draw[gray!35,line width=.25pt] (\x,-0.5) -- (\x,8.4);
			\node[endpoint] at (\x,-0.95) {$\x$};
		}
		
		\draw[gapinterval] (0,.55) -- (1,.55);
		\node[font=\scriptsize] at (.5,.18) {$M_1$};
		
		\draw[gapinterval] (1,.68) -- (2,.68);
		\node[font=\scriptsize] at (1.5,.31) {$M_2$};
		
		\draw[gapinterval] (2,.55) -- (3,.55);
		\node[font=\scriptsize] at (2.5,.18) {$M_3$};
		
		\draw[gapinterval] (3,.68) -- (4,.68);
		\node[font=\scriptsize] at (3.5,.31) {$M_4$};
		
		\draw[gapinterval] (4,.55) -- (5,.55);
		\node[font=\scriptsize] at (4.5,.18) {$M_5$};
		
		\draw[gapinterval] (5,.68) -- (6,.68);
		\node[font=\scriptsize] at (5.5,.31) {$M_6$};
		
		\draw[gapinterval] (6,.55) -- (7,.55);
		\node[font=\scriptsize] at (6.5,.18) {$M_7$};
		
		\draw[gapinterval] (7,.68) -- (8,.68);
		\node[font=\scriptsize] at (7.5,.31) {$M_8$};
		
		\draw[gapinterval] (8,.55) -- (9,.55);
		\node[font=\scriptsize] at (8.5,.18) {$M_9$};
		
		\draw[gapinterval] (9,.68) -- (10,.68);
		\node[font=\scriptsize] at (9.5,.31) {$M_{10}$};
		
		\draw[gapinterval] (10,.55) -- (11,.55);
		\node[font=\scriptsize] at (10.5,.18) {$M_{11}$};
		
		\draw[interval] (0,7.85) -- (2,7.85);
		\node[intervalname] at (2.08,7.85) {$J(x)$};
		
		\draw[interval] (0,7.35) -- (4,7.35);
		\node[intervalname] at (4.08,7.35) {$J(y)$};
		
		\draw[interval] (1,6.85) -- (4,6.85);
		\node[intervalname] at (4.08,6.85) {$J(ab)$};
		
		\draw[interval] (1,6.35) -- (5,6.35);
		\node[intervalname] at (5.08,6.35) {$J(ac)$};
		
		\draw[interval] (3,5.85) -- (5,5.85);
		\node[intervalname] at (5.08,5.85) {$J(bc)$};
		
		\draw[interval] (6,7.85) -- (8,7.85);
		\node[intervalname] at (8.08,7.85) {$J(de)$};
		
		\draw[interval] (6,7.35) -- (8,7.35);
		\node[intervalname] at (8.08,7.35) {$J(df)$};
		
		\draw[interval] (7,6.85) -- (9,6.85);
		\node[intervalname] at (9.08,6.85) {$J(eg)$};
		
		\draw[interval] (8,6.35) -- (9,6.35);
		\node[intervalname] at (9.08,6.35) {$J(fg)$};
		
		\draw[interval] (8,5.85) -- (10,5.85);
		\node[intervalname] at (10.08,5.85) {$J(fh)$};
		
		\draw[interval] (9,5.35) -- (11,5.35);
		\node[intervalname] at (11.08,5.35) {$J(z)$};
		
		\node[tag,label={[tagname]above:$T_a$}] at (1,2.95) {};
		\node[tag,label={[tagname]above:$T_b$}] at (3,2.95) {};
		\node[tag,label={[tagname]above:$T_c$}] at (5,2.95) {};
		\node[tag,label={[tagname]above:$T_d$}] at (6,2.95) {};
		\node[tag,label={[tagname]above:$T_e$}] at (7,2.95) {};
		\node[tag,label={[tagname]above:$T_f$}] at (8,2.95) {};
		\node[tag,label={[tagname]above:$T_g$}] at (9,2.95) {};
		\node[tag,label={[tagname]above:$T_h$}] at (10,2.95) {};
		
		\node[font=\scriptsize,anchor=east] at (-.15,.62) {gap markers};
		\node[font=\scriptsize,anchor=east] at (-.15,6.85) {witnesses};
		\node[font=\scriptsize,anchor=east] at (-.15,2.95) {private tags};
		
	\end{tikzpicture}
	\caption{The intervals in the multiset $\mathcal J$ of the track-labeled graph in Figure~\ref{fig:running-example-graph}.}
	\label{fig:running-example-intervals}
\end{figure}
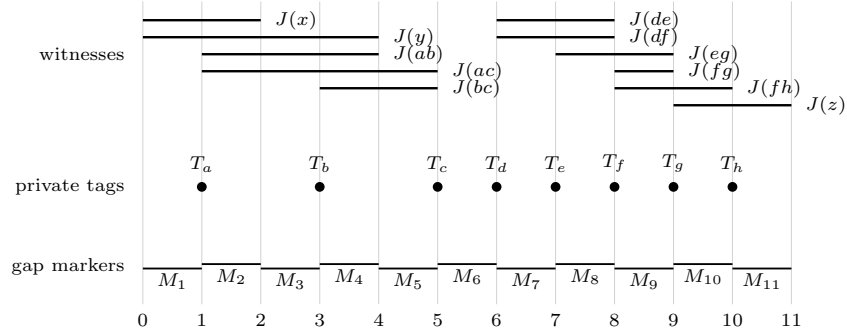
The intervals in these representing sets are depicted in Figure~\ref{fig:running-example-intervals}.
By Proposition~\ref{prop:canonical}, a canonical representation 
is obtained  by
adding the four singleton intervals $[0,0]$, $[2,2]$, $[4,4]$, and $[11,11]$ to the representation shown
in Figure~\ref{fig:running-example-intervals}. 
This is the canonical representation of an interval order with 34 elements and magnitude 12. The associated 
Hilbert extender graph $\mathcal{H}$ 
has $136$ vertices and $3,348$ edges. 
The eleven Hilbert sets in the table above correspond to eleven vertices that induce a subgraph of $\mathcal{H} $ isomorphic to the graph in Figure~\ref{fig:running-example-graph}.

\section{Eliminating bi-forks}
\label{sec:eliminate}

The goal of this section is to show that any connected track-labeled graph with bi-forks can be transformed into a bi-fork-free track-labeled graph.  
The transformation is achieved by a finite sequence of small changes to the labeling, called tectonic $\epsilon$-shifts.  
These shifts preserve the track-axioms while reducing the bi-fork structure.  
We begin by defining tectonic $\epsilon$-shifts and proving that they preserve track-labelings.

Let $H=(V,E)$ be a connected track-labeled graph with track-labeling $\lambda$. Fix a finite  label $\beta$ and let $D(\beta)=(B,\mathcal{A})$ be the $\beta$-kernel digraph, the
previously introduced subdigraph of the twin-arcs digraph.

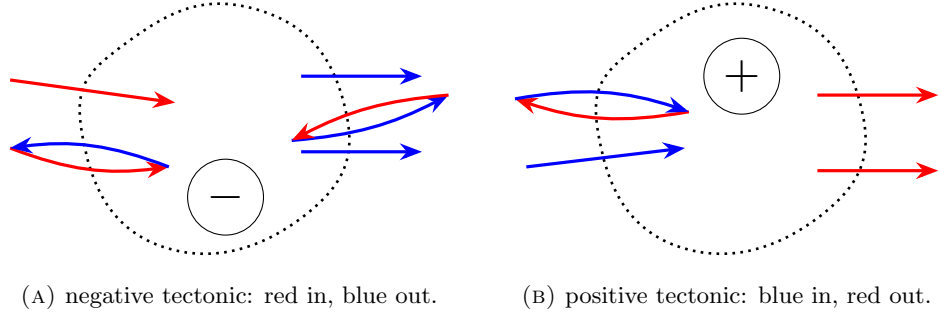
\begin{figure}
\centering
\begin{subfigure}[t]{0.48\textwidth}\centering
\begin{tikzpicture}[scale=.5,x=1cm,y=1cm,>=Stealth]
\def\W{8} 
\def\H{4} 

\begin{scope}[xshift=3.5cm,yshift=1.5cm,scale=3]
\draw[line width=.1em, dotted] plot[smooth cycle,tension=.85]
    coordinates {(-1.1,0) (-0.4,-1.0) (0.9,-0.7) (1.2,0.3) (0.4,1.1) (-0.7,0.8)};
\end{scope}
    
\coordinate (s) at (-1,0.5*\H);
\coordinate (p) at (0.55*\W,0.55*\H);

\draw[red,very thick,->] ($(s)+(-0.7,0.9)$) -- ($(p)+(-1.73,0.1)$);
\draw[red,very thick,->] ($(s)+(-.7,-0.9)$) to [bend right=.5cm] ($(p)+(-1.9,-1.6)$);
\draw[blue, very thick,<-] ($(s)+(-.7,-0.9)$) to [bend left =.5cm] ($(p)+(-1.9,-1.6)$);

\foreach \k/\y in {1/0.25,3/0.75}{
  \coordinate (r\k) at (-2+\W,\y*\H);     
  \coordinate (t\k) at (\W+1.2,\y*\H); 
  \draw[blue,very thick,->] (r\k) -- (t\k);
}\draw[blue,very thick,->] (5.75,1.3) to [bend right=.4cm]  (9.9,2.5);
\draw[red,very thick,<-] (5.75,1.3) to [bend left=.4cm] (9.9,2.5);

\node[draw,yshift=-1.1cm,circle,inner sep=2pt,scale=1.8] at (0.5*\W,0.5*\H) {$-$};
\end{tikzpicture}
\caption{negative tectonic: red in, blue out.}
\label{fig:-tectonic}
\end{subfigure}\hfill
\begin{subfigure}[t]{0.48\textwidth}\centering
\begin{tikzpicture}[scale=.5,x=1cm,y=1cm,>=Stealth]
\def\W{8} 
\def\H{4} 

\begin{scope}[xshift=3.5cm,yshift=2cm,scale=3]
\draw[line width=.1em, dotted] plot[smooth cycle,tension=.85]
   coordinates {(-1.1,0) (-0.4,-1.0) (0.9,-0.7) (1.2,0.3) (0.4,1.1) (-0.7,0.8)};
\end{scope}

\coordinate (s) at (-1,0.5*\H);
\coordinate (p) at (0.55*\W,0.55*\H);

\draw[red,very thick,<-] ($(s)+(-1,0.9)$) to [bend right=.5cm] ($(p)+(-1.8,0.36)$);
\draw[blue,very thick,->] ($(s)+(-1,0.9)$) to [bend left=.5cm]  ($(p)+(-1.8,0.36)$);
\draw[blue,very thick,->] ($(s)+(-.7,-0.9)$) -- ($(p)+(-1.9,-0.6)$);

\foreach \k/\y in {1/0.25,3/0.75}{
  \coordinate (r\k) at (-2+\W,\y*\H);     
  \coordinate (t\k) at (\W+1.2,\y*\H); 
  \draw[red,very thick,->] (r\k) -- (t\k);
}
\node[draw,circle,yshift=0.75cm,inner sep=2pt,scale=1.8] at (0.5*\W,0.5*\H) {$+$};
\end{tikzpicture}
\caption{positive tectonic: blue in, red out.}
\label{fig:+tectonic}
\end{subfigure}

\caption{Tectonic sets}
\label{fig:xi-tectonic}
\end{figure}

A nonempty proper subset
$U\subsetneq B$ is a \emph{negative tectonic set} if every arc from $U$ to
$B\setminus U$ is blue and every arc from $B\setminus U$ to $U$ is red.
Similarly, $U$ is a \emph{positive tectonic set} if every arc from $U$ to
$B\setminus U$ is red and every arc from $B\setminus U$ to $U$ is blue
(see Figure~\ref{fig:xi-tectonic}).  A \emph{$\beta$-tectonic set} is a subset of $B$ that 
is either positive tectonic or negative tectonic in $D(\beta)$.  
When the label $\beta$ is clear from context, we simply call such a set \emph{tectonic}.
  

Suppose that $U$ is a $\beta$-tectonic set.  For rational $\epsilon>0$, the
\emph{simultaneous $\epsilon$-shift of $U$} is the perturbed labeling
$\lambda_\epsilon$ defined by
$$
\lambda_\epsilon(v)=
\begin{cases}
	\lambda(v)+\epsilon, & \text{if } v\in U \text{ and } U \text{ is positive},\\
	\lambda(v)-\epsilon, & \text{if } v\in U \text{ and } U \text{ is negative},\\
	\lambda(v), & \text{if } v\notin U.
\end{cases}
$$
This shift is called the \emph{$U$-tectonic $\epsilon$-shift}, with its direction
determined by the sign of~$U$.
All labels in the tectonic set~$U$ rise together if the sign of~$U$ is positive, and fall together if it is negative.
 If $U$ is both positive and negative, 
then either sign may be assigned to $U$. 

\begin{proposition}\label{prop:tectonic-feasible}
	Suppose that $H=(V,E)$ is a connected track-labeled graph with track-labeling
	$\lambda$,  $\beta$ is a finite label, and $U$ is a $\beta$-tectonic set.  For sufficiently
	small rational $\epsilon>0$, the $U$-tectonic $\epsilon$-shift $\lambda_\epsilon$
	is a track-labeling of $H$.
\end{proposition}
\begin{proof}
Let
$$
\Delta_\beta=\min\{|\alpha-\beta|:\alpha \text{ is a finite label used by $\lambda$ and } \alpha\ne\beta\},
$$
with the convention that $\Delta_\beta=+\infty$ if $\beta$ is the only finite label.
Choose a rational $\epsilon>0$ such that
$$
0<\epsilon<\frac{\Delta_\beta}{2}.
$$
Thus $\beta\pm\epsilon$ remains closer to $\beta$ than to any other finite label; the extreme labels impose no restriction.	
We show that $\lambda_\epsilon$ satisfies the five track-axioms.

By construction, $\lambda_\epsilon$ differs from $\lambda$ only on $U$, and there only by the
common additive term $+\epsilon$ or $-\epsilon$. Hence comparisons not involving a vertex of $U$ are unchanged.
Because $B$ induces a clique in $H$ and the choice of $\epsilon$ guarantees the new label ($\beta + \epsilon$ or $\beta - \epsilon$) is distinct from prior labels, 
the shift $\lambda_\epsilon$ satisfies \clique.
The shift does not change which vertices carry extreme labels, so \independence\
and \claw\ remain valid.  Moreover, the choice of $\epsilon$ creates no equality
between vertices whose original labels were distinct.  Thus, if
$\lambda_\epsilon(u)=\lambda_\epsilon(v)$, then $\lambda(u)=\lambda(v)$.  Since
$\lambda$ satisfies \path, every induced path $(x,u,v)$ with
$\lambda_\epsilon(u)=\lambda_\epsilon(v)$ has $\lambda(x)$ finite, and hence
$\lambda_\epsilon(x)$ finite.  Therefore \path\ is also preserved.

It remains to verify \betweenness.
Let $D(\beta)=(B,\mathcal{A})$ be the $\beta$-kernel digraph of the track-labeled graph $(H,\lambda)$.
Assume, to the contrary, that \betweenness\ fails for
$\lambda_\epsilon$.  This implies there exist $a,b,c$ such that
$$
\lambda_\epsilon(a)<\lambda_\epsilon(b)<\lambda_\epsilon(c),\qquad ac\in E,
$$
but $\{a,b,c\}$ does not induce a triangle in $H$. Since \betweenness\  holds for $\lambda$, we must have
$\lambda(a)=\lambda(b)$ or $\lambda(b)=\lambda(c)$.

We use the following \emph{elementary observation}.  Suppose $\lambda(x)=\lambda(y)$ but
$\lambda_\epsilon(x)\ne\lambda_\epsilon(y)$.  If either both $x,y$ lie in $U$ or neither lies in $U$, then the two labels are changed in the same way, and hence remain equal.  
Therefore exactly one of $x,y$ lies in $U$.  
Now assume, without loss of generality, that $x\in U$ and $y\notin U$.  If $U$ is positive, then
$$
\lambda_\epsilon(x)=\lambda(x)+\epsilon
\quad\text{and}\quad
\lambda_\epsilon(y)=\lambda(y).
$$
Since $\lambda(x)=\lambda(y)$ and $\epsilon>0$, it follows that
$$
\lambda_\epsilon(x)>\lambda_\epsilon(y).
$$
Thus the vertex in $U$ has the larger $\lambda_\epsilon$-value.  Symmetrically, if $U$ is negative, then 
the vertex in $U$ has the smaller $\lambda_\epsilon$-value.

\smallskip
\noindent\textsc{Case 1.} $\lambda(a)=\lambda(b)$.
Since $\lambda(a)=\lambda(b)$ but $\lambda_\epsilon(a)<\lambda_\epsilon(b)$, the elementary observation implies that
$\lambda(a)=\lambda(b)=\beta$ and exactly one of $a,b$ lies in $U$.
By \clique, $ab\in E$.  Since $ac\in E$ by hypothesis and $\{a,b,c\}$ does not
induce a triangle, we have $bc\notin E$.  Thus $(c,a,b)$ is an induced path.

We claim that $\lambda(c)>\beta$.  Indeed, \path\ applied to the induced path
$(c,a,b)$ shows that $\lambda(c)$ is finite.  Also $\lambda(c)\ne\beta$, since
otherwise \clique\ would imply $bc\in E$.  The choice of $\epsilon$ ensures that
no finite label distinct from $\beta$ lies within $2\epsilon$ of $\beta$.  Since
$\lambda_\epsilon(b)<\lambda_\epsilon(c)$ and $\lambda_\epsilon(b)$ differs from
$\beta$ by at most $\epsilon$, it follows that $\lambda(c)>\beta$.

Hence $(c,a,b)$ is a twin triple, so $a\to b$ is a red arc of $D(\beta)$.  If
$U$ is negative, then the elementary observation and
$\lambda_\epsilon(a)<\lambda_\epsilon(b)$ imply that $a\in U$; hence the red arc
$a\to b$ leaves $U$, contradicting that $U$ is negative tectonic.  If $U$ is
positive, then the elementary observation implies that $b\in U$; hence the red
arc $a\to b$ enters $U$, contradicting that $U$ is positive tectonic.

\smallskip
\noindent\textsc{Case 2.} $\lambda(b)=\lambda(c)$.
Since $\lambda(b)=\lambda(c)$ but $\lambda_\epsilon(b)<\lambda_\epsilon(c)$, the elementary observation implies that
$\lambda(b)=\lambda(c)=\beta$ and exactly one of $b,c$ lies in $U$.
By \clique, $bc\in E$.  Since $ac\in E$ by hypothesis and $\{a,b,c\}$ does not
induce a triangle, we have $ab\notin E$.  Thus $(a,c,b)$ is an induced path.

We claim that $\lambda(a)<\beta$.  Indeed, \path\ applied to the induced path
$(a,c,b)$ shows that $\lambda(a)$ is finite.  Also $\lambda(a)\ne\beta$, since
otherwise \clique\ would imply $ab\in E$.  The choice of $\epsilon$ ensures that
no finite label distinct from $\beta$ lies within $2\epsilon$ of $\beta$.  Since
$\lambda_\epsilon(a)<\lambda_\epsilon(b)$ and $\lambda_\epsilon(b)$ differs from
$\beta$ by at most $\epsilon$, it follows that $\lambda(a)<\beta$.

Hence $(a,c,b)$ is a twin triple, so $c\to b$ is a blue arc of $D(\beta)$.  If
$U$ is negative, then the elementary observation and
$\lambda_\epsilon(b)<\lambda_\epsilon(c)$ imply that $b\in U$; hence the blue arc
$c\to b$ enters $U$, contradicting that $U$ is negative tectonic.  If $U$ is
positive, then the elementary observation implies that $c\in U$; hence the blue
arc $c\to b$ leaves $U$, contradicting that $U$ is positive tectonic.

Both cases yield a contradiction, so $\lambda_\epsilon$ satisfies \betweenness.
Because it satisfies all five track-axioms, $\lambda_\epsilon$ is a track-labeling of $H$.
\end{proof}

\begin{proposition}\label{prop:tectonic-exists}
If $D(\beta)=(B,\mathcal{A})$ is a nontrivial $\beta$-kernel of the connected track-labeled graph $(H,\lambda)$, 
then $B$ contains a positive $\beta$-tectonic set and a negative $\beta$-tectonic set.
\end{proposition}
\begin{proof}  We prove the existence of a negative $\beta$-tectonic set $U$.  
The existence of a positive tectonic set is symmetric\footnote{Simply define $G(\beta)$ by keeping all blue arcs and reversing all red arcs.}.	
	
Define an auxiliary digraph $G(\beta)=(B,F)$ as follows: for every red arc $x\to y$ of $D(\beta)$ put $x\to y$ into $F$, and for every blue arc $x\to y$ of $D(\beta)$ put the reversed arc $y\to x$ into $F$.  Thus, every arc $x \to y$ of $G(\beta)$ is \emph{realized} either by a red arc $x\to y$ of $D(\beta)$ or a blue arc $y\to x$ of $D(\beta)$.

A \emph{sink strongly connected component} of a digraph is 
a maximal subdigraph in which every vertex can reach every other vertex and from which
no arc leads to a vertex outside the component.

It is well known that every finite directed graph contains at least one sink strongly connected component.
Let $U\subseteq B$ be the vertex set of a sink strongly connected component of $G(\beta)$.  By definition, $U$ is nonempty.
We prove that $U$ is a negative $\beta$-tectonic set.

We first verify the color conditions across the cut $(U,B\setminus U)$.  Suppose
$a\in U$ and $b\in B\setminus U$.  If $a\to b$ were a red arc in $D(\beta)$,
then $a\to b$ would be an arc of $G(\beta)$ from $U$ to $B\setminus U$,
contradicting the choice of $U$ as the vertex set of a sink strongly connected component
of $G(\beta)$.  Similarly, if $b\to a$ were a blue arc in $D(\beta)$, then, by
the definition of $G(\beta)$, the reversed arc $a\to b$ would be an arc of
$G(\beta)$ from $U$ to $B\setminus U$, again a contradiction.  Therefore every
arc of $D(\beta)$ from $U$ to $B\setminus U$ is blue, and every arc of
$D(\beta)$ from $B\setminus U$ to $U$ is red.  Since $U$ is nonempty, it remains
only to prove that $U\ne B$.  This is the conclusion of the following sequence of claims.

\smallskip
\noindent\textsc{Claim 1.}
If $\{p,q,r\}$ induces a directed $3$-cycle in $G(\beta)$, then no pair from $\{p,q,r\}$ induces a bi-cycle in $D(\beta)$.

We may assume that the induced directed $3$-cycle $C$ is $p\to q \to r \to p$.
Assume, to the contrary, that some pair from $\{p,q,r\}$ induces a bi-cycle in $D(\beta)$.  
Without loss of generality, $p$ and $q$ induce a bi-cycle in
$D(\beta)$ with red arc $p\to q$ and blue arc $q\to p$.
Because $q\to r$ is in $G(\beta)$, either there is a red arc $q\to r$ or a blue arc $r\to q$ in $D(\beta)$.
Similarly, because $r\to p$ is in $G(\beta)$, either there is a red arc $r\to p$ or a blue arc $p\to r$ in $D(\beta)$.
In each of the four possible cases for these arcs, we shall obtain a bi-arrow in $D(\beta)$, contradicting 
Lemma~\ref{lem:kernel}(\ref{lem:kernel:no-biarrow}) which states that
$\beta$-kernels contain no bi-arrow.

\smallskip
\noindent
\emph{Case 1}: $q\to r$ and $r\to p$ are both red arcs in $D(\beta)$.  In this case,
$q$ is the center of the bi-fork $(q;p,r)$ in $D(\beta)$: blue $q\to p$ and red $q\to r$.  Lemma~\ref{lem:kernel}(\ref{lem:kernel:bi-fork-eye}) implies that this bi-fork $(q;p,r)$ extends to the eye-to-eye bi-cycle $(q;p,r)$.
In particular, a blue arc $r\to p$ exists in $D(\beta)$.  Combined with the red arc $r\to p$, this forms a forbidden bi-arrow in $D(\beta)$.

\smallskip
\noindent
\emph{Case 2}:
$q\to r$ is a red arc and $p\to r$ is a blue arc in $D(\beta)$.  In this case, $p$ is the center of a bi-fork $(p; r, q)$ in $D(\beta)$:
blue $p\to r$, red $p\to q$.  Lemma~\ref{lem:kernel}(\ref{lem:kernel:bi-fork-eye}) implies that this bi-fork $(p;r,q)$ extends to the eye-to-eye bi-cycle $(p;r,q)$. 
In particular, a blue arc $q\to r$ exists in $D(\beta)$.  Combined with the red arc $q\to r$, this forms a forbidden bi-arrow in $D(\beta)$.

\smallskip
\noindent
\emph{Case 3}:
$r\to q$ is a blue arc and $r\to p$ is a red arc in $D(\beta)$.  In this case,
$r$ is the center of a bi-fork $(r;q,p)$: blue $r\to q$, red $r\to p$.  
Lemma~\ref{lem:kernel}(\ref{lem:kernel:bi-fork-eye}) implies that this bi-fork $(r;q,p)$ extends to the eye-to-eye bi-cycle $(r;q,p)$. 
In particular, a blue arc $p\to q$ exists in $D(\beta)$.  Combined with the red arc $p\to q$, this forms a forbidden bi-arrow in $D(\beta)$.

\smallskip
\noindent
\emph{Case 4}:
$r\to q$ and $p\to r$ are both blue arcs in $D(\beta)$.  In this case,
$p$ is the center of a bi-fork $(p;r,q)$: blue $p\to r$, red $p\to q$. 
Lemma~\ref{lem:kernel}(\ref{lem:kernel:bi-fork-eye}) implies that this bi-fork $(p;r,q)$ extends to the eye-to-eye bi-cycle $(p;r,q)$. 
In particular, a red arc $r\to q$ exists in $D(\beta)$.  Combined with the blue arc $r\to q$, this forms a forbidden bi-arrow in $D(\beta)$.

\smallskip
\noindent\textsc{Claim 2.}
No strongly connected component of $G(\beta)$ contains a bi-cycle.

Assume, to the contrary, that some pair  of vertices spanning a bi-cycle in
$D(\beta)$ belongs to the same strongly connected component $K\subset G(\beta)$.
Consider such a pair $p,q\in K$; let $p\to q$ be red and $q\to p$ be blue. Since $p$ and $q$ lie in the
same strongly connected component of $G(\beta)$, there is a directed $qp$-path
in $G(\beta)$.  
Such a path together with the arc $p\to q$ forms a directed cycle in $K$.  Let $C$ be a directed cycle of minimum length in $K$ that contains an arc $p\to q$ such that   $(p,q)$ is a bi-cycle in $D(\beta)$.  Write $C$ as
$$
v_1\to v_2\to\cdots\to v_k\to v_1,
$$
where $v_1=q$ and $v_k=p$.  For notational convenience, set $v_0=v_k$ and
$v_{k+1}=v_1$.   Notice that the closing arc $v_k\to v_1$ belongs to $G(\beta)$ because
$p\to q$ is red in $D(\beta)$.

Observe that $q\to p$ is not an arc of $G(\beta)$.  Indeed, such an arc would
arise either from a red arc $q\to p$ in $D(\beta)$ or from a blue arc $p\to q$ in
$D(\beta)$.  In the first case, this red arc $q\to p$ together with the blue arc
$q\to p$ gives a bi-arrow; in the second case, this blue arc $p\to q$ together
with the red arc $p\to q$ gives a bi-arrow.  Both possibilities contradict
Lemma~\ref{lem:kernel}(\ref{lem:kernel:no-biarrow}).  Hence $k\ge 3$.  By
Claim~1, the cycle $C$ is not a triangle, so $k\ge 4$.

We next show that not every arc of $C$ is realized by a red arc of $D(\beta)$.
Suppose, to the contrary, that for every $j\in\{1,\ldots,k\}$, the arc
$v_j\to v_{j+1}$ of $G(\beta)$ is realized by the red arc $v_j\to v_{j+1}$ of
$D(\beta)$.  Since $p$ and $q$ span a bi-cycle, $D(\beta)$ also contains the blue
arc $q\to p$, that is, the blue arc $v_1\to v_k$.  Thus $v_1$ is the center of a
bi-fork $(v_1;v_k,v_2)$ in $D(\beta)$.  By
Lemma~\ref{lem:kernel}(\ref{lem:kernel:bi-fork-eye}), this bi-fork extends to an
eye-to-eye bi-cycle.  In particular, $v_k$ and $v_2$ span a bi-cycle, and
$D(\beta)$ contains the red arc $v_k\to v_2$.
It follows that $G(\beta)$ contains the arc $v_k\to v_2$, and therefore contains
the directed cycle
$$
v_k\to v_2\to v_3\to\cdots\to v_k.
$$
This cycle is shorter than $C$, and the pair $v_k,v_2$ spans a bi-cycle in
$D(\beta)$.  This contradicts the minimal choice of $C$.  Therefore at least one
arc of $C$ is not realized by a red arc of $D(\beta)$.

Since the closing arc $v_k\to v_1$ is realized by the red arc $p\to q$ of
$D(\beta)$, while not every arc of $C$ is realized by a red arc of $D(\beta)$,
there exists $i\in\{1,\ldots,k\}$ such that the arc $v_{i-1}\to v_i$ is 
realized by the blue arc $v_i\to v_{i-1}$ of $D(\beta)$, whereas the arc $v_i\to v_{i+1}$ is realized
by the red arc $v_i\to v_{i+1}$ of $D(\beta)$.
Thus $v_i$ is the center of the bi-fork $(v_i; v_{i-1},v_{i+1})$ in $D(\beta)$.  
By Lemma~\ref{lem:kernel}(\ref{lem:kernel:bi-fork-eye}), this extends to an eye-to-eye bi-cycle; in particular,
$D(\beta)$ contains the red arc $v_{i-1}\to v_{i+1}$.  Hence $G(\beta)$ contains
the arc $v_{i-1}\to v_{i+1}$.
Replacing the two-arc segment
$$
v_{i-1}\to v_i\to v_{i+1}
$$
of $C$ by the shortcut arc $v_{i-1}\to v_{i+1}$ produces a directed cycle in
$G(\beta)$ shorter than $C$.  Moreover, vertices $v_{i-1}$ and
$v_{i+1}$ induce a bi-cycle in $D(\beta)$, contradicting the minimal choice of $C$.
Therefore the original vertices $p$ and $q$ cannot belong to the same strongly
connected component of $G(\beta)$.

\smallskip
\noindent\textsc{Claim 3.} $U$ is a proper subset of $B$.

Because $D(\beta)$ is nontrivial, it contains a bi-fork.  By
Lemma~\ref{lem:kernel}(\ref{lem:kernel:bi-fork-eye}), this bi-fork extends to an
eye-to-eye bi-cycle.  By Claim~2, the two vertices spanning one of the bi-cycles
belong to distinct strongly connected components of $G(\beta)$.  Hence
$G(\beta)$ has more than one strongly connected component.  Since $U$ is the
vertex set of one strongly connected component of $G(\beta)$, it follows that
$U\ne B$.

Thus $U$ is nonempty and proper, and the cut-color conditions proved above show
that $U$ is a negative $\beta$-tectonic set.  The symmetric construction gives a
positive $\beta$-tectonic set, completing the proof.
\end{proof}

Now we are ready to prove the converse of Proposition~\ref{prop:axioms_satisfied}.

\begin{proposition}\label{prop:axioms_sufficient}
	If $H$ is a connected graph that admits a track-labeling, then $H$ is a Hilbert fragment.
\end{proposition}

\begin{proof}
	Let $(H,\lambda)$ be a connected track-labeled graph.  
	If the twin-arcs digraph of $(H,\lambda)$ contains a bi-fork, apply the bi-fork elimination procedure: at each step, choose a nontrivial $\beta$-kernel, select a positive or negative tectonic set $U$ as guaranteed by Proposition~\ref{prop:tectonic-exists}, and perform a sufficiently small rational $U$-tectonic $\epsilon$-shift in the corresponding direction.  This shift splits the chosen $\beta$-kernel into smaller label classes, so any remaining bi-fork supported on its vertices lies in a strictly smaller kernel.  Since each shift splits one finite label class into two nonempty parts and creates no new equality with any other finite label, the number of equal finite-labeled pairs strictly decreases.  Hence the procedure terminates.  Each shift preserves the track-axioms by Proposition~\ref{prop:tectonic-feasible}, and the terminal labeling is bi-fork-free.
	
	Thus $H$ admits a bi-fork-free track-labeling.  Since $H$ is connected,
	Theorem~\ref{thm:representation} implies that $H$ is a Hilbert fragment.
\end{proof}

More generally, we obtain the following characterization.

\begin{theorem}\label{thm:characterization}
	A graph $G$ is a Hilbert fragment if and only if every connected component of
	$G$ admits a track-labeling.
\end{theorem}
\begin{proof}
	Suppose that every connected component of $G$ admits a track-labeling.  
	Proposition~\ref{prop:axioms_sufficient} implies
	that every connected component of $G$ is a Hilbert fragment.  Consequently, Lemma~\ref{lem:disjointunion} implies that $G$ is a Hilbert fragment.
	
	Conversely, suppose that $G$ is a Hilbert fragment.  Then there exists an interval order 
	$P$ such that $G$ is isomorphic to an induced subgraph of $\mathcal{H}_P$.  Since any connected component of $G$ is also
	an induced subgraph of $\mathcal{H}_P$, it is a connected Hilbert fragment.  Proposition~\ref{prop:axioms_satisfied} implies it has
	a track-labeling.
\end{proof}

\section{Hilbert fragments are cocomparability graphs}
\label{sec:family}
In this section we place Hilbert fragments into the well-established family of cocomparability graphs by proving that the complement of every Hilbert fragment is a comparability graph. This connection links Hilbert fragments to the classical theory of partially ordered sets, interval graphs, and perfect graphs.

A graph $G$ is a \emph{comparability graph} if there exists a partially ordered set
$(V(G),\preceq)$ such that two distinct vertices $x$ and $y$ are adjacent precisely when
they are comparable in the poset, that is, when either $x\preceq y$ or $y\preceq x$.
Equivalently, $G$ is a comparability graph if and only if its edges admit a
\emph{transitive orientation}: an orientation $\vec G$ such that whenever
$x\to y$ and $y\to z$ are oriented edges of $\vec G$, the edge $xz$ belongs to $G$
and is oriented as $x\to z$.
A graph $G$ is a \emph{cocomparability graph} if its complement $\overline G$
is a comparability graph; equivalently, $G$ is the incomparability graph of a poset.

Interval orders form a classical subclass of partially ordered sets.
A fundamental theorem of Gilmore and Hoffman~\cite{GilmoreHoffman1964}
states that a graph is an interval graph if and only if it is the
incomparability graph of an interval order.
Consequently, interval graphs form a subclass of cocomparability graphs.

A graph is \emph{perfect} if every induced subgraph has chromatic number equal to
the size of its largest clique.
Cocomparability graphs constitute one of the classical perfect graph classes.
Indeed, every cocomparability graph is the incomparability graph of a poset.
In every induced subgraph, cliques correspond to antichains, whereas proper colorings
correspond to partitions of the underlying subposet into chains.
Dilworth's Theorem~\cite{dilworth} therefore implies that the chromatic number equals
the clique number for every induced subgraph, and hence every cocomparability graph
is perfect.
In particular, interval graphs are perfect.

We now prove that every Hilbert fragment is a cocomparability graph.
The key observation is that every gap-labeling of a connected Hilbert fragment
naturally induces an umbrella-free ordering.
We then apply the following characterization of cocomparability graphs due to
Kratsch and Stewart~\cite{KratschStewart1993DominationCocomparability}.

Let $H=(V,E)$ be a graph and let $\sigma=(v_1,\dots,v_n)$ be a linear ordering of $V$.
An \emph{umbrella} in $\sigma$ is a triple $(v_i,v_j,v_k)$ with $i<j<k$ such that
\[
v_iv_k\in E
\quad\text{and}\quad
v_iv_j\notin E,\;\;
v_jv_k\notin E.
\]
Equivalently, $\sigma$ is umbrella-free if and only if, for every
$i<j<k$,
\[
v_iv_k\in E
\quad\text{implies}\quad
(v_iv_j\in E \text{ or } v_jv_k\in E).
\]

\begin{lemma}[Kratsch and Stewart~\cite{KratschStewart1993DominationCocomparability}]
\label{lem:KS}
A graph $G$ is a cocomparability graph if and only if it admits an umbrella-free ordering.
\end{lemma}

Cocomparability graphs admit efficient recognition, and an umbrella-free ordering
can be computed in polynomial time (in fact, linear time).
(see~\cite{CorneilDusartHabibKoehler2016,DusartHabib2017,Koehler}).

\begin{theorem}\label{thm:HF-coco}
Every Hilbert fragment is a cocomparability graph.
\end{theorem}

\begin{proof}
	It is enough to prove the theorem for connected Hilbert fragments. Indeed, if each connected component of 
	$H$ admits an umbrella-free ordering, then listing these orderings in consecutive 
	blocks gives an umbrella-free ordering of $H$: whenever two vertices are adjacent, 
	they lie in the same component, and every vertex between them in the block ordering also lies in that component.
	
	Assume, then, that $H=(V,E)$ is connected, and let $(H,\lambda)$ be a gap-labeled connected Hilbert fragment. 
	The result is immediate if $|V|=1$, so assume $|V|\geq 2$. 
	As shown in the proof of Proposition~\ref{prop:axioms_satisfied}, 
	the labeling $\lambda$ satisfies the track-axioms. 
	In particular, we may use the \clique, \independence, \path, and \betweenness\ axioms.
	
	We now define a linear ordering $\sigma_\lambda$ of $V$. First list the vertices with 
	label $-\infty$, then the finite label classes in increasing order of their labels, and 
	finally the vertices with label $+\infty$. The order within each finite label class is 
	arbitrary. 
	
	Note that since $H$ is connected, every extreme-labeled vertex  has a finite-labeled neighbor by INDEPENDENCE. To order the two extreme label classes, define, for each vertex $x$ with $\lambda(x)=-\infty$,
	$$
	r(x)=\max\{\lambda(y):xy\in E,\ -\infty<\lambda(y)<+\infty\},
	$$
	and order these vertices by nondecreasing values of $r(x)$. Similarly, for each vertex $x$ with $\lambda(x)=+\infty$, define
	$$
	\ell(x)=\min\{\lambda(y):xy\in E,\ -\infty<\lambda(y)<+\infty\},
	$$
	and order these vertices by nondecreasing values of $\ell(x)$.
	
	We claim that $\sigma_\lambda$ is umbrella-free. First observe that if $x$ has an extreme 
	label and $x$ is adjacent to one vertex of a finite label class, then $x$ is adjacent 
	to every vertex of that finite label class. Indeed, suppose $xy\in E$, $\lambda(y)=\lambda(z)$ 
	is finite, and $xz\notin E$. By the \clique\  axiom, $yz\in E$, so $x,y,z$ induce a path 
	with the latter two vertices having the same label. The \path\  axiom then implies that $\lambda(x)$ is finite, a contradiction.
	
	It follows from this observation and the \betweenness\ axiom that the finite neighbors 
	of a vertex labeled $-\infty$ form an initial segment of the finite label classes. 
	More precisely, if $\lambda(x)=-\infty$, then $x$ is adjacent to every vertex whose 
	finite label is at most $r(x)$, and to no vertex whose finite label is greater than $r(x)$. 
	Dually, if $\lambda(x)=+\infty$, then the finite neighbors of $x$ form a final segment of 
	the finite label classes: $x$ is adjacent to every vertex whose finite label is at least 
	$\ell(x)$, and to no vertex whose finite label is smaller than $\ell(x)$. By the 
	\independence\  axiom, vertices with extreme labels have no neighbors with extreme labels.
	
	Therefore, if $u$ and $v$ both have label $-\infty$ and $u$ precedes $v$ in $\sigma_\lambda$, 
	then $N_H(u)\subseteq N_H(v)$. Similarly, if $u$ and $v$ both have label $+\infty$ 
	and $u$ precedes $v$ in $\sigma_\lambda$, then $N_H(v)\subseteq N_H(u)$.
	
	Now let $a,b,c$ be vertices with $a\prec_{\sigma_\lambda} b\prec_{\sigma_\lambda} c$, 
	and suppose that $ac\in E$. We prove that $ab\in E$ or $bc\in E$.
	
	First suppose that $\lambda(b)$ is finite. Since the ordering $\sigma_\lambda$ respects the order of the labels, we have
	$$
	\lambda(a)\leq \lambda(b)\leq \lambda(c).
	$$
	If both inequalities are strict, then the \betweenness\ axiom applied to $a,b,c$ 
	gives $ab\in E$ and $bc\in E$. Otherwise, either $\lambda(a)=\lambda(b)$ or 
	$\lambda(b)=\lambda(c)$. Since $\lambda(b)$ is finite, the \clique\  axiom gives $ab\in E$ in the first case and $bc\in E$ in the second.
	
	Next suppose that $\lambda(b)=-\infty$. Then $\lambda(a)=-\infty$ as well. 
	Since $ac\in E$ and vertices with extreme labels are independent, the vertex $c$ has finite label. 
	Because $a$ precedes $b$ in the $-\infty$ block, we have $N_H(a)\subseteq N_H(b)$. 
	Since $c\in N_H(a)$, it follows that $bc\in E$.
	
	Finally suppose that $\lambda(b)=+\infty$. Then $\lambda(c)=+\infty$. Since $ac\in E$ 
	and vertices with extreme labels are independent, the vertex $a$ has finite label. 
	Because $b$ precedes $c$ in the $+\infty$ block, we have $N_H(c)\subseteq N_H(b)$. 
	Since $a\in N_H(c)$, it follows that $ab\in E$.
	
	Thus $\sigma_\lambda$ contains no umbrella. By the Kratsch--Stewart characterization (Lemma~\ref{lem:KS}), $H$ is a cocomparability graph.
\end{proof}

A graph is \emph{chordal} if it contains no induced cycle of length at
least four. By the classical theorem of Gilmore and
Hoffman~\cite{GilmoreHoffman1964}, a graph is an interval graph if and only
if it is chordal and its complement is a comparability graph.

We proved in~\cite{LengthPolyhedron} that Hilbert fragments contain no induced
\(k\)-cycles for \(k \ge 5\). Combined with
Theorem~\ref{thm:HF-coco}, this implies that every \(C_4\)-free Hilbert
fragment is an interval graph. However, the classes of Hilbert fragments
and interval graphs are incomparable under inclusion. For example, the $\exists$-graph in Figure~\ref{fig:forbid6} 
(obtained from a claw by subdividing two of its edges) is an interval graph on
six vertices, but it is not a Hilbert fragment.

\section{Applications}
\label{sec:applications}

Order- and graph-theoretic models often acquire applications far removed from the questions that originally motivated them. A particularly relevant example is the theory of knowledge spaces introduced by Doignon and Falmagne in 1985~\cite{DoignonFalmagne1985}. Starting from a combinatorial model of possible states of knowledge, this theory developed over the following decades into the theory of learning spaces and provided the mathematical foundation for adaptive assessment and learning systems~\cite{DoignonFalmagneLearningSpaces}. Thus an abstract combinatorial structure, initially motivated by the problem of knowledge assessment, eventually became the basis of a practical educational methodology.

Interval models provide another instructive example. Beyond their classical applications in genetics, psychology, archaeology, and scheduling, interval graphs and their generalizations continue to appear in problems of scheduling and resource allocation, bioinformatics, and communication networks~\cite{ArdevolEtAl2024,BarYehudaEtAl2006,TrotterHarary1979}. In such applications, intervals typically encode the occupation of a one-dimensional resource---time, position, bandwidth, or another ordered parameter---while their intersection graph records conflicts or interactions. The usefulness of the model lies precisely in recovering exploitable structure from this underlying linear organization.

The examples in this section are offered in a similar, although deliberately exploratory, spirit. A Hilbert fragment records pairwise intersections while its track-labeling and an underlying interval order may reveal a hidden one-dimensional organization not apparent in the graph itself. We illustrate this point of view in two settings: perceptual profiles arising from sensory data and collaboration structures organized by levels of expertise. We do not propose these examples as developed application theories. Rather, they suggest ways in which the structural information recovered from a Hilbert fragment might be interpreted outside its original polyhedral setting. It would be interesting if, as has happened with other order-theoretic models, some of these interpretations eventually led to useful models of their own.

\subsection*{A hidden one-dimensional sensory axis}
\label{ex:C6bar-sensory}
	Humans experience the world through `sensory bundles', collections of perceptual features that occur simultaneously. From the overlaps among these bundles one may construct a graph whose vertices represent the bundles and whose edges represent pairwise overlaps. If this graph is a Hilbert fragment, then the data are compatible with an underlying one-dimensional experience.	
	The associated track-labeling then recovers a linear ordering of the irreducible perceptual profiles along that latent axis.

	To illustrate this idea, suppose a large panel of assessors repeatedly samples
	a sensory phenomenon, such as a flavor progression or aroma development.
	The number of assessors may be much larger than the number of robust perceptual
	profiles ultimately identified.  After aggregating the responses, suppose six
	stable perceptual profiles are extracted:
	$$
	A_1,A_2,A_3,B_1,B_2,B_3.
	$$
	These profiles, rather than the assessors, will be the vertices of the graph.
	Assume that two profiles are joined by an edge whenever they share at least one
	reported descriptor.  Suppose the resulting overlap graph has edge set
	$$
	\{A_iA_j:1\le i<j\le 3\}
	\cup
	\{B_iB_j:1\le i<j\le 3\}
	\cup
	\{A_iB_i:1\le i\le 3\}.
	$$
	Thus the three $A_i$ profiles pairwise overlap, the three $B_i$ profiles
	pairwise overlap, and each $A_i$ overlaps only with its corresponding $B_i$
	across the two groups.  This graph is isomorphic to $\overline{C_6}$; see
Figure~\ref{fig:C6bar-sensory-graph}.
	
	Notice that the observed overlap graph is not an interval graph; for instance,
	$A_1,A_2,B_2,B_1$ induce a $4$-cycle.  Thus the profile overlaps cannot be
	explained simply by representing the profiles themselves as intervals on a line.
	Nevertheless, the graph is a Hilbert fragment.  The Hilbert-fragment
	representation reveals a hidden one-dimensional perceptual axis: the profiles
	are not intervals themselves, but irreducible sensory phase signatures generated
	by an underlying interval order.

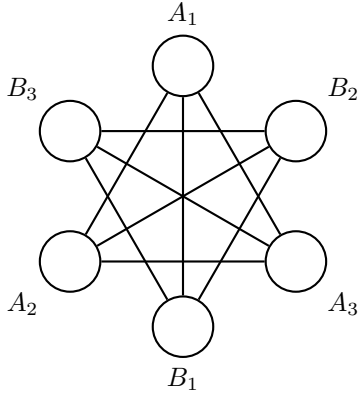
\begin{figure}[ht]
	\centering
	\begin{tikzpicture}[scale=.75]
		\tikzset{
			prof/.style={circle,draw,thick,minimum size=8mm,inner sep=1pt,font=\small}
		}
		
		\node[prof,label=above:$A_1$]      (A1) at (90:2.3)  {};
		\node[prof,label=above right:$B_2$] (B2) at (30:2.3)  {};
		\node[prof,label=below right:$A_3$] (A3) at (-30:2.3) {};
		\node[prof,label=below:$B_1$]      (B1) at (-90:2.3) {};
		\node[prof,label=below left:$A_2$] (A2) at (-150:2.3){};
		\node[prof,label=above left:$B_3$] (B3) at (150:2.3) {};
		
		\draw[thick] (A1)--(A2);
		\draw[thick] (A1)--(A3);
		\draw[thick] (A2)--(A3);
		
		\draw[thick] (B1)--(B2);
		\draw[thick] (B1)--(B3);
		\draw[thick] (B2)--(B3);
		
		\draw[thick] (A1)--(B1);
		\draw[thick] (A2)--(B2);
		\draw[thick] (A3)--(B3);
	\end{tikzpicture}
	\caption{The observed profile-overlap graph, isomorphic to $\overline{C_6}$.}
	\label{fig:C6bar-sensory-graph}
\end{figure}
	
	A track-labeling is obtained by assigning
	$$
	\lambda(A_i)=2
	\qquad\text{and}\qquad
	\lambda(B_i)=4
	\qquad\text{for }1\le i\le 3.
	$$
	This suggests a hidden linear profile ordering with the $A$-profiles occupying
	an earlier perceptual stage and the $B$-profiles occupying a later perceptual
	stage.
	
	To explain this labeling, consider the interval representation determined by the
	following intervals:
	$$
	M_2=[1,2],\qquad M_4=[3,4],
	$$
	$$
	J_i=[2,3]
	\qquad\text{for }1\le i\le 3,
	$$
	and private tags
	$$
	T_{A_i}=[2,2],\qquad T_{B_i}=[3,3]
	\qquad\text{for }1\le i\le 3.
	$$
	Here the three intervals $J_1,J_2,J_3$ are distinct represented intervals with
	the same endpoints.  The interval $M_2$ may be interpreted as a descriptor
	common to all early profiles, the interval $M_4$ as a descriptor common to all
	later profiles, and each $J_i$ as a transition descriptor linking the
	corresponding pair $A_i,B_i$.
	
	\begin{figure}
		\centering
		\begin{tikzpicture}[x=1.45cm,y=.78cm]
			\tikzset{
				axislabel/.style={font=\scriptsize},
				interval/.style={line width=.9pt},
				tag/.style={circle,fill=black,inner sep=1.3pt},
				iname/.style={font=\scriptsize,anchor=west}
			}
			
			\foreach \x in {0,...,5}{
				\draw[gray!35,line width=.25pt] (\x,0) -- (\x,5.6);
				\node[axislabel] at (\x,-.35) {$\x$};
			}
			
			\draw[interval] (1,.8) -- (2,.8);
			\node[iname] at (2.08,.8) {$M_2$};
			
			\draw[interval] (3,.8) -- (4,.8);
			\node[iname] at (4.08,.8) {$M_4$};
			
			\draw[interval] (2,2.1) -- (3,2.1);
			\node[iname] at (3.08,2.1) {$J_1$};
			
			\draw[interval] (2,2.6) -- (3,2.6);
			\node[iname] at (3.08,2.6) {$J_2$};
			
			\draw[interval] (2,3.1) -- (3,3.1);
			\node[iname] at (3.08,3.1) {$J_3$};
			
			\node[tag] at (2,4.1) {};
			\node[font=\scriptsize,anchor=east] at (1.92,4.1) {$T_{A_1}$};
			
			\node[tag] at (2,4.55) {};
			\node[font=\scriptsize,anchor=east] at (1.92,4.55) {$T_{A_2}$};
			
			\node[tag] at (2,5.0) {};
			\node[font=\scriptsize,anchor=east] at (1.92,5.0) {$T_{A_3}$};
			
			\node[tag] at (3,4.1) {};
			\node[font=\scriptsize,anchor=west] at (3.08,4.1) {$T_{B_1}$};
			
			\node[tag] at (3,4.55) {};
			\node[font=\scriptsize,anchor=west] at (3.08,4.55) {$T_{B_2}$};
			
			\node[tag] at (3,5.0) {};
			\node[font=\scriptsize,anchor=west] at (3.08,5.0) {$T_{B_3}$};

			\node[font=\scriptsize,anchor=west] at (-2.45,.8)  {gap markers};
			\node[font=\scriptsize,anchor=west] at (-2.45,2.6) {witnesses};
			\node[font=\scriptsize,anchor=west] at (-2.45,4.55) {private tags};
		\end{tikzpicture}
		\caption{A hidden one-dimensional interval representation explaining the profile-overlap graph. Only intervals needed to represent the Hilbert fragment are shown.}
		\label{fig:C6bar-sensory-intervals}
	\end{figure}
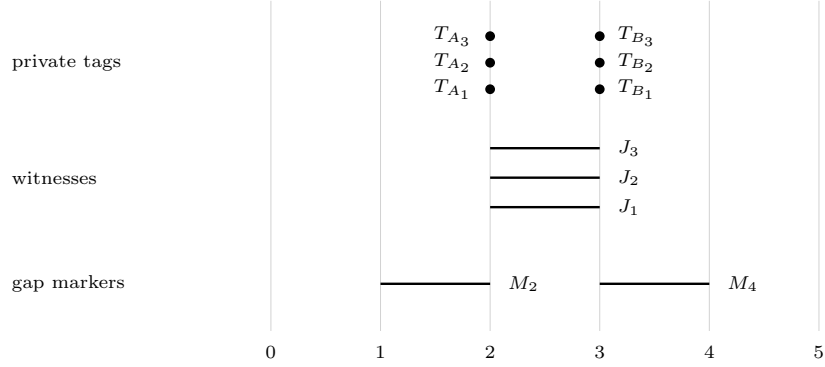
	
	The corresponding representing sets are
	$$
	\begin{array}{rclcl}
		S_{A_i}
		&=&
		G[1,2]\cup\{J_i,T_{A_i}\}
		&=&
		\{M_2,J_i,T_{A_i}\},
		\\[3pt]
		S_{B_i}
		&=&
		G[3,4]\cup\{J_i,T_{B_i}\}
		&=&
		\{M_4,J_i,T_{B_i}\},
	\end{array}
	\qquad 1\le i\le 3.
	$$
	These sets are irreducible fundamental extenders, and hence their augmented
	images are Hilbert sets.  Their intersection graph has precisely the edges
	$$
	A_iA_j,\quad B_iB_j,\quad A_iB_i
	\qquad
	(1\le i<j\le 3,\ 1\le i\le 3),
	$$
	so the observed overlap graph is the Hilbert fragment $\overline{C_6}$.
	
	In this interpretation, the Hilbert sets are irreducible perceptual profiles,
	or irreducible sensory phase signatures.  The graph is constructed directly from
	overlaps among reported perceptual profiles, while the interval axis is initially
	hidden.  The recovered track-labeling exposes a linear profile ordering: the
	profiles $A_1,A_2,A_3$ occupy one stage of the sensory experience, the profiles
	$B_1,B_2,B_3$ occupy a later stage, and each pair $A_i,B_i$ is linked by a
	shared transition descriptor.  Thus the data reveal that the perceptual
	phenomenon is compatible with an underlying one-dimensional experience.
	
	The representing sets above, and hence their augmented images as Hilbert sets,
	have a natural sensory interpretation.  For each $i$, the set
	$$
	S_{A_i}=\{M_2,J_i,T_{A_i}\}
	$$
	is an early-stage perceptual profile: $M_2$ is a descriptor common to all early
	profiles, $J_i$ is a transition descriptor shared only with the corresponding
	later profile $B_i$, and $T_{A_i}$ is a private descriptor distinguishing
	$A_i$ from the other early profiles.  Similarly,
	$$
	S_{B_i}=\{M_4,J_i,T_{B_i}\}
	$$
	is the corresponding later-stage profile: $M_4$ is common to all later profiles,
	$J_i$ links $B_i$ back to $A_i$, and $T_{B_i}$ distinguishes $B_i$ from the other
	later profiles.
	
	Thus each Hilbert set represents an irreducible sensory phase signature: a
	minimal perceptual profile generated by the hidden interval axis.  The observed
	graph records which such profiles share descriptors, while the track-labeling
	reveals the underlying one-dimensional ordering of these profiles.
	
\subsection*{A company scenario for track-labeling}
\label{sub:scenario}

Consider a company, a research institute, or a group of workers with varying levels of expertise.
Each expertise level is represented by a rational number, with larger
numbers corresponding to higher levels of expertise. \emph{Trainees} are assigned
the level \( -\infty \), whereas \emph{consultants} or \emph{emeriti} are assigned the level
\( +\infty \). All other workers will be called \emph{regular workers}. 

Workers are represented by the vertices of a connected graph, where
two vertices are adjacent if and only if the corresponding workers are
allowed to collaborate. If two workers are not allowed to collaborate,
this may be due either to insufficient ability or to company policy.

Our hypothesis is that the track-axioms provide a model for a
well-balanced collaboration structure when the vertices are labeled by
expertise level. The track-axioms impose the following conditions on the company.

\begin{description}

\item[\clique]
Workers with the same expertise level must be able to collaborate.

\item[\independence]
Trainees and emeriti do not collaborate with one another. This is typically
natural: trainees must first receive training before participating in
collaborative tasks, whereas emeriti act as individual consultants and are
not continuously present at the company.

\item[\path]
At any given expertise level, either all workers or none of them should be
able to collaborate with, and in particular train, a given trainee. Similarly,
either all or none of the workers at that level should be able to seek the
advice of a given emeritus. Thus the ability to collaborate with specific trainees or
emeriti should depend only on expertise level.

Together with the \betweenness\ axiom, this condition implies that the
expertise levels of workers who may train a trainee form a lower interval,
while those of workers who may collaborate with an emeritus form an upper
interval. For example, if workers at expertise level \(3\) may train a new
employee and there are workers at expertise level \(2\), then the latter
must also be allowed and able to train them.

\item[\betweenness]
If two workers with different expertise levels may collaborate, then every
worker whose expertise level lies strictly between theirs must be able to
collaborate with both. This prevents unnatural gaps in the ordered
compatibility structure.

\item[\claw]
No worker should act as a hub connecting three pairwise incompatible workers unless at least two of those three workers are trainees or at least two are emeriti. This condition prevents a regular worker from
being the sole link between large groups of other regular workers,
and thereby reduces the fragility of the collaboration network.
\end{description}

Here is a concrete example of this scenario.
Suppose that a research institute, or one of its departments, is being organized, and consider the role of its director.

\begin{enumerate}

\item The institute plans to employ a collection \(V\) of research positions. At this stage, the director knows what functions the positions are supposed to serve: what projects they support, what laboratories or resources they use, which groups must exchange information, etc.

\item The director constructs a collaboration graph $H$.
Each position is a vertex. Two positions are adjacent if and only if the institute wants the occupants of those positions to be able to collaborate directly.

\item If $H$ is a Hilbert fragment, the constructive proof of the characterization theorem supplies an underlying interval-order representation,
and it can be constructed using the constructive proof of the characterization theorem. 
The intervals may reveal a latent one-dimensional structure that can then be interpreted in terms of the organization.

\end{enumerate}

\begin{figure}
\begin{tikzpicture}[
    every node/.style={circle, draw, minimum size=8mm, inner sep=1pt},
    lab/.style={draw=none, circle=none, font=\small}
]

\node (a1) at (-2,  1.4) {$a_1$};
\node (a2) at (-2, -1.4) {$a_2$};
\node (b2) at ( 2,  1.4) {$b_2$};
\node (b1) at ( 2, -1.4) {$b_1$};
\node (m)  at ( 0,  0)   {$m$};
\node (t)  at (-4,  0)   {$t$};
\node (e)  at ( 4,  0)   {$e$};

\draw
(t) -- (a1)
(t) -- (a2)
(a1) -- (a2)
(a1) -- (m)
(a2) -- (m)
(m) -- (b1)
(m) -- (b2)
(b1) -- (b2)
(a1) -- (b2)
(a2) -- (b1)
(b1) -- (e)
(b2) -- (e);

\node[lab] at (-4, -0.8) {$-\infty$};
\node[lab] at (-2,  2.2) {$2$};
\node[lab] at (-2, -2.2) {$2$};
\node[lab] at ( 0, -0.8) {$4$};
\node[lab] at ( 2,  2.2) {$6$};
\node[lab] at ( 2, -2.2) {$6$};
\node[lab] at ( 4, -0.8) {$+\infty$};

\end{tikzpicture}
\caption{Organization graph for a research institute}\label{fig:institute}
\label{fig:instituto}
\end{figure}
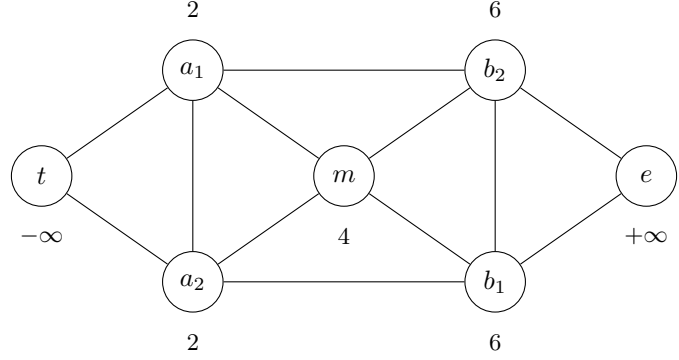

An example of the graph $H$ is shown in Figure~\ref{fig:instituto}. The director notices that the graph is a Hilbert fragment;a track-labeling is shown in the figure. With this track-labeling, there are two junior-level roles ($a_1$ and $a_2$), two senior-level roles ($b_1$ and $b_2$), an intermediate-level role ($m$), a trainee role $t$, and a consultant role $e$. The director may then recruit people for these roles. The intervals of the underlying interval order may reveal hidden information about the tasks and responsibilities.

	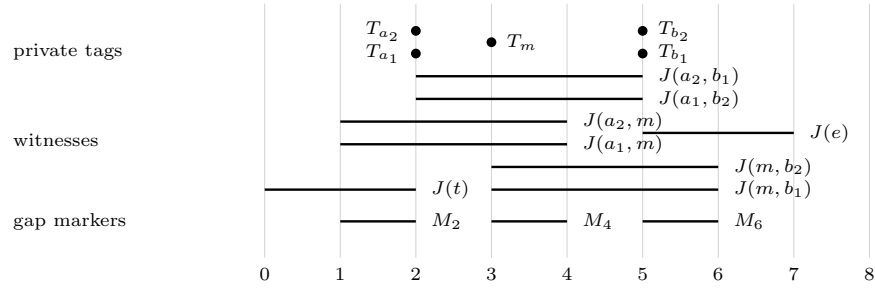
\begin{figure}
		\centering
		\begin{tikzpicture}[x=1cm,y=.6cm]
			\tikzset{
				axislabel/.style={font=\scriptsize},
				interval/.style={line width=.9pt},
				tag/.style={circle,fill=black,inner sep=1.3pt},
				iname/.style={font=\scriptsize,anchor=west}
			}
			
			\foreach \x in {0,...,8}{
				\draw[gray!35,line width=.25pt] (\x,0) -- (\x,5.6);
				\node[axislabel] at (\x,-.35) {$\x$};
			}
			
			\draw[interval] (1,.8) -- (2,.8);
			\node[iname] at (2.08,.8) {$M_2$};
			
			\draw[interval] (3,.8) -- (4,.8);
			\node[iname] at (4.08,.8) {$M_4$};
			
			\draw[interval] (5,.8) -- (6,.8);
			\node[iname] at (6.08,.8) {$M_6$};
			
			\draw[interval] (0,1.5) -- (2,1.5);
			\node[iname] at (2.08,1.5) {$J(t)$};
			
			\draw[interval] (1,2.5) -- (4,2.5);
			\node[iname] at (4.08,2.5) {$J(a_1,m)$};
			
			\draw[interval] (1,3) -- (4,3);
			\node[iname] at (4.08,3) {$J(a_2,m)$};
			
			\draw[interval] (3,1.5) -- (6,1.5);
			\node[iname] at (6.08,1.5) {$J(m,b_1)$};
			
			\draw[interval] (3,2) -- (6,2);
			\node[iname] at (6.08,2) {$J(m,b_2)$};
			
			\draw[interval] (2,3.5) -- (5,3.5);
			\node[iname] at (5.08,3.5) {$J(a_1,b_2)$};
			
			\draw[interval] (2,4) -- (5,4);
			\node[iname] at (5.08,4) {$J(a_2,b_1)$};
			
			\draw[interval] (5,2.75) -- (7,2.75);
			\node[iname] at (7.08,2.75) {$J(e)$};

			\node[tag] at (2,4.5) {};
			\node[font=\scriptsize,anchor=east] at (1.92,4.5) {$T_{a_1}$};
			
			\node[tag] at (2,5) {};
			\node[font=\scriptsize,anchor=east] at (1.92,5) {$T_{a_2}$};
			
			\node[tag] at (3,4.75) {};
			\node[font=\scriptsize,anchor=west] at (3.08,4.75) {$T_{m}$};
			
			\node[tag] at (5,4.5) {};
			\node[font=\scriptsize,anchor=west] at (5.08,4.5) {$T_{b_1}$};
			
			\node[tag] at (5,5) {};
			\node[font=\scriptsize,anchor=west] at (5.08,5) {$T_{b_2}$};

			\node[font=\scriptsize,anchor=west] at (-3.45,.8)  {gap markers};
			\node[font=\scriptsize,anchor=west] at (-3.45,2.6) {witnesses};
			\node[font=\scriptsize,anchor=west] at (-3.45,4.55) {private tags};
		\end{tikzpicture}
		\caption{Interval representation for the research institute. Only intervals needed to represent the Hilbert fragment are shown.}
		\label{fig:institute-rep}
	\end{figure}

An interval representation is shown in Figure~\ref{fig:institute-rep}. The Hilbert sets are the following.

\begin{align*}
S_t&=\varnothing\cup\{J(t)\}\\
S_{a_1}&=G[1,2]\cup \{J(a_1,b_2),T_{a_1}\}= \{M_2,J(t),J(a_1,m),J(a_2,m),J(a_1,b_2),T_{a_1}\}\\
S_{a_2}&=G[1,2]\cup \{J(a_2,b_1),T_{a_2}\}=\{M_2,J(t),J(a_1,m),J(a_2,m),J(a_2,b_1),T_{a_2}\}\\
S_m&=G[3,4]\cup \{T_m\}=\{M_4,J(a_1,m),J(a_2,m),J(m,b_1),J(m,b_2),J(a_1,b_2),J(a_2,b_1),T_m\}\\
S_{b_1}&=G[5,6]\cup \{J(a_2,b_1),T_{b_1}\}=\{M_6,J(e),J(m,b_1),J(m,b_2),J(a_2,b_1),T_{b_1}\}\\
S_{b_2}&=G[5,6]\cup\{J(a_1,b_2),T_{b_2}\}=\{M_6,J(e),J(m,b_1),J(m,b_2),J(a_1,b_2),T_{b_2}\}\\
S_e&=\varnothing\cup\{J(e)\}
\end{align*}

\section{Further research}
\label{sec:conclusion}

In this paper, we characterize Hilbert fragments by introducing track-labelings and identifying five axioms that define them. Our main theorem converts the recognition of Hilbert fragments into the problem of deciding whether a graph admits such a track-labeling. 
For a Hilbert fragment $H$, define 
$$m(H)=\min\{m(P): P\ \text{is an interval order},\ H\subseteq \mathcal{H}_P\},$$ 
and
\[
f(n)=\max\{m(H): H\text{ is a Hilbert fragment on }n\text{ vertices}\}.
\]
One can prove that, for every integer $n\ge1$,
$$
{\left\lceil\frac n2\right\rceil\le f(n)\le 2n.}
$$ 
Thus $f(n)=\Theta(n)$. In particular, the linear upper bound yields a finite search procedure for recognizing Hilbert fragments, and hence shows that the recognition problem is decidable. Decidability can also be established independently via integer linear programming.


\begin{problem}
Is there a polynomial-time algorithm to determine whether a connected graph admits a track-labeling?
\end{problem}

Since Hilbert fragments form a hereditary graph class, they have a characterization by minimal forbidden induced subgraphs.  Call a graph $H$ an \emph{obstruction} if $H$ is not a Hilbert fragment but every proper induced subgraph of $H$ is a Hilbert fragment. 
Our computations identify the small obstructions, while known infinite families show that there are infinitely many in total. 
There are 27 obstructions on at most nine vertices; all obstructions on at most six vertices are displayed in Appendix~\ref{app:obstructions}. 
The holes and odd antiholes provide infinite families \cite{LengthPolyhedron}. 

Hilbert fragments are cocomparability graphs. Consider the infinite families occurring in Gallai's forbidden induced-subgraph characterization of comparability graphs~\cite{Gallai1967}. Taking complements gives corresponding families of forbidden induced subgraphs for cocomparability graphs. Some of these do not yield infinite families of obstructions for Hilbert fragments: sufficiently large members already contain smaller Hilbert-fragment obstructions, such as the Pan graph (a $4$-cycle with an added pendant edge) in Figure~\ref{fig:forbid6}.

The natural structural problem is therefore to understand and, if possible, describe these minimal obstructions.

\begin{problem}
Determine the minimal forbidden induced subgraphs for the class of Hilbert fragments.
\end{problem}

\section*{Declaration of generative AI use}

ChatGPT (OpenAI) was used to check logical arguments, smooth the writing, and revise this manuscript.
The authors retain responsibility for verifying all arguments and references and for the final content of the paper.

\appendix

\section{Proof of the augmentation and concatenation lemmas}
\label{app:proofs}

\noindent{\bf Lemma \ref{lem:augmentation}.}
	\emph{Suppose that $\mathcal F=\{I_x=[\ell_x,r_x]:x\in X\}$ is a rational interval
	representation of an interval order $P=(X,\prec)$.  Let $P^+$ and
	$\mathcal F^+$ be the augmentation determined by $P$ and $\mathcal F$.  If
	$S\subseteq \mathcal F$ is an irreducible fundamental extender with respect to
	$\mathcal F$, then
	$$
	S^+=\{I_x^+:I_x\in S\}
	$$
	is a Hilbert set of $P^+$.}
\vskip.1em
\begin{proof}
	Throughout the proof, each interval retains the label of the element it represents.  In particular, an original interval $I_x^+$ and an
	added singleton interval $[q,q]$ are distinct represented intervals even if they
	have the same endpoints.
	
	Let
	$$
	c_1<c_2<\cdots<c_t
	$$
	be the distinct endpoints of $\mathcal F$, and let
	$$
	\alpha(s)=N(s-c_1),
	\qquad
	d_i=\alpha(c_i)
	$$
	be the positive affine map used in the construction of $\mathcal F^+$.  Thus
	$$
	0=d_1<d_2<\cdots<d_t=M.
	$$
	Write
	$$
	\mathcal F_X^+=\{I_x^+:x\in X\}\subseteq \mathcal F^+.
	$$
	For a subfamily $A\subseteq \mathcal F$, write
	$$
	A^+=\{I_x^+:I_x\in A\}\subseteq \mathcal F_X^+.
	$$
	Conversely, if $A^+\subseteq \mathcal F_X^+$, write
	$$
	A=\{I_x\in\mathcal F:I_x^+\in A^+\}
	$$
	for its preimage in $\mathcal F$.
	
	First we show that $S^+$ is a fundamental extender with respect to the
	canonical representation $\mathcal F^+$ of $P^+$.  Choose the fixed
	gap--borderline decomposition of $S$ with respect to $\mathcal F$:
	$$
	S=G[I]\cup Z.
	$$
	Suppose first that $I=[c_{h-1},c_h]$ is a bounded gap.  If
	$Z\subseteq L[I]$, then consider the unit gap
	$$
	J=[d_{h-1},d_{h-1}+1]
	$$
	of $\mathcal F^+$.  Since the original intervals in $\mathcal F_X^+$ have
	endpoints only among $d_1,\ldots,d_t$, an interval $I_x^+$ contains $J$ if and
	only if $I_x$ contains the original gap $[c_{h-1},c_h]$.  Hence
	$$
	G^+[J]=(G[I])^+.
	$$
	Also $Z^+\subseteq L^+[J]$.  Therefore
	$$
	S^+=G^+[J]\cup Z^+
	$$
	is a left fundamental extender of $P^+$.   
	
	Similarly, if $Z\subseteq R[I]$, then use the unit gap
	$$
	J=[d_h-1,d_h].
	$$
	Again $G^+[J]=(G[I])^+$, and now $Z^+\subseteq R^+[J]$.  Hence
	$$
	S^+=G^+[J]\cup Z^+
	$$
	is a right fundamental extender of $P^+$.
	
	The extreme-gap cases are immediate.  If $I=[-\infty,c_1]$, then
	$G[I]=\varnothing$ and $Z\subseteq R[I]$, so
	$$
	S^+=\varnothing\cup Z^+
	$$
	is a right fundamental extender at the left extreme gap $[-\infty,0]$ of
	$\mathcal F^+$.  If $I=[c_t,+\infty]$, then $G[I]=\varnothing$ and
	$Z\subseteq L[I]$, so
	$$
	S^+=\varnothing\cup Z^+
	$$
	is a left fundamental extender at the right extreme gap $[M,+\infty]$ of
	$\mathcal F^+$.
	
	It remains to prove that $S^+$ is irreducible with respect to $\mathcal F^+$.
	Suppose, for a contradiction, that $S^+$ is reducible.  Then there are proper
	fundamental extenders
	$$
	T_1^+,\ldots,T_a^+,
	\qquad a\ge 2,
	$$
	with respect to $\mathcal F^+$ such that
	$$
	S^+=T_1^+\dot{\cup}T_2^+\dot{\cup}\cdots\dot{\cup}T_a^+.
	$$
	Since $S^+\subseteq \mathcal F_X^+$, each $T_j^+$ is contained in
	$\mathcal F_X^+$.  For each $j$, let
	$$
	T_j=\{I_x\in\mathcal F:I_x^+\in T_j^+\}
	$$
	be the preimage of $T_j^+$ in $\mathcal F$.  We claim that each $T_j$ is a
	fundamental extender with respect to $\mathcal F$.
	
	Fix $j$, and choose a gap--borderline decomposition of $T_j^+$ in
	$\mathcal F^+$:
	$$
	T_j^+=G^+[J]\cup Y^+,
	$$
	where either $Y^+\subseteq L^+[J]$ or $Y^+\subseteq R^+[J]$.  Since
	$T_j^+\subseteq\mathcal F_X^+$, the subfamily $Y^+$ contains no added singleton
	interval.
	
	Assume first that $J=[q,q+1]$ is a bounded unit gap of $\mathcal F^+$.  There
	is a unique $h\in\{2,\ldots,t\}$ such that
	$$
	d_{h-1}\le q<q+1\le d_h.
	$$
	For intervals coming from $\mathcal F$, containment of $[q,q+1]$ is equivalent
	to containment of the original gap $[c_{h-1},c_h]$.  Thus
	$$
	G^+[q,q+1]\cap\mathcal F_X^+
	=
	(G[c_{h-1},c_h])^+.
	$$
	If $Y^+\subseteq L^+[q,q+1]$, then its preimage $Y$ satisfies
	$$
	Y\subseteq L[c_{h-1},c_h]
	$$
	when $q=d_{h-1}$, and $Y=\varnothing$ otherwise.  Hence
	$$
	T_j=G[c_{h-1},c_h]\cup Y
	$$
	is a fundamental extender with respect to $\mathcal F$.  The case
	$Y^+\subseteq R^+[q,q+1]$ is analogous: the preimage $Y$ satisfies
	$$
	Y\subseteq R[c_{h-1},c_h]
	$$
	when $q+1=d_h$, and $Y=\varnothing$ otherwise.  Again $T_j$ is a fundamental
	extender with respect to $\mathcal F$.
	
	If $J$ is the left extreme gap $[-\infty,0]$ of $\mathcal F^+$, then
	$G^+[J]=\varnothing$ and the preimage of $Y^+$ is contained in
	$R[-\infty,c_1]$.  Thus $T_j$ is a fundamental extender at the left extreme gap
	of $\mathcal F$.  If $J$ is the right extreme gap $[M,+\infty]$, the same
	argument gives a fundamental extender at the right extreme gap of $\mathcal F$.
	
	Therefore each $T_j$ is a fundamental extender with respect to $\mathcal F$.
	Since the map $I_x\mapsto I_x^+$ is a bijection from $\mathcal F$ to
	$\mathcal F_X^+$, the partition of $S^+$ pulls back to a partition
	$$
	S=T_1\dot{\cup}T_2\dot{\cup}\cdots\dot{\cup}T_a.
	$$
	Each $T_j$ is properly contained in $S$, because each $T_j^+$ is properly
	contained in $S^+$.  This contradicts the irreducibility of $S$ with respect
	to $\mathcal F$.
	
	Thus $S^+$ is an irreducible fundamental extender with respect to the canonical
	representation $\mathcal F^+$ of $P^+$; therefore, $S^+$ is a Hilbert set of $P^+$.
\end{proof}

\noindent{\bf Lemma~\ref{lem:concatenation}.}
	\emph{Let $P$ be the ordered concatenation of interval orders $P_1,\ldots,P_k$.  If
	$S$ is a Hilbert set of $P_i$, then  $S^{(i)}$ is a Hilbert set of $P$.}

\begin{proof}
	The endpoint set of $\mathcal C$ is
	$$
	\{0,1,\ldots,a_{k+1}-1\}.
	$$
	Every integer in this set occurs as both a left endpoint and a right endpoint,
	because this is true in each translated canonical block.  Hence $\mathcal C$
	is canonical by Proposition~\ref{prop:canonical}.
	
	We first record a local restriction property.  Suppose that
	$T\subseteq\mathcal C_i'$ is a fundamental extender with respect to
	$\mathcal C$.  We claim that its untranslated preimage in $\mathcal C(P_i)$ is
	a fundamental extender with respect to $\mathcal C(P_i)$.  Choose a
	gap--borderline decomposition of $T$ in $\mathcal C$.  If the witnessing gap
	lies inside the $i$th block, then the gap set and the two borderline sets,
	restricted to $\mathcal C_i'$, are exactly the translated copies of the
	corresponding gap set and borderline sets inside $\mathcal C(P_i)$.  If the
	witnessing gap is the empty gap immediately to the left of the $i$th block,
	then it restricts to the left extreme gap of $\mathcal C(P_i)$.  If the
	witnessing gap is the empty gap immediately to the right of the $i$th block,
	then it restricts to the right extreme gap of $\mathcal C(P_i)$.  Any other
	gap has gap set and borderline sets disjoint from $\mathcal C_i'$, and
	therefore cannot witness a nonempty fundamental extender contained in
	$\mathcal C_i'$. Thus the untranslated preimage of $T$ is a fundamental 
	extender with respect to $\mathcal C(P_i)$.
	
	Now let $S$ be a Hilbert set of $P_i$.  Its fixed decomposition in
	$\mathcal C(P_i)$ translates to a gap--borderline decomposition of $S^{(i)}$
	in $\mathcal C$, so $S^{(i)}$ is a fundamental extender of $P$.  Suppose, for a
	contradiction, that $S^{(i)}$ is not a Hilbert set of $P$, so it admits a nontrivial
	partition
	$$
	S^{(i)}=T_1\dot{\cup}T_2\dot{\cup}\cdots\dot{\cup}T_a,
	\qquad a\ge 2,
	$$
	where each $T_j$ is a proper fundamental extender of $P$ contained in
	$S^{(i)}$.  Since $S^{(i)}\subseteq\mathcal C_i'$, every $T_j$ is contained in
	$\mathcal C_i'$.  By the local restriction property, the untranslated preimage
	of each $T_j$ is a fundamental extender with respect to $\mathcal C(P_i)$.
	These preimages form a nontrivial partition of $S$ into proper fundamental
	extenders, contradicting the irreducibility of $S$ in $\mathcal C(P_i)$.
	Therefore $S^{(i)}$ is a Hilbert set of $P$.
\end{proof}

\section{Proof of the connected realization lemma}
\label{app:proofC}

\noindent{\bf Lemma~\ref{lem:connected-realization}.}
	\emph{Every connected Hilbert fragment admits a realization by a connected canonical
	interval representation.}

\begin{proof}
	Let $H=(V,E)$ be a connected Hilbert fragment, and choose a realization
	$(P,\varphi)$.  Consider the canonical representation $\mathcal C(P)$.  The
	empty bounded gaps of $\mathcal C(P)$ decompose this representation into
	maximal connected subrepresentations, called blocks.
	
	Since the union of the intervals of a fundamental extender is an interval, no Hilbert set of $P$ meets two distinct blocks.  	
	Since $H$ is connected, the Hilbert sets $\varphi(v)$, $v\in V$, all lie in
	one block, say $\mathcal F$.  Each $\varphi(v)$ is still a fundamental extender with respect to $\mathcal F$,
	using the same decomposition when the witnessing gap lies inside the block and
	the corresponding extreme-gap decomposition when the witnessing gap is a
	boundary gap of the block.  
	Moreover, each $\varphi(v)$ remains irreducible with respect to
	$\mathcal F$: any nontrivial partition of $\varphi(v)$ into fundamental
	extenders with respect to $\mathcal F$ would also be such a partition with
	respect to $\mathcal C(P)$, contradicting that $\varphi(v)$ is a Hilbert set of
	$P$.
	
	Now augment the connected representation $\mathcal F$.  By
	Lemma~\ref{lem:augmentation}, the images of the irreducible fundamental
	extenders $\varphi(v)$ are Hilbert sets in the augmented interval order.  The
	augmentation is connected, since every bounded gap of $\mathcal F$ is contained
	in some interval of $\mathcal F$, and subdividing such a gap into unit gaps
	preserves this property.  Intersections among the selected Hilbert sets are
	unchanged under the affine transformation and augmentation.  Hence the
	augmented representation gives a connected canonical realization of $H$.
\end{proof}

\section{Small obstructions}
\label{app:obstructions}
\begin{figure}[H]
\centering
\begingroup

\begin{subfigure}[t]{0.22\textwidth}
\centering
\tikzset{v/.style={circle,fill=black,minimum width=.1em,inner sep=.09em}}
\begin{tikzpicture}[scale=.43]
\foreach \i in {1,2,3,4,5}{
  \coordinate (P\i) at (90-72*\i:2);
  \node[v] at (P\i) {};
}
\draw (P1)--(P2)--(P3)--(P4)--(P5)--cycle;
\end{tikzpicture}
\caption*{5-hole}
\label{fig:G38}
\end{subfigure}\hfill
\begin{subfigure}[t]{0.22\textwidth}
\centering
\begin{tikzpicture}[scale=1.25]
\tikzset{v/.style={circle,fill=black,minimum width=.1em, inner sep=.09em}}
\foreach \i in {0,1,2}{
  \node[v] (a\i) at (-1+.75*\i,1) {};
}
\node[v] (bL) at (-1+.35,0) {};
\node[v] (bR) at (.15,0) {};
\foreach \i in {0,1,2}{
  \draw (a\i)--(bL) (a\i)--(bR);
}
\end{tikzpicture}
\caption*{$K_{2,3}$}
\label{fig:G44}
\end{subfigure}\hfill
\begin{subfigure}[t]{0.22\textwidth}
\centering
\begin{tikzpicture}[scale=1.1]
\tikzset{v/.style={circle,fill=black,minimum width=.1em, inner sep=.09em}}
\foreach \i in {0,1}{
  \node[v] (a\i) at (\i,0) {};
  \node[v] (b\i) at (\i,1) {};
}
\node[v] (c) at (1.8,1.5) {};
\draw (c)--(b1)--(a1)--(a0)--(b0)--(b1);
\end{tikzpicture}
\caption*{Pan}
\label{fig:pan}
\end{subfigure}\hfill

\endgroup
\end{figure}

\newcommand{\HFfigscaleSix}{0.5}
\begin{figure}[H]
\centering
\begingroup
\begin{subfigure}[t]{0.23\textwidth}
\centering

\begin{tikzpicture}[scale=0.450, baseline=(current bounding box.center)]
\tikzset{v/.style={circle,fill=black,minimum width=.1em, inner sep=.1em}}
\foreach \i in {1,...,6}{
  \coordinate (P\i) at (90-60*\i:2);
  \node[v] at (P\i) {};
}
\draw (P1)--(P2)--(P3)--(P4)--(P5)--(P6)--cycle;
\end{tikzpicture}
\caption*{6-hole}
\end{subfigure}\hfill
\begin{subfigure}[t]{0.23\textwidth}
\centering
\begin{tikzpicture}[scale=0.80, baseline=(current bounding box.center)]
\tikzset{v/.style={circle,fill=black,minimum width=.1em, inner sep=.09em}}
\node[v] (v1) at (1,-1) {};
\node[v] (v2) at (1,0) {};
\node[v] (v3) at (1,1) {};
\node[v] (v4) at (0,-1) {};
\node[v] (v5) at (0,0) {};
\node[v] (v6) at (0,1) {};
\draw (v1) -- (v4);
\draw (v3) -- (v6);
\draw (v2) -- (v5);
\draw (v1)--(v2) -- (v3);
\end{tikzpicture}
\caption*{$\exists$-graph}
\label{fig:Comb}
\end{subfigure}\hfill
\begin{subfigure}[t]{0.23\textwidth}
\centering
\begin{tikzpicture}[scale=0.90, baseline=(current bounding box.center)]
\tikzset{v/.style={circle,fill=black,minimum width=.1em, inner sep=.1em}}
\node[v] (v1) at (-1.11,-0.05) {};
\node[v] (v2) at (1.45,0.38) {};
\node[v] (v3) at (0.90,0.98) {};
\node[v] (v4) at (0.71,-1.24) {};
\node[v] (v5) at (0.35,-0.05) {};
\node[v] (v6) at (-1.30,-1.3) {};
\draw (v1) -- (v5);
\draw (v1) -- (v6);
\draw (v2) -- (v3);
\draw (v2) -- (v5);
\draw (v3) -- (v5);
\draw (v4) -- (v5);
\end{tikzpicture}
\caption*{Dog}
\end{subfigure}\hfill
\begin{subfigure}[t]{0.23\textwidth}
\centering
\begin{tikzpicture}[scale=0.70, baseline=(current bounding box.center)]
\tikzset{v/.style={circle,fill=black,minimum width=.1em, inner sep=.1em}}
\node[v] (v1) at (-0.34,-1.83) {};
\node[v] (v2) at (-0.29,0.08) {};
\node[v] (v3) at (-2.2,-0.65) {};%
\node[v] (v4) at (-1.95,0.5) {};%
\node[v] (v5) at (0.93,-1.08) {};
\node[v] (v6) at (.70,0.94) {};%
\draw (v1) -- (v2);
\draw (v1) -- (v5);
\draw (v2) -- (v3);
\draw (v2) -- (v4);
\draw (v2) -- (v5);
\draw (v2) -- (v6);
\draw (v3) -- (v4);
\end{tikzpicture}
\caption*{Goose}
\end{subfigure}\hfill
\begin{subfigure}[t]{0.23\textwidth}
\centering
\begin{tikzpicture}[scale=0.65, baseline=(current bounding box.center)]
\tikzset{v/.style={circle,fill=black,minimum width=.1em, inner sep=.1em}}
\node[v] (v1) at (-1.9,-1.35) {};
\node[v] (v6) at (0,2.25) {};
\node[v] (v3) at (1.9,-1.35) {};
\node[v] (v2) at (0,1) {};
\node[v] (v5) at (-.88,-0.5) {};
\node[v] (v4) at (.88,-0.5) {};
\draw (v6) -- (v2)--(v4);
\draw (v1) -- (v5)--(v2) (v3)--(v4)--(v5);
\draw (v5)--(v1) ;
\end{tikzpicture}
\caption*{Net}
\label{fig:Net}
\end{subfigure}\hfill
\begin{subfigure}[t]{0.43\textwidth}
\centering
\begin{tikzpicture}[scale=0.70, baseline=(current bounding box.center)]
\tikzset{v/.style={circle,fill=black,minimum width=.1em, inner sep=.1em}}
\node[v] (v1) at (-0.34,-1.83) {};
\node[v] (v2) at (-0.29,0.08) {};
\node[v] (v3) at (0.65,1.75) {};
\node[v] (v4) at (1.35,0.45) {};
\node[v] (v5) at (0.93,-1.08) {};
\node[v] (v6) at (-2.30,0.64) {};
\draw (v1) -- (v2);
\draw (v1) -- (v5);
\draw (v2) -- (v3);
\draw (v2) -- (v4);
\draw (v2) -- (v5);
\draw (v2) -- (v6);
\draw (v3) -- (v4);
\draw (v4) -- (v5);
\end{tikzpicture}
\caption*{Umbrella}
\end{subfigure}
\hfill
\begin{subfigure}[t]{0.23\textwidth}
\centering
\begin{tikzpicture}[scale=0.6, baseline=(current bounding box.center)]
\tikzset{v/.style={circle,fill=black,minimum width=.1em, inner sep=.1em}}
\node[v] (v1) at (1.76,0.39) {};
\node[v] (v2) at (-0.17,0.30) {};
\node[v] (v3) at (1.23,2.07) {};
\node[v] (v4) at (-2.40,-0.11) {};
\node[v] (v5) at (-1.18,-1.37) {};
\node[v] (v6) at (0.77,-1.27) {};
\draw (v1) -- (v2);
\draw (v1) -- (v3);
\draw (v1) -- (v6);
\draw (v2) -- (v3);
\draw (v2) -- (v4);
\draw (v2) -- (v5);
\draw (v2) -- (v6);
\draw (v4) -- (v5);
\draw (v5) -- (v6);
\end{tikzpicture}
\caption*{Blade}
\end{subfigure}\hfill
\begin{subfigure}[t]{0.23\textwidth}
\centering
\begin{tikzpicture}[scale=0.80, baseline=(current bounding box.center)]
\tikzset{v/.style={circle,fill=black,minimum width=.1em, inner sep=.1em}}
\node[v] (v1) at (1,1) {};
\node[v] (v2) at (1,-1) {};
\node[v] (v3) at (-1,-1) {};
\node[v] (v4) at (-1,1) {};
\node[v] (v5) at (0.2,.75) {};
\node[v] (v6) at (0,0) {};
\draw (v1)to [bend left=25] (v2)to [bend left=25] (v3)to [bend left=25] (v4)to [bend left=25] (v1);
\draw (v6) -- (v5);
\draw (v2) -- (v6)--(v3);
\draw (v1) -- (v6)--(v4);
\end{tikzpicture}
\caption*{Candle}
\end{subfigure}\hfill
\begin{subfigure}[t]{0.23\textwidth}
\centering
\rotatebox[origin=c]{-52}{%
\begin{tikzpicture}[scale=0.65, baseline=(current bounding box.center)]
\tikzset{v/.style={circle,fill=black,minimum width=.1em, inner sep=.1em}}
\node[v] (v1) at (0.84,-1.44) {};
\node[v] (v2) at (-0.86,-0.63) {};
\node[v] (v3) at (-0.92,-2.40) {};
\node[v] (v4) at (-1.2,2.00) {};
\node[v] (v5) at (0.34,1.63) {};
\node[v] (v6) at (2.22,0.35) {};
\draw (v1) -- (v2);
\draw (v1) -- (v3);
\draw (v1) -- (v6);
\draw (v2) -- (v3);
\draw (v2) -- (v4);
\draw (v2) -- (v5);
\draw (v4) -- (v5);
\draw (v5) -- (v6);
\end{tikzpicture}%
}
\vspace{-1.25em} 
\caption*{Hut}
\label{fig:Hut}
\end{subfigure}\hfill
\begin{subfigure}[t]{0.23\textwidth}
\centering
\begin{tikzpicture}[scale=0.65, baseline=(current bounding box.center)]
\tikzset{v/.style={circle,fill=black,minimum width=.1em, inner sep=.1em}}
\node[v] (v1) at (-1.75,-1.5) {};
\node[v] (v6) at (0,-1.5) {};
\node[v] (v3) at (1.75,-1.5) {};
\node[v] (v2) at (0,1.5) {};
\node[v] (v5) at (-.88,0) {};
\node[v] (v4) at (.88,0) {};
\draw (v5)--(v6) -- (v4);
\draw (v4) -- (v5) (v3)--(v2);
\draw (v2)--(v1) -- (v6)--(v3);
\end{tikzpicture}
\caption*{Tent}
\label{fig:Tent}
\end{subfigure}

\endgroup
\caption{Obstructions of order 5 and 6.}
\label{fig:forbid6}
\end{figure}
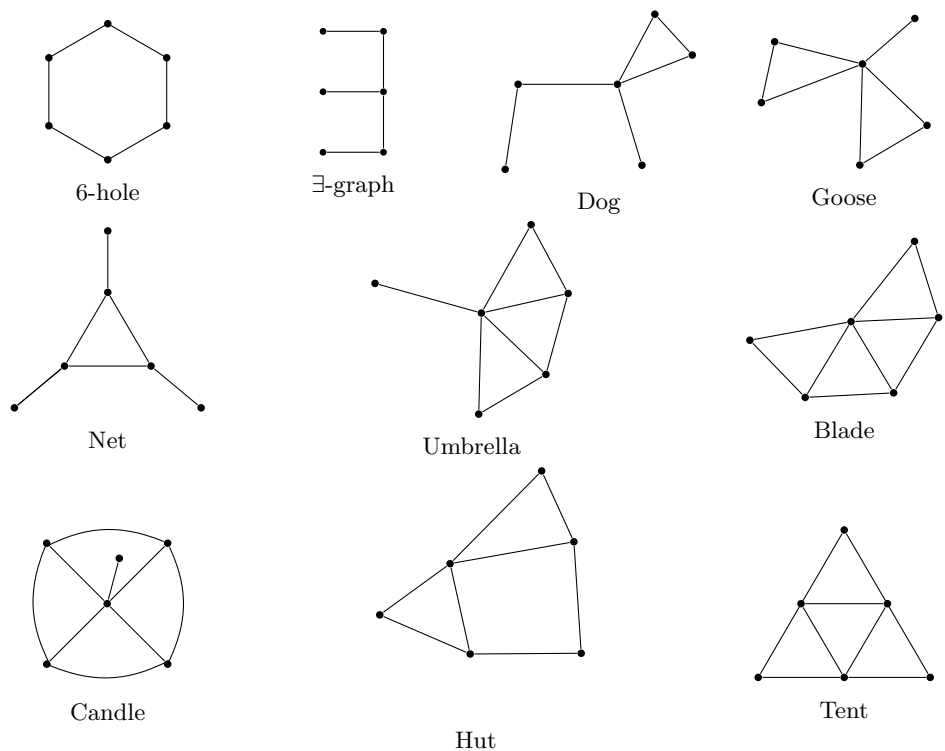



\bibliographystyle{abbrv}
\bibliography{TheReferences}

@article {LengthPolyhedron,
  AUTHOR = {Bir\'o, Csaba and K\'ezdy, Andr\'e{} E. and Lehel, Jen\H{o}},
  TITLE = {The length polyhedron of an interval order},
  JOURNAL = {Discrete Optim.},
  FJOURNAL = {Discrete Optimization},
  VOLUME = {60},
  YEAR = {2026},
  PAGES = {Paper No. 100938, 18},
  ISSN = {1572-5286,1873-636X},
  MRCLASS = {06A06 (05C20 05C62 52B05 90C27)},
  MRNUMBER = {5039507},
  DOI = {10.1016/j.disopt.2026.100938},
  URL = {https://doi-org.echo.louisville.edu/10.1016/j.disopt.2026.100938}}

@article{CorneilDusartHabibKoehler2016,
  author = {Derek G. Corneil and J{\'e}r{\'e}mie Dusart and Michel Habib and Ekkehard K{\"o}hler},
  doi = {10.1137/15M1012396},
  journal = {SIAM Journal on Discrete Mathematics},
  number = {1},
  pages = {569--591},
  title = {On the Power of Graph Searching for Cocomparability Graphs},
  volume = {30},
  year = {2016}}

@article{Dilworth,
  author = {Dilworth, R. P.},
  doi = {10.2307/1969503},
  fjournal = {Annals of Mathematics. Second Series},
  issn = {0003-486X},
  journal = {Ann. of Math. (2)},
  mrclass = {09.1X},
  mrnumber = {32578},
  mrreviewer = {P. M. Whitman},
  pages = {161--166},
  title = {A decomposition theorem for partially ordered sets},
  url = {https://doi-org.echo.louisville.edu/10.2307/1969503},
  volume = {51},
  year = {1950}}

@article{DusartHabib2017,
  author = {J{\'e}r{\'e}mie Dusart and Michel Habib},
  doi = {10.1016/j.dam.2015.07.016},
  journal = {Discrete Applied Mathematics},
  pages = {149--161},
  title = {A new {LBFS}-based algorithm for cocomparability graph recognition},
  volume = {216},
  year = {2017}}

@book{FishburnBook,
  author = {Fishburn, Peter C.},
  isbn = {0-471-81284-6},
  mrclass = {06-02 (05-02 05C99 06A10 90A06)},
  mrnumber = {776781},
  mrreviewer = {Ki Hang Kim},
  note = {A study of partially ordered sets, A Wiley-Interscience Publication},
  pages = {xi+215},
  publisher = {John Wiley \& Sons, Ltd., Chichester},
  series = {Wiley-Interscience Series in Discrete Mathematics},
  title = {Interval orders and interval graphs},
  year = {1985}}

@article{GilmoreHoffman1964,
    author  = {Paul C. Gilmore and Alan J. Hoffman},
    title   = {A characterization of comparability graphs and interval graphs},
    journal = {Canadian Journal of Mathematics},
    volume  = {16},
    year    = {1964},
    pages   = {539--548}}

@book{Greenough,
  author = {Greenough, Thomas Lockman},
  mrclass = {99-05},
  mrnumber = {2626092},
  note = {Thesis (Ph.D.)--Dartmouth College},
  pages = {56},
  publisher = {ProQuest LLC, Ann Arbor, MI},
  title = {Representation and Enumeration of Interval Orders and Semiorders},
  url = {http://gateway.proquest.com/openurl?url_ver=Z39.88-2004&rft_val_fmt=info:ofi/fmt:kev:mtx:dissertation&res_dat=xri:pqdiss&rft_dat=xri:pqdiss:7622642},
  year = {1976}}

@inproceedings{Koehler,
  address = {Cham},
  author = {Ekkehard K{\"o}hler and Lalla Mouatadid},
  booktitle = {Algorithm Theory -- {SWAT} 2014},
  doi = {10.1007/978-3-319-08404-6_28},
  editor = {R. Ravi and Inge Li G{\o}rtz},
  pages = {319--330},
  publisher = {Springer},
  series = {Lecture Notes in Computer Science},
  title = {Linear Time {LexDFS} on Cocomparability Graphs},
  volume = {8503},
  year = {2014}}

@article{KratschStewart1993DominationCocomparability,
  author = {Kratsch, Dieter and Stewart, Lorna},
  doi = {10.1137/0406032},
  journal = {SIAM Journal on Discrete Mathematics},
  number = {3},
  pages = {400--417},
  title = {Domination on Cocomparability Graphs},
  volume = {6},
  year = {1993}}

@incollection{TrotterSurvey,
        author="William T. Trotter",
        title="New perspectives on interval orders and interval graphs",
        booktitle="Surveys in Combinatorics, London Mathematical Society Lecture Note Series",
        editor="R. A. Bailey",
        volume=241,
        year=1997,
        pages="237-286"
        }

@book {SchrijverBook,
  AUTHOR = {Schrijver, Alexander},
  TITLE = {Theory of linear and integer programming},
  SERIES = {Wiley-Interscience Series in Discrete Mathematics},
  NOTE = {A Wiley-Interscience Publication},
  PUBLISHER = {John Wiley \& Sons, Ltd., Chichester},
  YEAR = {1986},
  PAGES = {xii+471},
  ISBN = {0-471-90854-1},
  MRCLASS = {90C05 (90C10)},
  MRNUMBER = {874114},
  MRREVIEWER = {J\"urgen K\"ohler}}

@book{TrotterBook,
  author = {Trotter, William T.},
  isbn = {0-8018-4425-8},
  mrclass = {06-02 (05-02 06A06 06A07)},
  mrnumber = {1169299},
  mrreviewer = {K. P. Bogart},
  note = {Dimension theory},
  pages = {xvi+307},
  publisher = {Johns Hopkins University Press, Baltimore, MD},
  series = {Johns Hopkins Series in the Mathematical Sciences},
  title = {Combinatorics and partially ordered sets},
  year = {1992}}

@book{WestBook,
     address = {Upper Saddle River, NJ},
     author = {West, Douglas B.},
     edition = {2},
     publisher = {Prentice Hall},
     title = {Introduction to Graph Theory},
     year = {2001}}

@incollection{Isaak2009,
 author = {Isaak, Garth},
 booktitle = {The mathematics of preference, choice and order},
 doi = {10.1007/978-3-540-79128-7\_17},
 isbn = {978-3-540-79127-0},
 mrclass = {06A06 (05C20)},
 mrnumber = {2648308},
 pages = {303--311},
 publisher = {Springer, Berlin},
 series = {Stud. Choice Welf.},
 title = {Interval order representations via shortest paths},
 url = {https://doi.org/10.1007/978-3-540-79128-7_17},
 year = {2009}}

@article{DoignonPauwels,
 author = {Doignon, Jean-Paul and Pauwels, Christophe},
 doi = {10.4000/msh.12061},
 fjournal = {Math\'{e}matiques et Sciences Humaines. Mathematics and Social Sciences},
 issn = {0987-6936},
 journal = {Math. Sci. Hum. Math. Soc. Sci.},
 mrclass = {06A06 (52B12 90C57 91B08)},
 mrnumber = {2917836},
 mrreviewer = {Esteban\ Indur\'{a}in},
 number = {195},
 pages = {55--71},
 title = {The representation cone of an interval order},
 url = {https://doi.org/10.4000/msh.12061},
 year = {2011}}

@article{BiKeLe,
 author = {Bir{\'o}, Csaba and K{\'e}zdy, Andr{\'e} E. and Lehel, Jen{\H{o}}},
 doi = {10.1007/s11083-025-09715-y},
 eid = {7},
 journal = {Order},
 note = {Article number 7; published online 23 December 2025},
 number = {1},
 title = {Two-count Interval Representation of a Permutation},
 url = {https://doi.org/10.1007/s11083-025-09715-y},
 volume = {43},
 year = {2026}}

@article{FMOS,
 author = {Francis, Mathew C. and Medeiros, L\'{\i}via S. and Oliveira, Fabiano S. and Szwarcfiter, Jayme L.},
 doi = {10.1016/j.dam.2021.12.029},
 fjournal = {Discrete Applied Mathematics. The Journal of Combinatorial Algorithms, Informatics and Computational Sciences},
 issn = {0166-218X,1872-6771},
 journal = {Discrete Appl. Math.},
 mrclass = {05C62 (05C35)},
 mrnumber = {4503027},
 pages = {236--251},
 title = {On subclasses of interval count two and on {F}ishburn's conjecture},
 url = {https://doi.org/10.1016/j.dam.2021.12.029},
 volume = {323},
 year = {2022}}

@article{Hajos1957,
  author  = {Haj{\'o}s, Gy{\"o}rgy},
  title   = {{\"U}ber eine Art von Graphen},
  journal = {Internationale Mathematische Nachrichten},
  volume  = {11},
  year    = {1957},
  pages   = {65}
}

@article{Benzer1959,
  author  = {Benzer, Seymour},
  title   = {On the Topology of the Genetic Fine Structure},
  journal = {Proceedings of the National Academy of Sciences of the United States of America},
  volume  = {45},
  number  = {11},
  pages   = {1607--1620},
  year    = {1959},
  doi     = {10.1073/pnas.45.11.1607}
}

@article{Luce1956,
  author  = {Luce, R. Duncan},
  title   = {Semiorders and a Theory of Utility Discrimination},
  journal = {Econometrica},
  volume  = {24},
  number  = {2},
  pages   = {178--191},
  year    = {1956},
  doi     = {10.2307/1905751}
}

@article{Wiener1914,
  author  = {Wiener, Norbert},
  title   = {A Contribution to the Theory of Relative Position},
  journal = {Proceedings of the Cambridge Philosophical Society},
  volume  = {17},
  pages   = {441--449},
  year    = {1914}
}

@article{ScottSuppes,
 author = {Scott, Dana and Suppes, Patrick},
 doi = {10.2307/2964389},
 fjournal = {The Journal of Symbolic Logic},
 issn = {0022-4812},
 journal = {J. Symbolic Logic},
 mrclass = {02.00 (04.00)},
 mrnumber = {115919},
 mrreviewer = {H. Kiesow},
 pages = {113--128},
 title = {Foundational aspects of theories of measurement},
 url = {https://doi-org.echo.louisville.edu/10.2307/2964389},
 volume = {23},
 year = {1958}}

@book{Golumbic2004,
  author    = {Golumbic, Martin Charles},
  title     = {Algorithmic Graph Theory and Perfect Graphs},
  edition   = {2},
  series    = {Annals of Discrete Mathematics},
  volume    = {57},
  publisher = {Elsevier},
  address   = {Amsterdam},
  year      = {2004},
  isbn      = {978-0-444-51530-8}
}

@article{DoignonFalmagne1985,
  author  = {Doignon, Jean-Paul and Falmagne, Jean-Claude},
  title   = {Spaces for the Assessment of Knowledge},
  journal = {International Journal of Man-Machine Studies},
  volume  = {23},
  number  = {2},
  pages   = {175--196},
  year    = {1985},
  doi     = {10.1016/S0020-7373(85)80031-6}
}

@book{DoignonFalmagneLearningSpaces,
  author    = {Falmagne, Jean-Claude and Doignon, Jean-Paul},
  title     = {Learning Spaces},
  series    = {Interdisciplinary Applied Mathematics},
  volume    = {31},
  publisher = {Springer},
  address   = {Berlin, Heidelberg},
  year      = {2011},
  doi       = {10.1007/978-3-642-01039-2}
}

@article{FishburnGraham1985,
  author  = {Fishburn, Peter C. and Graham, Ronald L.},
  title   = {Classes of Interval Graphs under Expanding Length Restrictions},
  journal = {Journal of Graph Theory},
  volume  = {9},
  number  = {4},
  pages   = {459--472},
  year    = {1985},
  doi     = {10.1002/jgt.3190090405}
}

@article{Gallai1967,
  author  = {Gallai, Tibor},
  title   = {Transitiv orientierbare Graphen},
  journal = {Acta Mathematica Academiae Scientiarum Hungaricae},
  volume  = {18},
  number  = {1--2},
  pages   = {25--66},
  year    = {1967},
  doi     = {10.1007/BF02020961}
}

@inproceedings{ArdevolEtAl2024,
  author    = {Ard{\'e}vol Mart{\'i}nez, Virginia and
               Rizzi, Romeo and
               Sikora, Florian and
               Vialette, St{\'e}phane},
  title     = {Generalizing {R}oberts' Characterization of Unit Interval Graphs},
  booktitle = {49th International Symposium on Mathematical Foundations
               of Computer Science (MFCS 2024)},
  series    = {Leibniz International Proceedings in Informatics (LIPIcs)},
  year      = {2024}
}

@article{BarYehudaEtAl2006,
  author  = {Bar-Yehuda, Reuven and Halld{\'o}rsson, Magn{\'u}s M.
             and Naor, Joseph (Seffi) and Shachnai, Hadas
             and Shapira, Irina},
  title   = {Scheduling Split Intervals},
  journal = {SIAM Journal on Computing},
  volume  = {36},
  number  = {1},
  pages   = {1--15},
  year    = {2006},
  doi     = {10.1137/S0097539703437843}
}

@article{TrotterHarary1979,
  author  = {Trotter, William T., Jr. and Harary, Frank},
  title   = {On Double and Multiple Interval Graphs},
  journal = {Journal of Graph Theory},
  volume  = {3},
  number  = {3},
  pages   = {205--211},
  year    = {1979},
  doi     = {10.1002/jgt.3190030302}
}
\end{document}